\documentclass[twoside,a4paper,reqno,11pt]{amsart} 
\usepackage[top=28mm,right=28mm,bottom=28mm,left=28mm]{geometry}
\usepackage{microtype}
\usepackage{amsfonts, amsmath, amssymb, amsbsy, mathrsfs, bm, latexsym, stmaryrd, hyperref, array, enumitem, textcomp, url}
\usepackage[table]{xcolor}

\renewcommand{\a}{\alpha}
\renewcommand{\b}{\beta}

\newcommand{\e}{\epsilon}
\renewcommand{\l}{\lambda}

\renewcommand{\O}{\Omega}

\newcommand{\la}{\langle}
\newcommand{\ra}{\rangle}
\newcommand{\leqs}{\leqslant}
\newcommand{\geqs}{\geqslant}

\newcommand{\rad}{\operatorname{Rad}}
\newcommand{\Sym}{\operatorname{S}}

\newcommand{\GL}{\operatorname{GL}}
\newcommand{\SL}{\operatorname{SL}}
\newcommand{\Sp}{\operatorname{Sp}} 
\newcommand{\SO}{\operatorname{SO}} 
\newcommand{\OO}{\operatorname{O}} 
\newcommand{\Spin}{\operatorname{Spin}} 
\newcommand{\Z}{\mathbb{Z}}
\newcommand{\NN}{\mathbb{N}}

\newcommand{\PP}{\mathbb{P}} 

\newcommand{\vs}{\vspace{3mm}}

\makeatletter
\newcommand{\imod}[1]{\allowbreak\mkern4mu({\operator@font mod}\,\,#1)}
\makeatother

\newenvironment{nalign}{
    \begin{equation}
    \begin{aligned}
}{
    \end{aligned}
    \end{equation}
    \ignorespacesafterend
}

\renewcommand{\arraystretch}{1.2}

\theoremstyle{plain}
\newtheorem{theorem}{Theorem} 
\newtheorem{conj}[theorem]{Conjecture} 
\newtheorem{corol}[theorem]{Corollary}
\newtheorem{thm}{Theorem}[section] 
\newtheorem{lem}[thm]{Lemma}
\newtheorem{prop}[thm]{Proposition} 

\newtheorem{cor}[thm]{Corollary} 
 
\newtheorem*{theorem*}{Theorem} 
\newtheorem*{conj*}{Conjecture}

\theoremstyle{definition}
\newtheorem{rem}[thm]{Remark}
\newtheorem{defn}[thm]{Definition}

\newtheorem{remk}{Remark}

\begin{document}

\title[Topological generation of simple algebraic groups]{Topological generation of simple  \\ algebraic groups}
 
\author{Timothy C. Burness}
\address{T.C. Burness, School of Mathematics, University of Bristol, Bristol BS8 1UG, UK}
\email{t.burness@bristol.ac.uk}

\author{Spencer Gerhardt}
\address{S. Gerhardt, Department of Mathematics, University of Southern California, Los Angeles, CA 90089-2532, USA}
\email{sgerhard@usc.edu} 
 
\author{Robert M. Guralnick}
\address{R.M. Guralnick, Department of Mathematics, University of Southern California, Los Angeles, CA 90089-2532, USA}
\email{guralnic@usc.edu}

\makeatletter
\@namedef{subjclassname@2020}{\textup{2020} Mathematics Subject Classification}
\makeatother


\date{\today} 

\thanks{We thank two anonymous referees for their helpful comments and suggestions. The third author was partially supported by the NSF
grant DMS-1901595 and a Simons Foundation Fellowship 609771.} 

\begin{abstract}   
Let $G$ be a simple algebraic group over an algebraically closed field and let $X$ be an irreducible subvariety of $G^r$ with $r \geqslant 2$. In this paper, we consider the general problem of determining if there exists a tuple $(x_1, \ldots, x_r) \in X$ such that $\la x_1, \ldots, x_r \ra$ is Zariski dense in $G$. We are primarily interested in the case where $X = C_1 \times \cdots \times C_r$ and each $C_i$ is a conjugacy class of $G$ comprising elements of prime order modulo the center of $G$. In this setting, our main theorem gives a complete solution to the problem when $G$ is a symplectic or orthogonal group. By combining our results with earlier work on linear and exceptional groups, this gives an almost complete solution for all simple algebraic groups.  We also present several applications. For example, we use our main theorem to show that many faithful representations of symplectic and orthogonal groups are generically free. We also establish new asymptotic results on the probabilistic generation of finite simple groups by pairs of prime order elements, completing a line of research initiated by Liebeck and Shalev over 25 years ago.
\end{abstract}

\maketitle

\setcounter{tocdepth}{1}
\tableofcontents

\section{Introduction} \label{s:intro}

Let $G$ be a simple algebraic group over an algebraically closed field $k$ of characteristic $p \geqs 0$. Let $r$ be a positive integer and let $X$ be a (locally closed) irreducible subvariety of $G^r = G \times \cdots \times G$ ($r$ factors). For $x = (x_1, \ldots, x_r) \in X$, let $G(x)$ denote the Zariski closure of $\langle x_1, \ldots, x_r \rangle$, so 
\begin{equation}\label{e:delta}
\Delta = \{ x \in X \,:\, G(x) = G\}
\end{equation}
is the set of tuples in $X$ that topologically generate $G$. Note that $G$ is locally finite if $k$ is algebraic over a finite field, in which case $\Delta$ is empty. Given this observation, we will be interested in the case where $k$ is not algebraic over a finite field. 

Let us observe that the existence of a tuple in $\Delta$ does not depend on the isogeny type of $G$. Indeed, the center of $G$ is contained in the Frattini subgroup, so a subgroup $H$ is dense in $G$ if and only if $HZ/Z$ is dense in $G/Z$, where $Z$ is any central subgroup of $G$. By a general theorem of Tits \cite{Tits}, every semisimple algebraic group over $k$ contains a Zariski-dense free subgroup on two generators, which of course implies that $G$ is topologically $2$-generated.

In this paper, we are interested in determining if $\Delta$ is nonempty for specific irreducible subvarieties $X$. If $p=0$ then a theorem of Guralnick \cite{GNATO} implies that $\Delta$ is nonempty if and only if it contains a nonempty open subvariety of $X$. In the general setting, we will work with \emph{generic} sets, which are subsets of $X$ containing the complement of a countable union of proper closed subvarieties. Note that the intersection of countably many generic subsets is generic. If $k$ is an uncountable algebraically closed field, then every generic subset of $X$ is dense (see \cite[Lemma 2.4]{BGGT}, for example), whereas a generic subset may be empty if $k$ is countable. In particular, if $k$ is uncountable then $\Delta$ is nonempty if it contains the intersection of countably many generic subsets.

In \cite[Theorem 2]{BGG} we proved that $\Delta$ is nonempty if and only if it is a dense subset of $X$. In view of Theorem \ref{t:gen} below, this is also equivalent to the property that $\Delta$ is generic.

\begin{theorem}\label{t:bgg}  
Let $k$ be an algebraically closed field that is not algebraic over a finite field. Then the following are equivalent:
\begin{itemize}\addtolength{\itemsep}{0.2\baselineskip}
\item[{\rm (i)}] $\Delta$ is nonempty.
\item[{\rm (ii)}] $\Delta(k')$ is nonempty for some extension $k'/k$. 
\item[{\rm (iii)}] $\Delta$ is a dense subset of $X$. 
\item[{\rm (iv)}] $\Delta$ is a generic subset of $X$.
\end{itemize} 
\end{theorem}

In (ii), $\Delta(k')$ is the set of elements in the variety $X(k')$ over $k'$ that topologically generate $G(k')$ (note that there is no need to assume that $k'$ is algebraically closed; if $k''$ is the algebraic closure of $k'$, then $\Delta(k') \subseteq \Delta(k'')$). In light of Theorem \ref{t:bgg}, we are free to assume that $k$ is uncountable in the proof of our main results on the topological generation of classical algebraic groups (see Theorem \ref{t:main2} below). 

The general set up applies in many different situations.  For example, if $H$ is a finitely generated group with a presentation $F/R$, where $F$ is a free group of rank $r$ and $R$ is a set of defining relations, then we can take $X$ to be an irreducible component of 
\[
\{(x_1, \ldots, x_r) \in G^r \, :\,  \varphi(x_1, \ldots, x_r) = 1 \mbox{ for all $\varphi \in R$}\}.
\]
of the representation variety  of $H$.

Another example that arises in this paper is the following. Given a locally closed irreducible subvariety $Y \subseteq G^m$ and words $w_1, \ldots, w_r$ in a free group of rank $m$, we may view each $w_i$ as a map from $G^m$ to $G$ and we can take
\[
X = \{ (w_1(y), \ldots, w_r(y)) \in G^r \,:\, y \in Y\},
\]
which is irreducible since it is the image of $Y$ under a morphism. Further examples include products of irreducible normal subsets of $G$, with $X = C_1 \times \cdots \times C_r$ an important special case, where each $C_i$ is a conjugacy class of $G$. We can also take $X$ to be an irreducible component of the subset of $C_1 \times \cdots \times C_r$ consisting of $r$-tuples satisfying some relations (for example, the product of the elements in each tuple is $1$).

The case where $X = C_1 \times \cdots \times C_r$ is a product of conjugacy classes was studied by Gerhardt \cite{Ger} for $G=\SL_n(k)$ (see Theorem \ref{t:sl}). A detailed treatment of this problem for exceptional algebraic groups was presented in \cite{B23, BGG} (see below for further details), where several more general results are established (including \cite[Theorem 2]{BGG}, as mentioned above). 
Our main goal in this paper is to extend the results in \cite{BGG, Ger} to all simple  algebraic groups.

In \cite{BGG}, the primary tool for studying the topological generation of exceptional algebraic groups by elements in specified conjugacy classes is encapsulated in \cite[Theorem 5]{BGG}, which involves computing the dimensions of fixed point spaces of elements acting on coset varieties of the form $G/H$, where $H$ is a maximal closed subgroup of $G$. While similar computations do arise in this paper, our approach is closer to the inductive method employed by Gerhardt in \cite{Ger}.  As explained below, several significant complications arise for the groups considered here. 

Let $G$ be a classical group with natural module $V$ and set 
$X = C_1 \times \cdots \times C_r$,
where each $C_i$ is a noncentral conjugacy class in $G$. By arguing inductively and applying  Gerhardt's result for $\SL_n(k)$, our aim is to identify certain generic subsets $Y\subseteq X$ such that the subgroups $G(y)$ for $y\in Y$ have restrictive properties. For example, $G(y)$ may be forced to contain a large subgroup of $G$ (typically defined in terms of the rank of $G$), or $G(y)$ may have to act irreducibly or primitively on $V$. Then by considering the maximal subgroups of $G$, our goal is to show that no proper subgroup of $G$ can simultaneously satisfy all of these conditions. If we can do this, then we deduce that  the intersection of these generic sets is contained in $\Delta$, which in turn allows us to conclude that $\Delta$ is nonempty (recall that we are free to assume $k$ is uncountable).

In this paper, we first consider topological generation in the general setting and we present a new result (Theorem \ref{t:main1}), which generalizes the observation that $\Delta$ is either empty or generic. We then turn our attention to the classical algebraic groups and we completely determine when one can generate topologically with elements from prescribed conjugacy classes, extending the earlier work in \cite{BGG, Ger} to all simple algebraic groups. We present some corollaries (also see Section \ref{s:corol}) and we then apply our results to obtain bounds on the dimensions of (not necessarily irreducible) $kG$-modules with a nontrivial generic stabilizer  (Theorem \ref{t:main3}).  In addition, we establish new asymptotic results on the random generation of finite simple groups of Lie type by a pair of elements of prime order, completing a line of research initiated by  Liebeck and Shalev in \cite{LiSh0} (see Theorem \ref{t:main4}).  
 
Let $G$ be a simple algebraic group over an algebraically closed field $k$ of characteristic $p \geqs 0$. In order to state our first result, recall that a closed subgroup $H$ of $G$ is \emph{$G$-irreducible} if it is not contained in a proper parabolic subgroup of $G$.  Also recall that the \emph{rank} of a closed subgroup $H$ of $G$, denoted ${\rm rk}\, H$, is the dimension of a maximal torus of the connected component $H^0$ (in particular, ${\rm rk}\, H = 0$ if $H$ is finite). 

Note that we allow $k$ to be algebraic over a finite field in the statement of Theorem \ref{t:main1}. In (ii), the subset $Y \subseteq X$ is generic and thus $Y(k)$ might be empty (but if $k'$ is an uncountable algebraically closed field containing $k$, then $Y(k')$ will be dense in $X(k')$). In addition, the set $Z$ is nonempty open and defined over $k$, so $Z(k)$ will be dense in $X(k)$ even when $k$ is algebraic over a finite field.  

\begin{theorem} \label{t:main1}   
Let $G$ be a simple algebraic group over an algebraically closed field $k$, let $r$ be a positive integer and let $X$ be a locally closed irreducible subvariety of $G^r$. 
Then one of the following holds:
\begin{itemize}\addtolength{\itemsep}{0.2\baselineskip}
\item[{\rm (i)}] For all $x \in X$, $G(x)$ is contained in a proper parabolic subgroup of $G$.
\item[{\rm (ii)}] There exists a unique (up to conjugacy) closed $G$-irreducible subgroup $H \leqs G$, a generic subset $Y$ and a
nonempty open subset $Z$  with $Y \subseteq Z \subseteq X$ such that 

\vspace{1mm}

\begin{itemize}\addtolength{\itemsep}{0.2\baselineskip}
\item[{\rm (a)}] ${\rm rk} \, G(x) \leqs {\rm rk}\, H$ for all $x \in X$; 
\item[{\rm (b)}] $G(y)$ is conjugate to $H$ for all $y \in Y$; and
\item[{\rm (c)}] $G(z)$ is contained in a conjugate of $H$ for all $z \in Z$.
\end{itemize}
\end{itemize}
\end{theorem}

It is worth noting that if (i) holds, then each $G(x)$ is contained in a conjugate of a fixed proper parabolic subgroup of $G$ (see Remark \ref{r:parab}).

\begin{remk}\label{r:1}
Let us highlight the special case in Theorem \ref{t:main1} when the conclusion in part (ii) holds with $H=G$, in which case $G(x)=G$ for all $x$ in a generic subset of $X$.
\begin{itemize}\addtolength{\itemsep}{0.2\baselineskip}
\item[{\rm (a)}] If $k$ is not algebraic over a finite field, then $\Delta$ is dense in $X$ (and hence nonempty) by Theorem \ref{t:bgg}.   
\item[{\rm (b)}] Now assume $k$ is the algebraic closure of a finite field, so $p>0$ and each $G(x)$ is finite, whence $\Delta(k)$ is empty. 
Let us assume $G$ is simply connected and let $k'$ be any algebraically closed field properly containing $k$. Note that $\Delta(k')$ is dense in $X(k')$ by Theorem \ref{t:bgg}. Fix a finite collection $S$ of rational irreducible $G$-modules, each of which is defined over $k$, and define
\[
W = \{ x \in X  \,:\, \mbox{$G(x)$ acts irreducibly on each module in $S$} \}.
\]
Note that $W$ is open in $X$ and is defined over $k$. Clearly, $W(k')$ contains $\Delta(k')$, so $W(k')$ is a dense open subset of $X(k')$ and we deduce that 
$W(k)$ is a dense open subset of $X(k)$. By choosing the modules in $S$ appropriately, and by arguing as in \cite{BGG} (or \cite{GT}), one can show that if $x \in W(k)$, then $G(x)$ contains a conjugate of $G(q)$ for some sufficiently large $p$-power $q$, where the finite group $G(q)$ is possibly twisted. We can exploit this observation to study the asymptotic generation properties of the finite groups of Lie type. For example, see Theorem \ref{t:main4} below.    
\end{itemize}
\end{remk}

Let us now specialize to the case where $G$ is a simple classical algebraic group with natural module $V$ and $k$ is not algebraic over a finite field. Set
\begin{equation}\label{e:conj}
X= C_1 \times \cdots \times C_r = x_1^G \times \cdots \times x_r^G
\end{equation}
with each $C_i = x_i^G$ a noncentral conjugacy class. Write $n = \dim V$ and let $d_i$ be the dimension of the largest eigenspace of $x_i$ on $V$. In this setting, there are two natural  obstructions to the existence of an element $x \in X$ with $G(x) = G$:
\begin{itemize}\addtolength{\itemsep}{0.2\baselineskip}
\item[{\rm (a)}] If $\sum_{i} d_i > n(r-1)$,  then $G(x)$ fixes a $1$-space in $V$ for all $x \in X$ and thus $\Delta$ is empty. 
\item[{\rm (b)}] We say that $x_i$ is \emph{quadratic} if it has a quadratic minimal polynomial on $V$ (and \emph{non-quadratic} otherwise). If $r=2$ and $x_1,x_2$ are quadratic, then every composition factor of $G(x)$ on $V$ is at most $2$-dimensional (see Lemma \ref{l:quadratic}) and thus $\Delta$ is empty if $n \geqs 3$.
\end{itemize}

By the following theorem of Gerhardt \cite[Theorem 1.1]{Ger}, these are the only obstructions  for linear groups $G = {\rm SL}_n(k)$ with $n \geqs 3$.   

\begin{theorem}[Gerhardt]\label{t:sl}  
Let $G=\SL_n(k)$, where $n \geqs 3$ and $k$ is an algebraically closed field that is not algebraic over a finite field. Define $X = C_1 \times \cdots \times C_r$ as in \eqref{e:conj}, where each $x_i$ is noncentral. Then $\Delta$ is empty if and only if 
\begin{itemize}\addtolength{\itemsep}{0.2\baselineskip}
\item[{\rm (i)}] $\sum_i d_i > n(r-1)$; or
\item[{\rm (ii)}] $r=2$ and $x_1,x_2$ are quadratic.
\end{itemize}
\end{theorem}  

This implies the same result for $G = \GL_n(k)$ if one replaces $\Delta$ by the set of $x \in X$ such that $G(x)$ contains $\SL_n(k)$. There is a similar result for $G={\rm SL}_{2}(k)$ which states that $\Delta$ is empty if and only if $r=2$ and $x_1,x_2$ are involutions modulo the center of $G$ (see \cite[Theorem 4.5]{Ger}).     

We refer the reader to \cite{BGG} for detailed results on the analogous problem for exceptional algebraic groups $G$. For example, \cite[Theorem 7]{BGG} states that if $X$ is defined as in \eqref{e:conj} then $\Delta$ is nonempty (and therefore dense) whenever $r \geqs 5$ (or $r \geqs 4$ if $G = G_2$). As explained in \cite{BGG}, it is easy to construct examples that  demonstrate the sharpness of both bounds.

In order to complete our study of topological generation for simple algebraic groups, it remains to extend the analysis to the orthogonal and symplectic groups, which is the main goal of this paper. Recall that the center of $G$ is contained in the Frattini subgroup of $G$, so the isogeny type of $G$ is not relevant. For convenience, we will work with the matrix groups ${\rm SO}_{n}(k)$ and ${\rm Sp}_{n}(k)$, where ${\rm SO}_{n}(k)$ is an index-two subgroup of the isometry group ${\rm O}_{n}(k)$ and $k$ is an algebraically closed field of characteristic $p \geqs 0$ that is not algebraic over a finite field.

Our main result is Theorem \ref{t:main2} below. Here $G = {\rm SO}_{n}(k)$ or ${\rm Sp}_{n}(k)$ and we will assume $n \geqs N$, where
\begin{equation}\label{e:N}
N = \left\{\begin{array}{ll}
10 & \mbox{if $G = \SO_n(k)$, $n$ even} \\
3 & \mbox{if $G = \SO_n(k)$, $n$ odd} \\
4 & \mbox{if $G = \Sp_n(k)$.} 
\end{array}\right.
\end{equation}
In order to justify this assumption, first recall that $\SO_4(k)$ is not simple and $\Sp_2(k) = \SL_2(k)$. In addition, the groups ${\rm SO}_6(k)$ and ${\rm SL}_{4}(k)$ are isogenous, so the result for $\SO_6(k)$ can be read off from Theorem \ref{t:sl} (see Theorem \ref{t:so6} for a version of Theorem \ref{t:main2} for $G = {\rm SO}_6(k)$ in terms of the $6$-dimensional natural module). The case $G = {\rm SO}_8(k)$ requires special attention because there are three restricted irreducible $8$-dimensional $kG$-modules and one needs to consider the eigenspaces of each $x_i$ on all three modules (see Theorems \ref{t:so88} and \ref{t:so8}). In addition, since there are isogenies between the classical groups of type $B_m$ and $C_m$ in characteristic $2$, we may assume $p \ne 2$ when $G = \SO_n(k)$ and $n$ is odd. 

In the statement of Theorem \ref{t:main2} we assume each $x_i$ in \eqref{e:conj} has prime order modulo the center $Z(G)$ of $G$ (if $p=0$ we allow $x_i$ to be an arbitrary nontrivial unipotent element). Our methods can be extended to handle more general conjugacy classes (as in Theorem \ref{t:sl} for $\SL_n(k)$), but the analysis turns out to be considerably more complicated and many exceptions arise. Furthermore, the case where the elements in $C_i$ have prime order (modulo the center) is sufficient for our applications.  However, it is worth noting that with only minor modifications to the proof, one could replace the prime order assumption by a more general hypothesis where we assume each $x_i$ is either unipotent or semisimple and the following two conditions are satisfied:
\begin{itemize}\addtolength{\itemsep}{0.2\baselineskip}
\item[{\rm (a)}] If $x_i$ is semisimple and $p \ne 2$, then either $x_i$ is an involution, or $-1$ is not an eigenvalue of $x_i$ on the natural module. 
\item[{\rm (b)}] If $x_i$ is unipotent and $p=2$, then $x_i$ is an involution.
\end{itemize}

\begin{remk}\label{r:2}
Notice that if $p=0$ then nontrivial unipotent elements have infinite order. In order to avoid the need to repeatedly highlight this special situation, we will simply view all nontrivial unipotent elements in characteristic $0$ as having prime order. Alternatively, we could assume $p>0$ throughout and then deduce the corresponding results in characteristic $0$ by a standard compactness argument, but we prefer to adopt the former approach.
\end{remk}

\begin{remk}\label{r:orth}
Suppose $G = \Sp_n(k)$ with $n \geqs 4$ and $p=2$. Let $e_i = \dim V^{x_i}$ be the dimension of the $1$-eigenspace of $x_i$ on $V$. As noted above, if $\sum_ie_i > n(r-1)$ then each $G(x)$ acts reducibly on $V$ and thus $\Delta$ is empty. In fact, it turns out that $\Delta$ is also empty if $\sum_ie_i = n(r-1)$ (see Lemma \ref{l:oinsp}), which explains the additional condition in Theorem \ref{t:main2} in this special  case.
\end{remk}

We are now in a position to state our main result. 
 
\begin{theorem}\label{t:main2} 
Let $G=\SO_{n}(k)$ or $\Sp_n(k)$, where $n \geqs N$ and $k$ is an algebraically closed field of characteristic $p \geqs 0$ that is not algebraic over a finite field. Define 
$X = C_1 \times \cdots \times C_r$ as in \eqref{e:conj}, where each $x_i$ has prime order modulo $Z(G)$. Assume the following conditions are satisfied:
\[
\sum_i d_i \leqs n(r-1), \mbox{ and also } \sum_i e_i < n(r-1) \mbox { if $G = \Sp_n(k)$ and $p=2$.}
\]
Then $\Delta$ is empty if and only if one of the following holds:
\begin{itemize}\addtolength{\itemsep}{0.2\baselineskip}
\item[{\rm (i)}] $r=2$ and $x_1, x_2$ are quadratic on the natural $kG$-module.
\item[{\rm (ii)}] $r \in \{2,3,4\}$ and the $x_i$ are recorded in Table \ref{tab:main} or \ref{tab:main2} (up to ordering of the $x_i$).
\end{itemize}
\end{theorem} 

\renewcommand{\arraystretch}{1.1}
\begin{table}
\begin{center}
\[
\begin{array}{lllll}\hline
G & \mbox{Conditions} & \multicolumn{1}{c}{x_1} & \multicolumn{1}{c}{x_2} \\ \hline
{\rm SO}_{2m}(k) & \mbox{$m \geqs 5$ odd} & \begin{array}{l}
(I_2, \l I_{m-1}, \l^{-1} I_{m-1}) \\
(J_3^2,J_2^{m-3}), \, p \ne 2
\end{array} & (J_2^{m-1},J_1^2)^*   \\ 
& & &  \\
{\rm SO}_{2m}(k) & \mbox{$m \geqs 6$ even} & \begin{array}{l}
(I_2, \l I_{m-1}, \l^{-1} I_{m-1}) \\
(J_3^2,J_2^{m-4}, J_1^2), \, p \ne 2 \\
(J_3,J_2^{m-2},J_1), \, p \ne 2
\end{array} & (J_2^{m})^*  \\ 
& & &  \\
{\rm SO}_{2m+1}(k) &  \mbox{$m \geqs 2$ even} & \hspace{1.8mm} (I_1,  \l I_m, \l^{-1}I_m)  & (J_2^m,J_1)   \\ 
& & & \\
{\rm Sp}_{4}(k) & p \ne 2 & \hspace{1.8mm} (-I_2,I_2) & \mbox{non-regular}  \\ \hline
\end{array}
\]
\caption{Some special cases with $r=2$ in Theorem \ref{t:main2}}
\label{tab:main}
\end{center}
\end{table}
\renewcommand{\arraystretch}{1}

\renewcommand{\arraystretch}{1.1}
\begin{table}
\begin{center}
\[
\begin{array}{lcccccc}\hline
G & p & r & \multicolumn{1}{c}{x_1} & \multicolumn{1}{c}{x_2} & \multicolumn{1}{c}{x_3} & \multicolumn{1}{c}{x_4} \\ \hline
{\rm SO}_5(k)  & \ne 2 & 3 & (J_2^2,J_1) & (J_2^2,J_1) &  (J_2^2,J_1) & \\
{\rm Sp}_8(k)  & \ne 2 & 3 & (-I_2,I_6) &  (-I_2,I_6) & (-I_4,I_4) & \\
{\rm Sp}_6(k)  & \ne 2 & 3 & (-I_2,I_4) & (-I_2,I_4) & (-I_2,I_4) & \\
{\rm Sp}_4(k)  & \ne 2 & 3 & (-I_2,I_2) & (-I_2,I_2) & \mbox{quadratic} & \\
& & 4 & (-I_2,I_2) &(-I_2,I_2) &(-I_2,I_2) &(-I_2,I_2) \\
& 2 & 3 & (J_2^2)^* & (J_2^2)^* & \mbox{quadratic} \\
& & 4 & (J_2^2)^* & (J_2^2)^* & (J_2^2)^* & (J_2^2)^* \\ \hline
\end{array}
\]
\caption{The special cases with $r \in \{3,4\}$ in Theorem \ref{t:main2}}
\label{tab:main2}
\end{center}
\end{table}
\renewcommand{\arraystretch}{1}

\begin{remk}\label{r:main2}
Let us record some comments on the statement of Theorem \ref{t:main2}.
\begin{itemize}\addtolength{\itemsep}{0.2\baselineskip}
\item[{\rm (a)}] Recall that $\Delta$ is empty if $\sum_i d_i > n(r-1)$, or if $r=2$ and $x_1,x_2$ are quadratic, so the theorem shows that these are essentially the only obstructions to the existence of a tuple $x \in X$ with $G(x)=G$, apart from a handful of special cases with $r \leqs 4$ (see Remark \ref{r:orth} for the additional condition $\sum_ie_i < n(r-1)$ when $G = \Sp_n(k)$ and $p=2$). As previously noted, Theorem \ref{t:bgg} states that $\Delta$ is nonempty if and only if it is generic and dense in $X$.
\item[{\rm (b)}] The elements $x_i$ appearing in Tables \ref{tab:main} and \ref{tab:main2} are presented up to conjugacy in $G$ and scalars in $Z(G)$. For unipotent elements, we give the Jordan form of $x_i$ on the natural module $V$ for $G$, where $J_m$ denotes a unipotent Jordan block of size $m$. Similarly, we describe semisimple elements $x_i$ in terms of their eigenvalues on $V$, where $I_m$ is the identity matrix of size $m$ and $\l$ is any nonzero scalar in $k$ with $\l^2 \ne 1$.
\item[{\rm (c)}] In the first two rows of Table \ref{tab:main}, the asterisk in the final column indicates that if $p=2$ then we take $x_2$ to be an $a$-type involution with the given Jordan form. Here we are using the standard Aschbacher-Seitz notation from \cite{AS} for unipotent involutions in classical groups (see Remark \ref{r:as}). In this notation, the elements appearing in the final two rows of Table \ref{tab:main2} are involutions of type $a_2$ (i.e. short root elements).
\item[{\rm (d)}] In the final row of Table \ref{tab:main}, we can choose $x_2 \in \Sp_4(k)$ to be any non-regular element of prime order modulo $Z(G)$ (note that this is equivalent to the condition $d_2 \geqs 2$). Similarly, for the two cases in Table \ref{tab:main2} with $G = {\rm Sp}_{4}(k)$ and $r=3$ we can take $x_3$ to be any quadratic element, which means that $x_3$ is either semisimple of the form $(-I_2,I_2)$ or $(\l I_2,\l^{-1}I_2)$, or unipotent with Jordan form $(J_2^2)$ or $(J_2,J_1^2)$.
\item[{\rm (e)}] As noted above, the corresponding result for $\SO_6(k)$ can be read off from the result for the isogenous group $\SL_4(k)$ (see Theorems \ref{t:sl} and \ref{t:so6}). The case $G = \SO_8(k)$ requires special attention and we refer the reader to Theorems \ref{t:so88} and \ref{t:so8}.
\end{itemize}
\end{remk}

It is worth noting that several new difficulties arise in the analysis of orthogonal and symplectic groups, in comparison to the linear groups handled in \cite{Ger}. For instance, we have to consider subspace stabilizers of both totally singular and nondegenerate spaces. Similarly, we need to distinguish the eigenspaces of a semisimple element, noting that a 
$\l$-eigenspace is nondegenerate if $\l = \pm 1$, otherwise it is totally singular. The unipotent conjugacy classes are also more complicated, especially in characteristic $2$ when the class of a unipotent element is not always uniquely determined by its Jordan form on the natural module $V$. Several key features of the proof are also more difficult in this setting. For example, there are considerably more special cases to consider and there are additional complications in applying the main induction argument. Indeed, the main idea in the proof of Theorem \ref{t:sl} for $\SL_n(k)$ involves passing to the stabilizer of a $1$-dimensional subspace of $V$ (or a hyperplane). But in an orthogonal or  symplectic group,  the largest irreducible composition factor of the stabilizer of a totally singular $1$-space has codimension $2$ in $V$, rather than codimension $1$.  

\begin{remk}\label{r:lie}
Similar results (although not quite as precise) are obtained in \cite{GG1} on the generation of Lie algebras. For example, \cite[Proposition 6.4]{GG1} gives essentially the same conditions for the Lie algebra $\mathfrak{sl}_n(k)$ of type A as obtained in Theorem \ref{t:sl} when $p \ne 2$ and 
the generating elements are all contained in the same $\SL_n(k)$-orbit (with the additional condition that the elements are either nilpotent or semisimple).  There are some advantages in working with Lie algebras (for instance, one can multiply by scalars and take closures more easily), but serious issues arise due to the existence of special isogenies when $p=2$ or $3$. Indeed, for Lie algebras one cannot ignore the issue of isogenies.   
\end{remk}

\begin{remk}\label{r:gs}
Let us also highlight related results of Guralnick and Saxl \cite[Theorems 8.1, 8.2]{GS}, which give an upper bound on the number of conjugates of a given noncentral element in a simple algebraic group required to generate a Zariski dense subgroup (in addition, they establish similar results for the corresponding finite groups of Lie type). For example, if $G$ is a classical group with an $n$-dimensional natural module, then $n$ conjugates of a given noncentral element $x \in G$ will topologically generate, unless $(G,x)$ is one of a handful of known cases (for instance, $n+1$ conjugates are needed if $G = \Sp_n(k)$, $p=2$ and $x$ is a transvection). The bounds in \cite{GS} are best possible (for classical groups), but they are not sensitive to the choice of element $x$ and they do not extend to the more general situation we consider here, where $C_1, \ldots, C_r$ are arbitrary conjugacy classes of noncentral elements (containing elements of prime order modulo $Z(G)$).
\end{remk}

In Section \ref{s:corol} we will prove the following corollary. In the statement, we define
\begin{equation}\label{e:M}
M = \left\{\begin{array}{ll}
10 & \mbox{if $G = \SO_n(k)$, $n$ even} \\
7 & \mbox{if $G = \SO_n(k)$, $n$ odd} \\
6 & \mbox{if $G = \Sp_n(k)$} \\
3 & \mbox{if $G = \SL_n(k)$} 
\end{array}\right.
\end{equation}

\begin{corol}\label{c:cor2} 
Let $G=\SL_n(k)$, $\SO_{n}(k)$ or $\Sp_n(k)$, where $n \geqs M$ and $k$ is an algebraically closed field of characteristic $p \geqs 0$ that is not algebraic over a finite field. Define $X$ as in \eqref{e:conj}, where each $x_i$ has prime order modulo $Z(G)$, and assume there exists $y \in X$ such that $G(y)$ acts irreducibly
 on the natural $kG$-module. Then $\Delta$ is empty if and only if $G=\Sp_{n}(k)$, $p=2$ and $G(x)^0 \leqs \SO_{n}(k)$ for generic $x \in X$. 
\end{corol} 

In addition, by combining the above results with Theorem \ref{t:so88} on $G = \SO_8(k)$, one can  obtain the following corollary (note that noncentral is the only condition on the $x_i$).

\begin{corol}\label{c:main}
Let $G$ be a simple algebraic group over an algebraically closed field of characteristic $p \geqs 0$ that is not algebraic over a finite field. Define $X$ as in \eqref{e:conj}, where $r \geqs 5$ and each $x_i$ is noncentral. If $G$ is a classical group, then define $d_i$ and $e_i$ as above. Then $\Delta$ is empty if and only if $G$ is classical and either 
\begin{itemize}\addtolength{\itemsep}{0.2\baselineskip}
\item[{\rm (i)}] $\sum_i d_i > n(r-1)$; or 
\item[{\rm (ii)}] $G = \Sp_n(k)$, $p=2$ and $\sum_ie_i = n(r-1)$,
\end{itemize}
where $n$ is the dimension of the natural $kG$-module.
\end{corol}

The proof of this corollary relies on extending our basic set up to a slightly more general situation. We will do this in a sequel and so we do not give a proof in this paper.  

Let us now turn to some applications. Recall that if $G$ is an algebraic group acting on a variety $V$, then $G$ has a \emph{generic stabilizer} on $V$ if there is a nonempty open subvariety $V_0$ of $V$ and a closed subgroup $H$ of $G$ such that the $G$-stabilizer of each point $v \in V_0$ is conjugate to $H$. Richardson \cite[Theorem A]{Ri} proved that generic stabilizers exist in characteristic $0$ if $V$ is a smooth affine irreducible variety.  However, this result does not extend to semisimple groups in positive characteristic (for example, see \cite[Theorem 1(ii)]{GL}). It is true that generic stabilizers always exist (in any characteristic) when $G$ is simple and $V$ is an irreducible $kG$-module, even as group schemes  \cite{GGL,GL}. Moreover, in the latter situation, the generic stabilizers have been determined in all cases \cite{GL} (and also for group schemes \cite{GGL}); the generic stabilizer is typically trivial whenever $\dim V > \dim G$. Indeed, by the results in \cite{GL} we deduce that if $G$ is simple, $V$ is irreducible and $\dim V > \dim G$, then the generic stabilizer is always a finite group scheme (but not necessarily smooth).  

Let $G$ be a simple algebraic group over an algebraically closed field $k$ of characteristic $p \geqs 0$ and let $V$ be a $kG$-module (possibly reducible). Set
\[
V^G = \{v \in V \,:\, gv = v \mbox{ for all $g \in G$}\}.
\]
By combining Theorem \ref{t:main2} with the main results in \cite{BGG, Ger}, we can show that if $\dim V/V^G$ is sufficiently large, then the generic stabilizer is trivial. The analogous result for Lie algebras was proved in \cite{GG1}. Moreover, when combined with the results in \cite{GG1} we can prove that generic stabilizers are trivial as a group scheme under suitable hypotheses (see Corollary \ref{c:main3} below).

In the statement of the following result, we say that $V$ is \emph{generically free} if the generic stabilizer for the action of $G$ on $V$ is trivial. Note that if $G$ is an exceptional type group, then $d(G) = 3(\dim G - {\rm rk}\, G)$.

\begin{theorem} \label{t:main3}   
Let $G$ be a simple algebraic group over an algebraically closed field $k$ of characteristic $p \geqs 0$. Let $V$ be a finite 
dimensional faithful rational $kG$-module and define $d(G)$ as in Table \ref{tab:dG}. If $\dim V/V^G > d(G)$, then $V$ is generically free.    
\end{theorem}

\renewcommand{\arraystretch}{1.1}
\begin{table}
\begin{center}
\[
\begin{array}{llll}\hline
G & d(G) & d'(G) & \mbox{Conditions} \\ \hline
{\rm SL}_n(k) &  6 &  9 & n=2  \\
 &  \frac{9}{4}n^2 & \frac{9}{4}n^2  & n \geqs 3 \\ 
{\rm Sp}_n(k) & \frac{9}{8}n^2+2 &\frac{3}{2}n^2 & \mbox{$n = 4$ or $(n,p) = (6,2)$} \\
&  \frac{9}{8}n^2 & \frac{3}{2}n^2 & \mbox{$n \geqs 6$ and $(n,p) \ne (6,2)$} \\
{\rm SO}_{n}(k) &  \frac{9}{8}n^2 & 2(n-1)^2 & n \geqs 7 \\ 
E_8(k) & 720 & 1200 & \\
E_7(k) & 378 & 630 & \\
E_6(k) & 216 &360& \\
F_4(k) & 144 &240 & \\
G_2(k) & 36 & 48  & \\ \hline
 \end{array}
\]
\caption{The values of $d(G)$ and $d'(G)$ in Theorem \ref{t:main3} and Corollary \ref{c:main3}}
\label{tab:dG}
\end{center}
\end{table}
\renewcommand{\arraystretch}{1}

\begin{remk}\label{r:generic}
Note that if $\dim V < \dim G$, then the generic stabilizer is clearly positive dimensional. Here we highlight some interesting examples with $\dim V > \dim G$.
\begin{itemize}\addtolength{\itemsep}{0.2\baselineskip}
\item[{\rm (a)}] Let $G = \SL(W) = \SL_n(k)$ and $V = \Sym^2(W) \oplus \Sym^2(W)$ with $p \ne 2$, so 
\[
\dim V = n(n+1)>\dim G.
\] 
Here the generic stabilizer is equal to the intersection of two generic conjugates of the orthogonal subgroup ${\rm SO}(W)$, which is an elementary abelian $2$-group of rank $n-1$ (see \cite[Theorem 8]{BGS}).   
\item[{\rm (b)}] If $G=\Sp(W)= \Sp_n(k)$ and we take $V$ to be the direct sum of $n-1$ copies of $W$, then $\dim V = n(n-1)$ and the generic stabilizer is positive dimensional. Indeed, if $G$ fixes a generic point of $V$, then it acts trivially on a hyperplane of $W$ and hence fixes the $1$-dimensional radical $R$ of this hyperplane. The stabilizer of $R$ is a maximal parabolic subgroup $P = QL$, where the unipotent radical $Q$ is the subgroup of $P$ which acts trivially on the hyperplane $R^{\perp}$. So the generic stabilizer is $Q$, which  has dimension $2n-1$. 
\item[{\rm (c)}] If $G=\SO(W) = \SO_n(k)$, $p \ne 2$ and $V=L(\omega_2)$ is the nontrivial composition factor of $\Sym^2(W)$, then
the generic stabilizer is a nontrivial elementary abelian $2$-group.  
\end{itemize}
\end{remk}
 
By combining Theorem \ref{t:main3} with \cite[Theorem A]{GG1}, we obtain the following corollary. Recall that $p$ is \emph{special} for a simple algebraic group $G$ if $p=3$ and $G=G_2$, or if $p=2$ and $G$ is of type $B_n$, $C_n$ or $F_4$.  Let $\mathfrak{g}$ be the Lie algebra of $G$, with derived subalgebra $[\mathfrak{g},\mathfrak{g}]$.

\begin{corol} \label{c:main3}  
Let $G$ be a simple algebraic group over an algebraically closed field $k$ of characteristic $p \geqs 0$. Let $V$ be a finite 
dimensional faithful rational $kG$-module and define $d'(G)$ as in Table \ref{tab:dG}.   Let $V'$ be the subspace of $V$  annihilated by $[\mathfrak{g},\mathfrak{g}]$.
If $\dim V/V' > d'(G)$ and $p$ is not special for $G$, then there exists a nonempty open subset $V_0$ of $V$ such that the stabilizer of each $v \in V_0$ is trivial as a group scheme.  
\end{corol} 

Note that the condition on $p$ in Corollary \ref{c:main3} is necessary. Indeed, if $p$ is special then we refer the reader to \cite{GG3} for examples where $\dim V$ is arbitrarily large, $V'=0$ and the 
generic stabilizer is nontrivial. The proof of Theorem \ref{t:main3} is presented in Section \ref{s:generic}, together with a short argument for  Corollary \ref{c:main3}.

Finally, let us present a completely different application of our results to a problem on the random generation of finite simple groups, which was originally studied by Liebeck and Shalev (see \cite{LiSh0}). Let $L$ be a finite group, let $r,s$ be prime divisors of $|L|$ and let $I_m(L)$ be the set of elements in $L$ of order $m$. Then  
\begin{equation}\label{e:Prs}
\mathbb{P}_{r,s}(L) = \frac{|\{(x,y) \in I_r(L) \times I_s(L) \,:\, L = \la x,y \ra\}|}{|I_r(L)||I_s(L)|}
\end{equation}
is the probability that $L$ is generated by a random pair of elements $(x,y) \in L \times L$, where $x$ has order $r$ and $y$ has order $s$. We say that $L$ is \emph{$(r,s)$-generated} if $\mathbb{P}_{r,s}(L)>0$.

Recall that every finite simple group is $2$-generated. With this result in hand, it is natural to ask how the generating pairs for a simple group are distributed across the group; the related problem of determining the existence (and abundance) of generating pairs of elements of prime order has been studied for more than a century. In this direction there has been a particular interest in understanding the simple groups that are $(2,3)$-generated, noting that they coincide with the finite simple quotients of the modular group ${\rm PSL}_{2}(\Z) = \Z_2 \ast \Z_3$. As far back as 1901, Miller \cite{Miller} proved that every alternating group of degree $n \geqs 9$ is $(2,3)$-generated. The main theorem of \cite{LiSh0} states that if $G \ne {\rm PSp}_{4}(q)$ is a finite simple classical or alternating group, then $\mathbb{P}_{2,3}(G) \to 1$ as $|G| \to \infty$, and this result was recently extended to the exceptional groups of Lie type in \cite[Theorem 8]{GLLS} (excluding the Suzuki groups, which do not contain elements of order $3$). It is interesting to note that the $4$-dimensional symplectic groups are genuine exceptions (see \cite{LiSh0}):
\[
\lim_{f \to \infty}\mathbb{P}_{2,3}({\rm PSp}_{4}(p^f)) = \left\{\begin{array}{ll} 0 & \mbox{if $p=2,3$} \\
1/2 & \mbox{if $p \geqs 5$}
\end{array}\right.
\]
Indeed, none of the groups ${\rm PSp}_{4}(p^f)$ with $p \in \{2,3\}$ are $(2,3)$-generated \cite[Theorem 1.6]{LiSh0}.

For fixed primes $r,s$ (with $s > 2$), the main theorem of \cite{LiSh} states that $\mathbb{P}_{r,s}(G) \to 1$ for all finite simple classical groups $G$ of sufficiently large rank (where the bound on the rank depends on $r$ and $s$) and for all alternating groups of sufficiently large degree. Gerhardt \cite[Theorem 1.4]{Ger} has recently proved that if $(G_i)$ is a sequence of linear or unitary groups of fixed rank, where $|G_i| \to \infty$ and each $|G_i|$ divisible by $r$ and $s$, then $\mathbb{P}_{r,s}(G_i) \to 1$. An analogous result for exceptional groups of Lie type was established in \cite[Theorem 12]{BGG}. As an application of Theorem \ref{t:main2}, we can extend the above results to all finite simple groups (using the earlier work in \cite{LiSh0, LiSh} to reduce the problem to Lie type groups of bounded rank).  

\begin{theorem} \label{t:main4}  
Fix primes  $r,s$ with $s > 2$ and let $\mathcal{S}_{r,s}$ be the set of finite simple 
groups whose order is divisible by both $r$ and $s$. Let $(G_i)$ be a sequence of simple groups in $\mathcal{S}_{r,s}$ with $|G_i| \to \infty$. Then either $\mathbb{P}_{r,s}(G_i) \to 1$, or 
$(r,s) \in \{ (2,3), (3,3)\}$ and there is an infinite subsequence of groups of the form ${\rm PSp}_{4}(q)$.
\end{theorem}  

\begin{remk}\label{r:sp4}
As noted above, the anomaly of the groups ${\rm PSp}_4(q)$ when $(r,s) = (2,3)$ was originally observed by Liebeck and Shalev \cite{LiSh0}. In Theorem \ref{t:main4} we see that the case $(r,s) = (3,3)$ is also noteworthy. Indeed, if we write $q=p^f$ with $p$ a prime, then we will show that 
\[
\lim_{f \to \infty}\mathbb{P}_{3,3}({\rm PSp}_{4}(p^f)) = \left\{\begin{array}{ll} 0 & \mbox{if $p=3$} \\
1/2 & \mbox{if $p=2$} \\
3/4 & \mbox{if $p \geqs 5$}
\end{array}\right.
\]
See Theorem \ref{t:sp4rs} for a proof (we also include a new proof of the 
result from \cite{LiSh0} when $(r,s) = (2,3)$). 
\end{remk}

As an application of Theorem \ref{t:main4}, we prove Corollary \ref{c:sylow} below on the  generation of simple groups by two Sylow subgroups (see the end of Section \ref{s:random} for the proof). Our main motivation stems from the following conjecture. 

\begin{conj}\label{c:syl}
Let $G$ be a finite simple group and let $r$ and $s$ be primes dividing $|G|$. Then there exists a Sylow $r$-subgroup $P$ and a Sylow $s$-subgroup $Q$ of $G$ such that $G = \la P, Q\ra$.
\end{conj}

By the main theorem of \cite{Gur}, this conjecture holds if $r=s=2$, and more generally if $r=2$ by \cite{BG23}.  It has also been verified for all sporadic  and alternating groups by Breuer and Guralnick. In addition, \cite[Theorem 1.8]{BrG} shows that if $G$ is simple and $r$ is any prime divisor of $|G|$, then there exists a prime divisor $s$ of $|G|$ such that $G = \la P,Q\ra$ for some Sylow $r$-subgroup $P$ and Sylow $s$-subgroup $Q$ of $G$.  Here we establish the following asymptotic version, which verifies Conjecture \ref{c:syl} for all sufficiently large finite simple groups.

\begin{corol}\label{c:sylow}
Let $r$ and $s$ be primes. Then for all sufficiently large finite simple groups $G$ with $|G|$ divisible by $r$ and $s$, there exists a Sylow $r$-subgroup $P$ and a Sylow $s$-subgroup $Q$ of $G$ such that $G = \la P,Q \ra$.
\end{corol}

We refer the reader to Remark \ref{r:sylow} for some additional comments on the probability that a simple group is generated by two randomly chosen Sylow subgroups corresponding to fixed primes $r$ and $s$.

We close the introduction with some brief comments on the organization of the paper. In Section \ref{s:thm1}, we study the general set up and we prove Theorem \ref{t:main1}. Here we also present several additional results that will play a key role in the proof of Theorem \ref{t:main2} (for example, see Lemmas \ref{l:closures} and \ref{l:parabolic}). Section \ref{s:prel2} covers a wide range of preliminary results that we will need in the proof of Theorem \ref{t:main2}, most of which are set up specifically for the case we are interested in, where $G$ is a classical group and $X$ is a product of conjugacy classes. In particular, this section includes various results that allow us to deduce that the groups $G(x)$ satisfy a certain property on a generic subset of $X$ just from the existence of such a group for a specific tuple $x \in X$. The proof of Theorem \ref{t:main2} is presented in Sections \ref{s:ort} (orthogonal groups) and \ref{s:symp} (symplectic groups), with the analysis partitioned in to various subcases. The main arguments are inductive on the rank of $G$, with Gerhardt's theorem for ${\rm SL}_n(k)$ in \cite{Ger} playing a key role. Finally, our main applications are discussed in Sections \ref{s:generic} (generic stabilizers) and \ref{s:random} (random generation), including the proofs of Theorems \ref{t:main3} and \ref{t:main4}. We close by presenting a proof of Corollary \ref{c:cor2} in Section \ref{s:corol}.

\section{Proof of Theorem \ref{t:main1}}\label{s:thm1}

In this section we prove Theorem \ref{t:main1}. Unless stated otherwise, $G$ is a simply connected simple algebraic group over an algebraically closed field $k$ of characteristic $p \geqs 0$. Let $r \geqs 2$ be an integer and let $X$ be a locally closed irreducible subvariety of $G^r = G \times \cdots \times G$ (with $r$ factors). Recall that if $x =(x_1, \ldots, x_r)  \in X$, then $G(x)$ denotes the Zariski closure of $\langle x_1, \ldots, x_r \rangle$ and we define 
\[
\Delta = \{ x \in X \,:\, G(x) = G \}
\]
as in \eqref{e:delta}. For a closed subgroup $H$ of $G$, we set
\begin{equation}\label{e:XH}
X_H = \{ x \in X \,:\, \mbox{$G(x) \leqs H^g$ for some $g \in G$}\},
\end{equation}
which coincides with the set of elements $x \in X$ such that $G(x)$ has a fixed point on the coset variety $G/H$. Note that if $k'$ is a field extension of $k$, then we can consider $G(k')$, $X(k')$, $\Delta(k')$, etc., which are defined in the obvious way.  

Recall that a subset of $X$ is \emph{generic} if it contains the complement of a countable union of proper closed subvarieties of $X$. We say that $G(x)$ is \emph{generically $\mathcal{P}$} for some property $\mathcal{P}$ if $G(x)$ has the relevant property for all $x$ in a nonempty generic subset of $X$. Although a generic subset of $X$ may have no points over $k$, \cite[Lemma 2.4]{BGGT} implies that every generic subset is dense if $k$ is an uncountable algebraically closed field.

We begin by recording some consequences of \cite[Theorem 2]{BGG}. The following result was stated as Theorem \ref{t:bgg} in Section \ref{s:intro}. 

\begin{thm}\label{t:gen}  Let $k$ be an algebraically closed  field that is not algebraic over a finite field. Then the following are equivalent:
\begin{itemize}\addtolength{\itemsep}{0.2\baselineskip}
\item[{\rm (i)}] $\Delta(k')$ is a nonempty subset of $X(k')$ for some field extension $k'/k$.
\item[{\rm (ii)}] $\Delta$ is a dense subset of $X$.
\item[{\rm (iii)}] $\Delta$ is a generic subset of $X$.
\end{itemize} 
\end{thm}

\begin{proof}
The equivalence of (i) and (ii) is \cite[Theorem 2]{BGG}. Suppose $\Delta$ is a generic subset of $X$ and let $k'$ be an uncountable field extension of $k$. Then $\Delta(k')$ is dense in $X(k')$ and thus (i) holds. Therefore, to complete the proof of the theorem it suffices to show that if $\Delta$ is nonempty then it is generic.

Let $\mathcal{M}$ be a set of representatives of the conjugacy classes of maximal positive dimensional closed subgroups of $G$, which is finite by \cite[Corollary 3]{LS04}. For $m \in \mathbb{N}$ and $H \in \mathcal{M}$, set
\[
X_m = \{x \in X \, : \, |G(x)| \leqs m\}
\]
and define $X_H$ as in \eqref{e:XH}. Note that $\Delta$ is the complement in $X$ of the countable union 
\[
\bigcup_{m \in \mathbb{N}}X_m \cup \bigcup_{H \in \mathcal{M}}X_H.
\]
Here each $X_m$ is a closed subvariety of $X$ and Lemma \ref{l:parabolic} below implies that the same conclusion holds for $X_H$ when $H$ is a parabolic subgroup. So it remains to show that if $\Delta$ is nonempty and $H \in \mathcal{M}$ is nonparabolic, then $X_H$ is not dense in $X$.  

Let $H \in \mathcal{M}$ be a nonparabolic subgroup and let $V$ be an irreducible finite dimensional rational $kG$-module on which $H$ acts reducibly (as noted in the proof of \cite[Lemma 2.5]{BGG}, there exists a finite collection of such modules with the property that each $H \in \mathcal{M}$ acts reducibly on at least one of them). Then $G(x)$ is reducible on $V$ for all $x$ in the closure of $X_H$, so $X_H$ is not dense if $\Delta$ is nonempty. The result follows.
\end{proof}

\begin{lem} \label{l:closures}   
Let $\bar{X}$ be the Zariski closure of $X$ in $G^r$ and assume $G(x) = G$ for some $x \in \bar{X}$. Then $\Delta$ is a dense and generic subset of $X$.
\end{lem}
 
\begin{proof}  
First recall that $X$ is a locally closed subset of $G^r$, which means that $X$ is open in $\bar{X}$. Let us also note that $\bar{X}$ is irreducible. Set $\Delta' = \{x \in \bar{X} \,:\, G(x) = G\}$, which we are assuming is nonempty (so in particular, $k$ is not algebraic over a finite field). By \cite[Theorem 2]{BGG}, it follows that $\Delta'$ is a dense subset of $\bar{X}$, whence 
$\Delta = \Delta' \cap X$ is nonempty and we conclude by applying Theorem \ref{t:gen}.
\end{proof} 

Next we consider the action of $G$ on a complete variety. If $G$ acts on a variety $Y$, then we write $Y^g$ to denote the subvariety of fixed points of $g \in G$. Similarly, if $H$ is a closed subgroup of $G$ then we define
\[
Y^H = \{ y \in Y \,:\, \mbox{$hy = y$ for all $h \in H$}\} = \bigcap_{h \in H}Y^h.
\] 

\begin{lem}\label{l:complete}   
Suppose $G$ acts on a complete variety $Y$ and $S$ is a subset of $G$.
\begin{itemize}\addtolength{\itemsep}{0.2\baselineskip}
\item[{\rm (i)}] We have $\dim Y^g \geqs  \min\{\dim Y^s \,:\,  s \in S\}$ for all $g \in \bar{S}$.
\item[{\rm (ii)}] If every $g \in S$ has a fixed point on $Y$, then the same is true for all $g \in \bar{S}$.
\end{itemize} 
\end{lem}

\begin{proof}   
Let $d \geqs 0$ be an integer and set $G(d) = \{ g \in G \,:\, \dim Y^g  \geqs d\}$. Consider the closed subvariety $W:=\{(g,y) \in G \times Y \,:\,  gy = y\}$. 
The projection map $\pi: G \times Y \to G$ is closed since $Y$ is complete, so $\pi(W)$ is a closed subset of $G$ and it follows that $G(d)$ is also closed. 

Set $d =  \min\{\dim Y^s \,:\, s \in S\}$. Then $S \subseteq G(d)$ and thus $\bar{S} \subseteq G(d)$, proving (i). Similarly, if $S$ is contained in $\pi(W)$, then so is $\bar{S}$ and we deduce that (ii) holds. 
\end{proof}

\begin{rem}\label{r:linear} 
With a minor modification, we can extend Lemma \ref{l:complete} to the case where $Y$ is a $kG$-module. To do this, we consider the induced action of $G$ on the projective space $Y_0 = \mathbb{P}^1(Y)$, which is a complete variety. For $g \in G$, let $\a(g)$ be the dimension of the largest eigenspace of $g$ on $Y$. Then $\dim Y_0^g = \a(g) - 1$ and by applying Lemma \ref{l:complete}(i) we deduce that if $S$ is a subset of $G$, then
$\a(g) \geqs \min\{\a(s) \,:\, s \in S\}$
for all $g \in \bar{S}$.
\end{rem} 

Recall that if $P$ is a parabolic subgroup of $G$, then the homogeneous space $Y = G/P$ is a projective (and hence complete) variety.  

\begin{lem} \label{l:parabolic} 
If $P$ is a proper parabolic subgroup of $G$, then $X_P$ is a closed subvariety of $X$.
\end{lem}

\begin{proof} 
Let $Y = G/P$ and set 
\[
W=\{(x, y) \in X \times Y \,:\, G(x)y=y\}.
\]
Let $\pi$ be the projection map from $W$ into $X$. Since $Y$ is complete, it follows that 
$\pi(W)=X_P$ is closed as required.
\end{proof} 

\begin{rem}\label{r:spaces}
Let $G$ be a classical algebraic group of the form $\Sp(V)$ or $\SO(V)$ and assume $k$ is uncountable. Suppose $G(x)$ generically preserves a totally singular $m$-dimensional subspace of $V$. By the previous lemma, the set of elements $x \in X$ such that $G(x)$ preserves a totally singular $m$-space is closed. Since every generic subset of $X$ meets every nonempty open subset, it follows that $G(x)$ must preserve a totally singular $m$-space for all $x \in X$. This basic observation will be applied repeatedly in the proof of Theorem \ref{t:main2}. 
\end{rem}

Recall that a closed subgroup $H \leqs G$ is \emph{$G$-irreducible} if it is not contained in a proper parabolic subgroup of $G$. We can now prove most of Theorem \ref{t:main1}.

\begin{thm}\label{t:generic}   
Let $k$ be an algebraically closed field. Then exactly one of the following holds:
\begin{itemize}\addtolength{\itemsep}{0.2\baselineskip}
\item[{\rm (i)}] $\Delta(k')$ is nonempty for some algebraically closed field extension $k'/k$.
\item[{\rm (ii)}] For all $x \in X$, $G(x)$ is contained in a proper parabolic subgroup of $G$.
\item[{\rm (iii)}]  There exists a unique (up to conjugacy)  proper closed $G$-irreducible subgroup $H$ of $G$ such that

\vspace{1mm}
 
\begin{itemize}\addtolength{\itemsep}{0.2\baselineskip}
\item[{\rm (a)}] $X_H$ contains a nonempty open subset of $X$; and
\item[{\rm (b)}]  $\{ x \in X \,:\, \mbox{$G(x) = H^g$ for some $g \in G$}\}$ is a generic subset of $X$.  
 \end{itemize} 
\end{itemize}
\end{thm}

\begin{proof}   
If $\Delta$ is nonempty, then it is generic and so  at most one of the three conclusions can hold.  Therefore, we may assume that neither (i) nor (ii) holds.
So $\Delta(k')$ is empty for every algebraically closed field extension $k'/k$, and $G(x)$ is $G$-irreducible for some $x \in X$. 
 
Let $X_{{\rm par}}$ be the set of tuples $x \in X$ such that $G(x)$ 
is contained in a proper parabolic subgroup of $G$. Since $G$ has only finitely many conjugacy classes of parabolic subgroups, Lemma \ref{l:parabolic} implies that $X_{{\rm par}}$ is a proper closed subvariety of $X$. 

Let $H$ be any closed subgroup of $G$. Let $\Omega=G^r$ and note that $\Omega_H$ is the image of the morphism $G \times H^r \rightarrow \Omega$ given by $(g,h_1, \ldots, h_r) \mapsto (h_1^g, \ldots, h_r^g)$.  In particular, $\Omega_H$ is open in its closure and so $X_H$ is open in $X$ (but may be empty). 

By \cite[Lemma 4.1]{Martin}, there are only countably many conjugacy classes of $G$-irreducible subgroups of $G$ and each one is defined over some finite extension of the prime subfield $k_0$ of $k$ (in particular, these subgroups are defined over the algebraic closure of $k_0$). Let $\{H_i \, :\, i \in \NN\}$ be representatives of the conjugacy classes of the proper $G$-irreducible subgroups of $G$ and set $X_i = X_{H_i}$. Since (i) and (ii) do not hold, it follows that $\bigcup_i X_i$ contains the nonempty open subvariety $X \setminus X_{{\rm par}}$.  
 
First assume $k$ is uncountable. For a closed subgroup $H$ of $G$, let 
$\bar{X}_H$ denote the closure of $X_H$ in $X$. Then $\bar{X}_H = X$ for some proper $G$-irreducible subgroup $H$ (recall that $X$ is irreducible, so it is not a countable union of proper closed subvarieties), which implies that 
$\bar{X}_H=X$ over any algebraically closed field.  Therefore, $X_H$ is open and dense in $\bar{X}_H=X$ and by the minimum condition on subvarieties, we can choose $H$ minimal subject to this condition.     

For any proper $G$-irreducible subgroup $J$, let 
\begin{equation}\label{e:YJ}
Y_J = \{ x \in X \,:\, \mbox{$G(x) = J^g$ for some $g \in G$} \}.
\end{equation}
Note that $X$ is the union of $X_{{\rm par}}$, the complement of $X_H$ and the subsets $Y_J$, where $J$ runs over a set of conjugacy class representatives of the $G$-irreducible subgroups of $H$, so the irreducibility of $X$ implies that $X=\bar{Y}_J$ for some $J$. If $J < H$ then the inclusion $Y_J \subseteq X_J$ implies that $X_J$ is dense in $X$, but this contradicts the minimality of $H$. Therefore, $J = H$ and we deduce that $Y_H$ is a generic subset of $X$ (indeed, each $x \not\in Y_H$ is contained in one of $X_{\rm par}$, the complement of $X_H$, or in one of countably many subsets $\bar{X}_L$, where each $L$ is a proper $G$-irreducible subgroup of $H$).  

To complete the proof for $k$ uncountable, we show that $H$ is unique up to conjugacy in $G$. If $L$ is a $G$-irreducible subgroup satisfying (a) and (b) in (iii), then $Y_L$ is generic and therefore intersects the open subvariety $X_H$, which in turn implies that $L$ is conjugate to a subgroup of $H$. Then by the minimality of $H$, we conclude that $L$ is conjugate to $H$, as required.

Finally, let us assume $k$ is countable and let $k'/k$ be any uncountable algebraically closed field extension. Then as above, there is a unique (up to conjugacy) proper $G(k')$-irreducible subgroup $H$ of $G(k')$ such that $X'_{H}$ is open in $X'$, where $X' = X(k')$.  The complement of this open subset is thus a proper closed subset of $X'$. Since $H$ is defined over $k$,  it follows that the complement of $X_H$ in $X$ is a proper closed subset, so $X_H$ is open and dense in $X$. The result follows.   
\end{proof} 

\begin{rem}\label{r:parab}
Consider the conclusion in part (ii) of Theorem \ref{t:generic} and let $P_1, \ldots, P_m$ be representatives of the conjugacy classes of proper parabolic subgroups of $G$. If (ii) holds, then $X = \bigcup_{i}X_{P_i} = \bigcup_{i} \bar{X}_{P_i}$ and thus $X = \bar{X}_{P}$ for some proper parabolic subgroup $P$. But $X_P$ is closed by Lemma \ref{l:parabolic}, so $X = X_P$ and we conclude that each $G(x)$ is contained in a conjugate of a fixed proper parabolic subgroup.  
\end{rem}
  
As a corollary, we obtain the following result. In the statement, we refer to $G$-orbits and $G$-invariance, which are both defined in terms of simultaneous conjugation by $G$. So for example, if $X = C_1 \times \cdots \times C_r$ is a product of conjugacy classes, then $X$ is $G$-invariant.
  
\begin{cor} \label{c:finite}  
Suppose $G(x)$ is generically finite.
\begin{itemize}\addtolength{\itemsep}{0.2\baselineskip}
\item[{\rm (i)}] There exists a positive integer $d$ such that $|G(x)| \leqs d$ for all $x \in X$.
\item[{\rm (ii)}] If $G(y)$ is $G$-irreducible for some $y \in X$, then there is a nonempty open subvariety of $X$ contained in a $G$-orbit. In particular, if $X$ is $G$-invariant, then $X$ has an open dense $G$-orbit of dimension equal to $\dim G$. 
\end{itemize}
\end{cor}
 
\begin{proof}   
We can always pass to an extension field, so without any loss of generality we may assume $k$ is uncountable.    

Suppose that $G(x)$ is infinite for some $x$. Then the set of $x \in X$ with $G(x)$ infinite is generic since each subvariety $\{x \in X \,:\, |G(x)| \leqs m \}$ with $m \in \NN$ is closed and proper. Since the intersection
of countably many generic sets over an uncountable algebraically closed field is generic, we have a contradiction. Hence  
\[
X = \bigcup_{m \in \NN} \{x \in X \,:\, |G(x)| \leqs m \}
\]
is a countable union of closed subvarieties and thus (i) follows from the irreducibility of $X$.

Now assume that $G(y)$ is $G$-irreducible for some $y \in X$. Then by (i) and Theorem \ref{t:generic},  there exists a finite $G$-irreducible subgroup $H$
of $G$ such that $X_H$ is open and $Y_H$ is generic, where $Y_H$ is defined as in \eqref{e:YJ}. In fact, since $H$ has only finitely many subgroups, the proof of Theorem \ref{t:generic} implies that $Y_H$ is open. 

First assume that $X$ is $G$-invariant. Here $X_H$ is contained in the closure of the image of the morphism 
\[
f: G \times H^r \to G^r,\;\; (g,h_1, \ldots, h_r) \mapsto (h_1^g, \ldots, h_r^g),
\] 
which implies that $X$ itself is contained in the closure of the image of $f$.  
Since $H$ is $G$-irreducible, its centralizer is finite (see 
\cite[Lemma 2.1]{LT}), and so the dimension of the image of $f$ is equal to $\dim G$.    Thus, the dimension of the $G$-orbit of any $x \in X$ with 
$G(x)=H$ is equal to $\dim G = \dim X$ and so this orbit contains a dense open subset of $X$.  This establishes (ii) in the case where $X$ is $G$-invariant.  

In the general case, we can work in the $G$-invariant variety that is the image of the morphism $G \times X \rightarrow G^r$ given
by $(g,x_1, \ldots, x_r) \mapsto (x_1^g, \ldots, x_r^g)$ and the result follows. 
 \end{proof} 
   
\begin{rem}\label{r:gen}
Suppose $X$ is $G$-invariant, $\dim X > \dim G$ and $G(x)$ is $G$-irreducible for some $x \in X$. Then Corollary \ref{c:finite} implies that $G(x)$ is generically positive dimensional. 
In the special case where $X = C_1 \times \cdots \times C_r$ is a product of noncentral conjugacy classes, it was shown in \cite{BGG}, using results from \cite{GMT}, that $G(x)$ is generically positive dimensional if $r \geqs 3$. In addition, \cite[Corollary 5.14]{GMT} shows that if $r=2$ and $G(x)$ is generically finite, then $G(x)$ is always contained in a Borel subgroup. Also see Lemma \ref{l:infinite}.
\end{rem}   
     
The final ingredient we need to complete the proof of Theorem \ref{t:main1} is provided by Theorem \ref{t:toralrank} below (recall that the \emph{rank} of a closed subgroup $H$ of $G$, denoted ${\rm rk}\, H$, is the dimension of a maximal torus of $H^0$). The bound in part (i) completes the proof of Theorem \ref{t:main1}. In order to explain the notation in part (ii), let $V$ be a nontrivial finite dimensional irreducible rational $kG$-module (the choice of $V$ is irrelevant) and write $f(g) \in k[\mathbf{x}]$ for the characteristic polynomial of $g \in G$ acting on $V$. This defines a morphism 
\begin{equation}\label{e:char}
f: G \to \mathcal{M}_d(\textbf{x}),
\end{equation}
where $\mathcal{M}_d(\textbf{x})$ is the variety of monic polynomials in $k[\mathbf{x}]$ of degree equal to $d = \dim V$. 
    
\begin{thm}\label{t:toralrank}  
Suppose there exists a closed subgroup $H \leqs G$ such that $X_H$ 
contains a nonempty open subset of $X$. Then for all $x \in X$,
     \begin{itemize}\addtolength{\itemsep}{0.2\baselineskip}
\item[{\rm (i)}] ${\rm rk}\, G(x) \leqs {\rm rk}\, H$; and
   \item[{\rm (ii)}] if $S$ is a maximal torus of $G(x)^0$, then $f(S)$ is contained in the closure of $f(H)$.  
   \end{itemize} 
    \end{thm} 
    
\begin{proof}   
Let $V$ be a nontrivial finite dimensional irreducible rational $kG$-module corresponding to the map $f$ in \eqref{e:char}, where $d = \dim V$. For each $g \in G$, observe that $f(g)$ is determined by the conjugacy class of the semisimple part of $g$. Let us also note that every fiber of $f$ is finite. To see this, let $T$ be a maximal torus of $G$. Then $f(t_1)=f(t_2)$ with $t_1, t_2 \in T$ if and only if $t_1$ and $t_2$ are conjugate in $\GL(V)$ and the claim follows because every semisimple class in $\GL(V)$ intersects $T$ in a finite set. 
    
Set $e = {\rm rk}\,H$. We claim that $\dim f(H) = e$. To see this, first observe that $f(H^0) = f(S)$, where $S$ is a maximal torus of $H^0$. Since $f$ has finite fibers, this implies that $\dim f(S) = \dim S = e$. If $h \in H$ embeds in $\GL(V)$ as a semisimple element, then 
\cite[7.5]{St}  implies that some conjugate of $h$ normalizes $S$ and this gives $\dim f(hS) \leqs \dim S = e$ (in fact, this is a strict inequality unless $h$ centralizes $S$). This justifies the claim. 
    
Let $Z \subseteq X_H$ be a nonempty open subset of $X$ and let $Y$ be the closure of $f(H)$ in the variety $\mathcal{M}_d(\mathbf{x})$. Let $x \in X$ and let $w$ be an element of the free group of rank $r$, which we may view as a map from $X$ to $G$. Then $w(x)$ is contained in the closure of $w(Z)$ and thus $f(w(x))$ is in the closure of $f(w(Z)) \subseteq Y$. Therefore $f(G(x)) \subseteq Y$, proving both parts of the theorem.   
     \end{proof}  
          
\begin{rem} 
It would be interesting to know if it is possible to replace rank by dimension in part (i) of Theorem \ref{t:toralrank}. 
\end{rem}

We are now in a position to prove Theorem \ref{t:main1}. 

\begin{proof}[Proof of Theorem \ref{t:main1}] 
We combine Theorems \ref{t:generic} and \ref{t:toralrank}(i). More precisely, if (i) holds in Theorem \ref{t:generic}, then Theorem \ref{t:bgg} implies that $\Delta$ is dense and generic in $X$, so the conclusions in case (ii) of Theorem \ref{t:main1} are satisfied with $H=G$, $Y = \Delta$ and $Z=X$. Clearly, if  Theorem \ref{t:generic}(ii) holds then we are in case (i) of Theorem \ref{t:main1}. Finally, if part (iii) in Theorem \ref{t:generic} holds, then the existence of $H$ and the appropriate subsets $Y,Z \subseteq X$ in parts (b) and (c) of Theorem \ref{t:main1}(ii) follows immediately, and the rank condition in (ii)(a) follows from Theorem \ref{t:toralrank}(i). This completes the proof of Theorem \ref{t:main1}.
\end{proof}

We close this section by recording the following corollary.   

\begin{cor} \label{c:toralrank}   
If $k$ is uncountable, then
\[
\{x \in X \,:\, {\rm rk}\, G(x) = m\}
\]
is a generic subset of $X$, where $m = \max\{ {\rm rk}\,G(x) \,:\, x \in X\}$.
\end{cor} 
     
     \begin{proof}   
     If $G(x)$ is $G$-irreducible for some $x \in X$ then the result follows immediately
     from Theorems \ref{t:generic} and \ref{t:toralrank}, so we may assume that $G(x)$ is contained in a proper parabolic subgroup for every $x \in X$. Fix a parabolic subgroup $P$ of $G$ which is minimal with respect to the property that each group $G(x)$ is contained in a conjugate of $P$ (see Remark \ref{r:parab}). Let $Q$ be the unipotent radical of $P$ and note that $X = X_P$. 
     
     Let $Y$ be an irreducible component of $X \cap P^r$ such that the morphism $\psi:G \times Y \to X$ sending 
      $(g,y_1, \ldots, y_r)$ to $(y_1^g, \ldots, y_r^g)$ is dominant. By the minimality of $P$, the set 
      \[
      \{ y \in Y \,:\, \mbox{$G(y)Q$ is contained in a proper parabolic subgroup of $P$}\}
     \]
      is a proper closed subvariety of $Y$. Since $P/Q$ has only countably many $P/Q$-irreducible subgroups, the proof of Theorem \ref{t:generic} shows that there exists a closed subgroup $H$ of $P$ such that $HQ/Q$ is $P/Q$-irreducible and the sets 
     \[
     \{ y \in Y \,:\, \mbox{$G(y)Q$ is $G$-conjugate to a subgroup of $HQ$}\}
     \]
     and
     \[
     \{ y \in Y \,:\, \mbox{$G(y)Q$ is $G$-conjugate to $HQ$}\}
     \]
     are open and generic in $Y$, respectively. Then by arguing precisely as in the proof of Theorem \ref{t:toralrank}, we deduce that ${\rm rk}\, G(y) \leqs {\rm rk}\, HQ = {\rm rk}\, H$ for all $y \in Y$, with equality on a generic subset of $Y$. Since $\psi(\{1\} \times Y)$ contains a nonempty open subset of ${\rm im}(\psi)$, it follows that equality holds on a generic subset of $X$.  This completes the proof. 
\end{proof}  

\section{Preliminaries for Theorem \ref{t:main2}}\label{s:prel2}

In this section we record various results that will be needed in the proof of Theorem \ref{t:main2}, which is our main result on the topological generation of classical algebraic groups. 

Let $G$ be a simply connected simple algebraic group over an algebraically closed field $k$ of characteristic $p \geqs 0$. Let $r \geqs 2$ be an integer and let $X$ be a locally closed irreducible subvariety of $G^r$. We begin by presenting some general results, before specializing to the case where $X = C_1 \times \cdots \times C_r$ is a product of noncentral conjugacy classes. We define $G(x)$ and $\Delta$ as before. In view of Theorem \ref{t:gen}, in order to prove Theorem \ref{t:main2} we may (and do) assume $k$ is uncountable. 

\subsection{Modules}\label{ss:gen}

Recall that since $k$ is uncountable, every generic subset of $X$ is nonempty and dense. In particular, $\Delta$ is nonempty if it contains the intersection of countably many generic subsets. Therefore, we are interested in establishing the genericity of certain subsets of $X$ defined in terms of $G$ (recall that $G(x)$ is \emph{generically $\mathcal{P}$} for some property $\mathcal{P}$, if $G(x)$ has the relevant property for all $x$ in a generic subset of $X$, and similarly for the connected component $G(x)^0$). 

With this goal in mind, the collection of results presented in the next two lemmas will be useful in the proof of Theorem \ref{t:main2}. A version of the following result is proved in \cite[Section 3]{BGGT} (see Lemma \ref{l:submodules-conn} below for a version that is stated in terms of  the connected components).

\begin{lem} \label{l:submodules}
Let $V$ be a finite dimensional rational $kG$-module.
\begin{itemize}\addtolength{\itemsep}{0.2\baselineskip}
\item[{\rm (i)}] If there exists $y \in X$ such that  $G(y)$ has a $d$-dimensional composition factor on $V$, then the set of $x \in X$ such that $G(x)$ has a composition factor on $V$ of dimension at least $d$ is a nonempty open subset of $X$. In particular, if $G(y)$ acts irreducibly on $V$ for some $y \in X$, then $G(x)$ is irreducible on $V$ for $x$ in a nonempty open subset of $X$.  
\item[{\rm (ii)}]  If $d$ is the minimal dimension of a composition factor (respectively, nonzero submodule) of $G(y)$ on $V$ for some $y \in X$, then for all $x$ in a 
nonempty open subset of $X$, the minimal dimension of a composition factor (respectively, nonzero submodule) of $G(x)$ is at least $d$.
\item[{\rm (iii)}] If $\dim C_{\mathrm{End}(V)}(G(y)) = d$ for some $y \in X$, then $\dim C_{\mathrm{End}(V)}(G(x)) \leqs d$ for all $x$ in a nonempty open subset of $X$.  
\end{itemize} 
\end{lem} 

\begin{proof} 
Set $n = \dim V$ and let $F(t_1, \ldots, t_r)$ be a polynomial identity for each matrix algebra $M_e(k)$ with $e < d$, but not for $M_d(k)$ (see \cite[Section 1.3]{Ro}). 
Then $F^n(x) \ne 0$ if and only if $G(x)$ has a composition factor on $V$ of dimension at least $d$ and this is clearly an open condition, proving (i).  

By Lemma \ref{l:parabolic}, the set of $x \in X$ such that $G(x)$ fixes a  subspace of dimension less than $d$ is a closed subvariety. Our assumption is that this is a proper subvariety, so the assertion regarding submodules in (ii) follows.   For composition factors,
we modify the argument slightly. Consider the set of flags of $V$ of the form $0 \subseteq U \subseteq W \subseteq V$ with $\dim W/U < d$.  This is a finite union of projective varieties
(one for each possible pair of dimensions of $U$ and $W$) and thus the proof of Lemma \ref{l:parabolic} implies that the set of $x \in X$ such that $G(x)$ has a fixed space on this variety is closed. Our assumption in (ii) implies that this is a proper closed subvariety and hence (ii) follows. 

Finally, part (iii) is clear since
$\dim C_{\mathrm{End}(V)}(G(x))$ is upper semicontinuous. 
\end{proof}  

\begin{rem}\label{r:words}
Let $F$ be the free group on $r$ generators. As noted in Section \ref{s:intro}, given a collection of words $w_1, \ldots, w_m$ in $F$, which we view as maps $G^r \to G$, the set 
\[
X_w:=\{(w_1(x), \ldots, w_m(x))\,:\, x \in X \}
\]
for $w=(w_1, \ldots, w_m)$ is also an irreducible
subvariety of $G^r$. In particular, we can apply all of our general results to $X_w$. 
\end{rem}

Let $d$ be a positive integer and let $F$ be the free group on $r$ generators, where $X \subseteq G^r$ as above. Let $F_d$ be the intersection of all subgroups of $F$ with index at most $d$ and note that $F_d$ is a characteristic subgroup of $F$ with finite index. In particular, $F_d$ is finitely generated and we may choose generators $w_1, \ldots, w_m$, where we view each $w_i$ as a word map from $G^r$ to $G$. Then for $x \in X$ we define
\[
G_d(x) = \overline{\langle w_1(x), \ldots, w_m(x) \rangle} \leqs G,
\]
which is independent of the choice of generators for $F_d$. 

The following result records some basic observations. Recall that a group acts \emph{primitively} on a vector space if it does not preserve a nontrivial direct sum decomposition of the space.

\begin{lem} \label{l;primitive}  
\mbox{ }
\begin{itemize}\addtolength{\itemsep}{0.2\baselineskip}
\item[{\rm (i)}]  $G(x)^0 \leqs G_d(x) \leqs G(x)$ for all $x \in X$, $d \in \mathbb{N}$.
\item[{\rm (ii)}] Let $V$ be a finite dimensional irreducible rational $kG$-module and let $d$ be an integer such that $d \geqs \dim V$. Assume that

\vspace{1mm}

\begin{itemize}\addtolength{\itemsep}{0.2\baselineskip}
\item[{\rm (a)}] $G(x)$ is generically irreducible on $V$; and
\item[{\rm (b)}] For some $y \in X$, $G_d(y)$ acts irreducibly on a submodule $W$ of $V$ such that $\dim W > d/s$ where $s$ is the smallest
prime divisor of $d$.   
\end{itemize}
\vspace{1mm} 
\noindent Then $G(x)$ is generically primitive on $V$.
\end{itemize} 
\end{lem}

\begin{proof}  
Part (i) is clear since $G(x)^0$ has no proper closed subgroups of finite index. Now consider (ii).   There is no harm in assuming that $d=\dim V$.  Seeking a contradiction, suppose $G(x)$ is not generically primitive on $V$. 
Then the condition in (a) implies that $G(x)$ is generically conjugate to a subgroup of 
$\GL_e(k) \wr S_{d/e}$ for some proper divisor $e$ of $d$.  In particular, $G_d(x)$ is generically contained in a direct product of $d/e$ copies of $\GL_e(k)$ and so
for all $x \in X$,  the largest composition factor of $G_d(x)$ on $V$ has dimension at most $e \leqs d/s$. But this is incompatible with the condition in (b) and we have reached a contradiction.
\end{proof}  

\begin{rem}\label{r:connected}
Let $V$ be a finite dimensional irreducible rational $kG$-module. If $G(y)^0$ acts irreducibly on $V$ for some $y \in X$, then $G_d(x)$ is generically irreducible for all $d \in \mathbb{N}$. In addition, if $d$ is large enough then $G_d(y)=G(y)^0$ and so $G(x)^0$ is generically irreducible as well. In particular, this implies that $G(x)^0$, and hence $G(x)$, is generically primitive on $V$.    
\end{rem}

Next we establish a version of Lemma \ref{l:submodules} with respect to the connected components. 

\begin{lem} \label{l:submodules-conn}
Let $V$ be a finite dimensional rational $kG$-module.
\begin{itemize}\addtolength{\itemsep}{0.2\baselineskip}
\item[{\rm (i)}] If there exists $y \in X$ such that  $G(y)^0$ has a composition factor  of dimension at least $e$ on $V$, then the set of $x \in X$ such that $G(x)^0$ has a composition factor on $V$ of dimension at least $e$ is a generic subset of $X$. In particular, if $G(y)^0$ acts irreducibly on $V$ for some $y \in X$, then $G(x)^0$ is generically irreducible on $V$.
\item[{\rm (ii)}]  If the minimal dimension of a composition factor (respectively, nonzero submodule) of $G(y)^0$ on $V$ is at least  $e$ for some $y \in X$, then for generic $x \in X$, the
minimal dimension of a composition factor (respectively, nonzero submodule) of $G(x)^0$ is at least $e$.
\item[{\rm (iii)}] If $\dim C_{\mathrm{End}(V)}(G(y)^0) \leqs e$ for some $y \in X$, then $\dim C_{\mathrm{End}(V)}(G(x)^0) \leqs e$ for all $x$ in a generic subset of $X$.  
\end{itemize} 
\end{lem} 

\begin{proof}   
Suppose there exists $y \in X$ such that $G(y)^0$ has any of the properties described in the lemma. Then $G_d(y)$ has the same property for every positive integer $d$ and thus Lemma \ref{l:submodules} implies that the set of $x \in X$ such that $G_d(x)$ fails to have the given property is a proper closed subvariety of $X$. Since there are only countably many positive integers, we deduce that $G(x)^0$ satisfies the property on a generic subset of $X$. 
\end{proof} 

\begin{rem}
One can modify the proof of Lemma \ref{l:submodules-conn} in order to show that if $G(y)^0$ acts irreducibly on $V$ for some $y \in X$, then the set of $x \in X$ such that $G(x)^0$ acts irreducibly on $V$ is actually open, rather than just being generic.
This follows by noting that if $d > \dim V$, then $G_d(x)$ is irreducible on $V$ if and only if $G(x)^0$ is irreducible.  
\end{rem}

Next we introduce the notion of a strongly regular element.

\begin{defn}\label{d:sr}
Let $V$ be a finite dimensional rational $kG$-module and let $T$ be a maximal torus of $G$. We say that $x \in T$ is \emph{strongly regular on $V$} if every $x$-invariant subspace of $V$ is also $T$-invariant. 
\end{defn}

Equivalently, if $\chi_1, \ldots, \chi_m$ are the distinct characters of $T$ that occur in $V$, so $V = V_1 \oplus \cdots \oplus V_m$ with $0 \ne V_i = \{v \in V \,:\, tv = \chi_i(t)v \mbox{ for all } t \in T\}$, then $x \in T$ is strongly regular on $V$ if and only if $\chi_i(x) \ne \chi_j(x)$ for $i \ne j$. From the latter characterization it is clear that the elements in $T$ that are strongly regular on $V$ form an open subset since the complement is the intersection over all pairs $i,j$ of the closed subvarieties $\{t \in T \,:\, \chi_i(t) = \chi_j(t)\}$. Let us also note that if $G$ is a classical algebraic group, then $x \in T$ is strongly regular on the natural $kG$-module $V$ if and only if all the eigenvalues of $x$ on $V$ are distinct.

By \cite[Theorem 11.7]{GT}, there is a finite collection $\mathcal{M}$ of finite dimensional irreducible rational  $kG$-modules such that no proper closed subgroup of $G$ acts irreducibly on all of these modules. Then with respect to this collection of modules, we say that $x \in T$ is  \emph{strongly regular} if it is strongly regular on each module in $\mathcal{M}$.  

\begin{rem}\label{r:strong}
Note that the set of strongly regular elements in $T$ is open. Moreover, since each regular semisimple element in $G$ is conjugate to an element of $T$, and since the strongly regular property is invariant under conjugation
and the set of regular semisimple elements in $G$ is open, 
it follows that the set of strongly regular elements in $G$ is also open. 
\end{rem}

\begin{lem} \label{l:sr1} 
Let $V$ be a finite dimensional rational $kG$-module, let $T$ be a maximal torus of $G$ and suppose $x \in T$ is strongly regular on $V$. Let $H$ be a closed
subgroup of $G$ containing $x$ and assume $V$ has an $H$-invariant subspace that is not $G$-invariant. Then $\langle H, T \rangle$ is contained in a proper closed subgroup of $G$.
\end{lem}

\begin{proof}   
Let $W$ be an $H$-invariant subspace of $V$ that is not $G$-invariant. Since $W$ is $x$-invariant, it decomposes into a direct sum of eigenspaces for $x$ and so it is also $T$-invariant.  Therefore $\langle H, T \rangle$ preserves $W$ and the result follows. 
\end{proof}
 
In part (iii) of the next result, recall that $X_H$ is defined in \eqref{e:XH}. Also recall that a closed subgroup of $G$ has \emph{maximal rank} if it contains a maximal torus of $G$.

\begin{lem} \label{l:sr}    
Let $V$ be a finite dimensional irreducible rational $kG$-module and suppose there exists $u \in X$ such that $G(u)$ contains an element that is strongly regular on $V$. Then the following hold:
\begin{itemize}\addtolength{\itemsep}{0.2\baselineskip}
\item[{\rm (i)}]  There exists a nonempty open subset $Y$ of $X$ such that $G(y)$ contains a strongly regular element on $V$ for all $y \in Y$.
\item[{\rm (ii)}] If $G(u)^0$ contains a strongly regular element on $V$, then for generic $x \in X$, $G(x)^0$ contains a strongly regular element on $V$.  
\item[{\rm (iii)}]  Either $\Delta$ is nonempty, or there exists a maximal closed maximal rank subgroup $H$ of $G$ such that $X_H$ contains a nonempty open subset of $X$.
\end{itemize}
\end{lem} 

\begin{proof}  
Let $\mathcal{S}$ be the set of elements in $G$ that are strongly regular on $V$ and recall that $\mathcal{S}$ contains a nonempty open subset of $G$ (see Remark \ref{r:strong}). Since $G(u) \cap \mathcal{S}$ is nonempty, it follows that there is a word $w$ in the free group $F$ of rank $r$, which we may view as a map $w:X \to G$, such that $w(u) \in \mathcal{S}$ (since the abstract group generated by the coordinates of $u$ is dense in $G(u)$ by definition). This implies that $Y = \{y \in X \,:\, w(y) \in \mathcal{S}\}$ is a nonempty open subset of $X$, which  proves (i). 

Now let us turn to (ii). By the proof of (i), we deduce that $\{x \in X \,:\, G_d(x) \cap \mathcal{S} \ne \emptyset\}$ is a nonempty open subset of $X$ for each positive integer $d$. The desired result now follows since $G_d(x) = G(x)^0$ for $d \gg 0$.

Finally, let us consider (iii). As above, we define strongly regular elements in $G$ with respect to a finite set $\mathcal{M}$ of irreducible $kG$-modules. We may assume $\Delta$ is empty, which implies that each $G(x)$ acts reducibly on at least one of the modules in $\mathcal{M}$. Define $w \in F$ as in the first paragraph of the proof and fix $x \in X$ such that $w(x)$ is strongly regular. Let $V$ be a module in $\mathcal{M}$ on which $G(x)$ acts reducibly and define $\mathcal{S}$ and $Y$ as above, with respect to this module. Then any $w(x)$-invariant subspace of $V$ is also $T$-invariant, where $T=C_G(w(x))$, so $\langle G(x), T \rangle$ acts reducibly on $V$ and thus $G(x)$ is contained in a maximal closed maximal rank subgroup of $G$. Since there are only finitely many conjugacy classes of such subgroups, it follows that there is a maximal closed maximal rank subgroup $H$ and a nonempty open subset $Z \subseteq Y$ such that $G(z)$ is conjugate to a subgroup of $H$ for all $z \in Z$.  
\end{proof} 

We close with three results that will be applied directly in the proof of Theorem \ref{t:main2}. The first is essentially a corollary of part (i) of Lemma \ref{l:submodules}. 

\begin{cor} \label{c:sr}  
Let $G$ be one of the classical groups $\SL_n(k)$ $(n \geqs 2)$, $\Sp_n(k)$ $(n \geqs 4)$, or $\SO_n(k)$ $(n \geqs 3, \, n \ne 4)$ and let $V$ be the natural $kG$-module. 
Assume there exists $x,y \in X$ such that $G(x)^0$ acts irreducibly on $V$ and $G(y)$ contains a strongly regular element on $V$. Then either
\begin{itemize}\addtolength{\itemsep}{0.2\baselineskip}
\item[{\rm (i)}]  $\Delta$ is nonempty; or
\item[{\rm (ii)}] $G=\Sp_n(k)$, $p=2$ and $G(x)^0$ is generically contained in a conjugate of 
$\SO_n(k)$.
\end{itemize}
\end{cor}

\begin{proof} 
Suppose $\Delta$ is empty. By applying Lemmas \ref{l:submodules}(i) and \ref{l:sr}, we deduce that there is a maximal closed maximal rank subgroup $H$ of $G$ such that $G(x)^0$ is irreducible on $V$ and is contained in a conjugate of $H^0$ for generic $x \in X$. By considering the connected irreducible maximal rank subgroups of $G$, we deduce that $G=\Sp_n(k)$, $p=2$ and $H^0 = \SO_n(k)$ is the only possibility.
\end{proof}  

The following result is an easy consequence of the classification of low dimensional representations of simple algebraic groups. 

\begin{lem} \label{l:largeranksubs}
Let $G$ be one of the groups $\Sp_n(k)$ $(n \geqs 6)$ or $\SO_n(k)$ $(n \geqs 9)$, and let $H$ be a closed connected proper subgroup of $G$ that acts irreducibly on the natural $kG$-module $V$.
\begin{itemize}\addtolength{\itemsep}{0.2\baselineskip}
\item[{\rm (i)}]  If $G=\Sp_n(k)$, then either ${\rm rk}\, H \leqs \lfloor n/4 \rfloor +1$, or $p=2$ and $H = \SO_n(k)$.
\item[{\rm (ii)}]  If $G=\SO_n(k)$, then ${\rm rk}\, H \leqs \lfloor n/4 \rfloor +1$ if $n$ is even, otherwise ${\rm rk}\, H \leqs (n+1)/4$.
\end{itemize}
\end{lem}

\begin{proof}
First assume $H$ is not simple, in which case $V$ is tensor decomposable as a $kH$-module.  If $n=2m$ with $m$ even, then the largest rank self-dual non-simple closed connected subgroup of $G$ is of the form $\Sp_2(k) \otimes L$, where $L = \Sp_m(k)$ or $\SO_m(k)$, which has rank $m/2+1 = n/4+1$. Similarly, if $G$ is symplectic and $n=2m$ with $m$ odd, then the same argument shows that the maximum rank is $(m + 1)/2 = \lfloor n/4\rfloor+1$ when $p \ne 2$. Here the bound is even better when $p=2$ since no group has a nontrivial odd-dimensional irreducible self-dual module in even characteristic. 

The remaining cases where $G$ is orthogonal and $H$ is non-simple can be handled in a similar fashion. If $n = 2m$ with $m$ odd, then we may assume $p \ne 2$ (since there are no closed connected tensor decomposable subgroups of $G$ when $p=2$). Here the largest tensor decomposable subgroups of $G$ are of the form $L_a \otimes L_b$, where $L_a$ is a symplectic or orthogonal group with an $a$-dimensional natural module (and similarly for $L_b$), $2m=ab$ and $3 \leqs a < b$. If $n$ is odd, then $p \ne 2$ and the same argument applies.  

Finally, suppose $H$ is simple. By inspecting \cite{Lu}, one checks that aside from a handful of very low rank cases (excluded by the conditions on $n$ in the statement of the lemma), the self-dual irreducible $kH$-modules have relatively large dimension (excluding Frobenius twists of the natural module when $H$ is a classical group). The result quickly follows.
\end{proof} 

Finally, we present the following well known and elementary observation.  See \cite[Lemma 3.14]{BGS}, for example.

\begin{lem} \label{l:quadratic}   
Let $a, b \in \GL(V) = \GL_n(k)$ be quadratic elements. Then each composition factor of $\langle a, b \rangle$ on $V$ is at most $2$-dimensional.  In particular, if $n \geqs 3$ then  $\langle a, b \rangle$ acts reducibly on $V$. 
\end{lem}

As a consequence, if $G \ne \SL_2(k)$ is a classical group and $a,b \in G$ act quadratically on the natural $kG$-module, then $\langle a, b \rangle$ is not Zariski dense. 

\subsection{Homogeneous spaces}\label{ss:hom}

For the remainder of Section \ref{s:prel2}, we will assume 
\begin{equation}\label{e:X}
X = C_1 \times \cdots \times C_r = x_1^G \times \cdots \times x_r^G
\end{equation}
where each $C_i = x_i^G$ is a noncentral conjugacy class.

Here we establish some general results concerning the action of $G$ on coset varieties $G/H$, where $H$ is a closed subgroup. Our first lemma provides a useful criterion to ensure that $G(x)$ is not generically contained in a conjugate of $H$ (the relevant condition is sufficient, but not always necessary). Recall that $X_H$ is defined in \eqref{e:XH}. 
  
\begin{lem}\label{l:basic} 
Let $H$ be a closed subgroup of $G$ and set $Y=G/H$. If
\[
\sum_{i=1}^{r} \dim Y^{x_i}  < (r-1) \dim Y
\]
then $X_H$ is contained in a proper closed subvariety of $X$.
\end{lem}   

\begin{proof}
First we recall that 
\begin{equation}\label{e:fpr}
\dim Y - \dim Y^g = \dim g^G - \dim (g^G \cap H)
\end{equation}
for all $g \in H$ (see \cite[Proposition 1.14]{LLS}). Clearly, $X_H$ is nonempty if and only if $C_i \cap H$ is nonempty for all $i$, so we may assume each $x_i$ is contained in $H$. We will work with the variety
\[
Z = \left\{ (g_1, \ldots, g_r,y) \,:\, g_i \in C_i, \, y \in \bigcap_i Y^{g_i} \right\} \subseteq X \times Y
\]
and the projection maps $\pi_1 : Z \to X$ and $\pi_2 : Z \to Y$, noting that $X_H$ coincides with the image of $\pi_1$. 

All fibers of $\pi_2$ have the same dimension, so  
\[
\dim Z = \dim Y + \sum_{i=1}^r \dim (C_i \cap H)
\]
and by applying \eqref{e:fpr} we deduce that
\[
\dim Z  = \dim Y + \sum_{i=1}^r\left(\dim C_i - \dim Y + \dim Y^{x_i}\right) < \dim X
\]
since $\sum_{i} \dim Y^{x_i}  < (r-1) \dim Y$. Therefore, $\pi_1$ is not dominant and thus $X_H$ is contained in a proper closed subvariety of $X$.
\end{proof}
  
We also record a version of \eqref{e:fpr} for subgroups. Recall that if $H$ and $L$ are closed subgroups of $G$, then $T:=T(L,H) 
  =\{g \in G \,:\, L^g \leqs H \}$ is the \emph{transporter of $L$ into $H$}. Note that $T$ is a union of cosets of $N = N_G(L)$ and $T/N$ is a variety.   
  
  \begin{lem} \label{l:fpr}   
  Let $H$ and $L$ be closed subgroups of $G$ and set $Y = G/H$. If $T=T(L,H)$ is nonempty, then
  \[
  \dim Y - \dim Y^L = \dim Y - \dim T/N  = \dim G - \dim T.
  \]
  \end{lem}  
  
  \begin{proof}   
  Let $Z = \{ (g, y) \in G \times Y \,:\, \mbox{$y$ is fixed by $L^g$}\}$. By projecting onto each factor, we see that $\dim Z = \dim Y + \dim T = \dim G +  \dim Y^L$
  and the result follows. 
  \end{proof}
  
Note that Lemma \ref{l:fpr} holds for any closed subset $L$ of $G$ (or one can replace
$L$ by the closure of the subgroup it generates). In the special case $L = \{g\}$ we have $N=C_G(g)$
and $T/N$ can be identified with $g^G \cap H$, so we recover the equation in \eqref{e:fpr}. 

We can also establish the following generalization of Lemma \ref{l:basic}, working with subgroups rather than elements. The proof is identical and we omit the details.
 
\begin{lem} \label{l:basic2}  
Let $H$ be a closed subgroup of $G$ and set $Y=G/H$. Let $L_1, \ldots, L_r$ be closed subgroups of $G$ such that  
\[
\sum_{i=1}^r \dim Y^{L_i} < (r-1) \dim Y.
\]
Then there exist $g_i \in G$ such that $\langle L_1^{g_1}, \ldots, L_r^{g_r} \rangle$ is not contained in a conjugate of $H$.
\end{lem}

We will also need the following elementary observation. 
 
\begin{lem} \label{l:products}  
Suppose $D$ is a $G$-class such that an element of $D$ is contained in the closure of 
$\langle g_1, g_2 \rangle$ for some $g_i \in x_i^G$  and  that $G(y) = G$ for some $y \in Y$, where $Y = D \times C_3 \times \cdots \times C_r$. Then $\Delta$ is nonempty.   
\end{lem} 

\begin{proof}  
Suppose $G(y) = G$ for some $y = (d, g_3, \ldots, g_r) \in Y$ with $d \in D$ and $g_i \in C_i$ for $i > 2$. If $x =(g_1, g_2, \ldots, g_r) \in X$ then $G(y) \leqs G(x)$ and the result follows.
\end{proof} 

In particular, we can take  $D$ to be a conjugacy class contained in $C_1C_2$. 
    
\subsection{Scott's Theorem and the adjoint module}\label{ss:adj}

We begin by recalling Scott's Theorem \cite[Theorem 1]{scott}, which we will then apply in the special case where $G$ is acting on its adjoint module ${\rm Lie}(G)$. Recall that if $W$ is a $kG$-module and $J \subseteq G$ is a subset, then $[J,W]$ is the subspace $\la gw - w \,:\, g \in J, w \in W \ra$ of $W$.
Note that $\dim \, [J,W] = \dim W - \dim (W^*)^J$.  

\begin{thm}[Scott]\label{t:scott}  
If $G= \langle y_1, \ldots, y_r \rangle \leqs \GL(W)$ and $y_0 = (y_1 \cdots y_r)^{-1}$, then 
\[
\sum_{i=0}^r \dim \, [y_i,W]  \geqs  \dim W - \dim W^G + \dim \, [G,W].
\]
\end{thm}

Recall that $X = C_1 \times \cdots \times C_r$, where each $C_i = x_i^G$ is a noncentral conjugacy class in $G$. Fix an additional noncentral conjugacy class $C_0 = x_0^G$ and set 
\begin{equation}\label{e:Z}
Z=\{(z_0, \ldots, z_r) \in C_0 \times \cdots \times C_r \,:\,  z_0z_1 \cdots z_r = 1\} \subseteq G^{r+1}.
\end{equation}  
For $z = (z_0, \ldots, z_r) \in Z$, let $G(z)$ be the Zariski closure of $\la z_0, \ldots, z_r \ra$. 

\begin{lem}\label{l:adjointprod}  If $L = {\rm Lie}(G)$ and $G=G(z)$ for some $z \in Z$, then 
\[
\sum_{i=0}^{r}\dim C_i \geqs  \sum_{i=0}^{r} \dim \, [x_i,L] \geqs    2\dim G - \dim Z(L). 
 \]
 \end{lem}  
 
\begin{proof}  
Since $G$ is simply connected, we have $[G,L]=L$ and $L^G = Z(L)$, the center of the Lie algebra of $G$. The second inequality now follows from Scott's Theorem, while the first holds since 
 \[
 \dim g^G =  \dim L - \dim C_G(g)  \geqs \dim L - \dim L^g = \dim \, [g, L]
 \]
for all $g \in G$.   
\end{proof}  
 
For each $y = (y_1, \ldots, y_r) \in X$, set $y_0 = (y_1 \cdots y_r)^{-1}$ so $z = (y_0, \ldots, y_r) \in Z$ with $C_0 = y_0^G$ (see \eqref{e:Z}). 
Since $\dim y_0^G \leqs \dim G - {\rm rk}\, G$, we obtain the following corollary which provides a useful, necessary condition for topological 
generation by a tuple in $X$.

\begin{cor} \label{c:adjoint}   
If $L = {\rm Lie}(G)$ and $G = G(x)$ for some $x \in X$, then 
 \[
  \dim X \geqs \sum_{i=1}^{r} \dim \, [x_i,L] \geqs  \dim G + {\rm rk}\, G - \dim Z(L).
 \] 
\end{cor}

\begin{rem}\label{r:scott}
Suppose $\dim X < \dim G + {\rm rk}\, G - \dim Z(L)$ and let $x \in X$.
\begin{itemize}\addtolength{\itemsep}{0.2\baselineskip}
\item[{\rm (a)}] By Scott's Theorem, either $L^{G(x)}$ strictly contains $Z(L)$, or $[G(x),L] \ne L$. In particular, $\Delta$ is empty.
\item[{\rm (b)}] Recall that $p$ is \emph{special} for $G$ if $p=3$ and $G=G_2$, or if $p=2$ and $G$ is of type $B_n$, $C_n$ or $F_4$. If $p$ is not special for $G$, then $L/Z(L)$ is a self-dual irreducible $kG$-module and $G(x)$ has nonzero fixed points on this module. In particular, if $L$ is a simple Lie algebra, then  $G(x)$ has nonzero fixed points on $L$.  
\end{itemize}
\end{rem}
 
The following results can also be stated in terms of the variety $Z$ in \eqref{e:Z}. But since $X$ is our main focus, we leave this to the reader. 

Recall that the prime $p=2$ is  \emph{bad} for all simple algebraic groups except type $A_n$; $p=3$ is also bad for all exceptional groups and $p=5$ is bad for $E_8$. All other primes (and also $p=0$) are \emph{good} for $G$.
 
\begin{cor} \label{c:adjreducible}   
Suppose that either $L=\mathrm{Lie}(G)$ is simple or the characteristic $p$ of $k$ is good for $G$. If $\dim X < \dim G + {\rm rk}\, G$, then $G(x)$ acts reducibly on every finite dimensional rational $kG$-module for all $x \in X$.
\end{cor}
 
\begin{proof}    
First assume that $L$ is simple and let $x \in X$. As noted in Remark \ref{r:scott}, the hypotheses and Scott's Theorem imply that there exists $0 \ne \ell \in L$ fixed by $G(x)$. Let $W$ be an irreducible $kG$-module. Untwisting by a Frobenius morphism of $G$, if necessary, we may assume that $W = W_0 \otimes W_1^{(p)}$, where $W_0$ is a nontrivial restricted $kG$-module and $\ell$ acts nontrivially on $W$. Since $G(x)$ fixes $\ell$, it preserves the eigenspaces of $\ell$ on $W$ and thus $G(x)$ acts reducibly on $W$.   
 
Finally, suppose $L$ is not simple and $p$ is good for $G$, in which case $G=\SL_n(k)$ and $p$ divides $n$. Here we can apply Scott's Theorem directly with respect to the action of $G$ on the Lie algebra $L_1$ of $\GL_n(k)$.  For $x \in X$, this gives the inequality 
\[
\sum_{i=1}^r \dim C_i  + (n^2-n) \geqs 2n^2 - 2  - (\dim L_1^{G(x)}  -1)  - (\dim (L_1^*)^{G(x)}-1)
\]
and thus  
\[
\dim X  \geqs  {\rm rk}\, G +  \dim G   -  (\dim L_1^{G(x)} -1)  - (\dim (L_1^*)^{G(x)} -1 ).     
\]                  

Since $\dim X < \dim G + {\rm rk}\, G$ and $L_1$ is self dual, it follows that $\dim L_1^{G(x)} \geqs 2$. Therefore, $G(x)$ fixes a noncentral element of $L_1$ and so it also fixes a noncentral element of $L$ (just choose an element of trace zero with the same eigenspaces). We can now conclude by repeating the argument given in the first paragraph. 
\end{proof} 
 
We present another consequence of the above observations. To do this, we need the following lemma. 
 
\begin{lem} \label{l:ssgood}  
Suppose the characteristic $p$ of $k$ is good for $G$ and let $s$ be a noncentral semisimple element of ${\rm Lie}(G)$. Then the following hold:
 \begin{itemize}\addtolength{\itemsep}{0.2\baselineskip}
\item[{\rm (i)}] $C_G(s)$  is connected.
 \item[{\rm (ii)}] $C_G(s) = C_G(S)$ for some nontrivial torus $S$ in $G$. 
 \end{itemize}
 \end{lem}
 
\begin{proof}  
Part (i) is due to Steinberg (see \cite[Theorem 3.14]{St2}). 
 
Now let us turn to part (ii). If $G=\SL_n(k)$ (or a quotient) then the result holds because we can choose a semisimple element $g \in G$ that has the same eigenspaces as $s$ on the natural module. In the remaining cases, ${\rm Lie}(G)$ is simple and we note that  $s \in \mathrm{Lie}(T)$ for some maximal torus $T$ of $G$ (see \cite[Theorem 13.3, Remark 13.4]{Hu}). Therefore, $T \leqs C_G(s)= C$ is a maximal rank connected reductive subgroup of $G$.  If $C$ is not semisimple, then $S:=Z(C)$ is a nontrivial torus and $C_G(S) = C$ as required. Now assume $C$ is semisimple. Let $D$ be a maximal connected subgroup of $G$ containing $C$ and note that $D$ is semisimple 
(since $Z(D) \leqs Z(C)$). Then $p$ is good for any simple factor of $D$ (if $G$ has type A, there are no such subgroups; if $G$ is symplectic or orthogonal, then any simple factor is classical; if $G$ is exceptional, the observation follows by inspection of the possibilities for $D$). By induction, $C = C_D(s)$ has a positive dimensional center, which is incompatible with the assumption that $C$ is semisimple. 
 \end{proof}  
 
\begin{cor} \label{c:adjreducible2}  
Suppose the characteristic of $k$ is good for $G$. If $\dim X < \dim G + {\rm rk}\, G$, 
then $\dim C_{G}(G(x))>0$ for all $x \in X$.   
\end{cor}
 
\begin{proof}  
First suppose that $G=\SL_n(k)$ and let $x \in X$. As in the proof of Corollary \ref{c:adjreducible}, we see that $G(x)$ centralizes a noncentral element $\ell$ of
the Lie algebra of $\GL_n(k)$ and so also for the Lie algebra of $G$.    Using elementary linear algebra, we see that $C_G(\ell)$ has a positive
dimensional center, whence $\dim C_{G}(G(x))>0$ for all $x \in X$. 

In the remaining cases, the Lie algebra $L = {\rm Lie}(G)$ is simple and irreducible as a $kG$-module. Let $x \in X$ and let $0 \ne \ell \in L$ be fixed by $G(x)$ (see Remark \ref{r:scott}), which we may assume is either nilpotent or semisimple. If $\ell$ is semisimple, then Lemma \ref{l:ssgood} applies, so we can assume $\ell$ is nilpotent. Here the Springer correspondence implies that $C_G(\ell) \cong C_G(u)$ for some nontrivial unipotent element $u \in G$ and we know that $\dim C_G(u)>0$ (see Seitz \cite{Seitz2}, for example). 
\end{proof}
 
\begin{rem}
We close by recording a couple of comments on the above results:
\begin{itemize}\addtolength{\itemsep}{0.2\baselineskip}
\item[{\rm (a)}] First observe that the conclusion to Lemma \ref{l:ssgood}(ii) is false (in general) if $p$ is a bad prime for $G$. For example, if $G = \Sp_n(k)$ with $p=2$ and $n \equiv 0 \imod{4}$, then there are semisimple elements $s$ in the Lie algebra of $G$ such that $C_G(s) = \Sp_m(k) \times \Sp_{n-m}(k)$ is semisimple. In particular, $C_G(s)$ is not the centralizer of a torus in this situation. Similarly, if $p=3$ then $G=G_2(k)$ has a subgroup $\SL_3(k)$ with a $1$-dimensional fixed space on the Lie algebra of $G$. 

\item[{\rm (b)}] Let us also note that Corollary \ref{c:adjreducible2} is equivalent to the statement that every $G$-orbit on $X$ has dimension strictly less than $\dim G$ (this property is stronger than stating that $\Delta$ is empty). As remarked above, the conclusion extends to tuples in $Z$ (see \eqref{e:Z}) if we assume the condition $\dim Z < 2 \dim G$. 
\end{itemize}
\end{rem}

\subsection{Unipotent classes}\label{ss:classes}

For the remainder of Section \ref{s:prel2}, unless stated otherwise, we will assume $G$ is a simple classical algebraic group over $k$ of the form ${\rm SL}(V)$, ${\rm Sp}(V)$ or ${\rm SO}(V)$, where $\dim V = n$. In the statement of Theorem \ref{t:main2}, we assume that each $x_i$ in \eqref{e:conj} has prime order modulo $Z(G)$, which implies that the corresponding elements in $G/Z(G)$ are either semisimple or unipotent (as noted in Remark \ref{r:2}, if $p=0$ then $x_i$ can be an arbitrary nontrivial unipotent element). In view of Lemma \ref{l:closures}, in order to establish the existence of a tuple in $\Delta$ we may replace $X$ by its closure in $G^r$, so we are naturally interested in the closure properties of semisimple and unipotent classes. 

The situation for semisimple classes is transparent: every such class is closed and conjugacy of semisimple elements is essentially determined by the multiset of eigenvalues on the natural module $V$ (one has to be slightly careful if $G = \SO(V)$ and $V^{x^2} =0$, in which case $n$ is even and there are two $G$-classes of semisimple elements with the same eigenvalues as $x$ on $V$, which are fused in $\OO(V) = G.2$). Our main aim in this section is to briefly recall the parametrization of unipotent classes in the classical algebraic groups, together with some of their closure properties that will be needed later. We will generally follow the notation in \cite{LS}. The results discussed below are essentially all consequences of Spaltenstein \cite{Sp}. 

\subsubsection{Linear groups} 
First recall that the conjugacy classes of unipotent elements in $G=\SL(V)$ are in bijection with partitions of $n$. Write $C(\pi)$ for the conjugacy class in $\SL(V)$ corresponding to the partition $\pi$ and note that if $p>0$ then the elements in $C(\pi)$ have order $p$ if and only if each part of $\pi$ is at most $p$. If $\pi_1$ and $\pi_2$ are partitions of $n$, then $C(\pi_2)$ is in the closure of $C(\pi_1)$ if and only if $\pi_1$ dominates $\pi_2$ in the usual partial ordering on the set of partitions of $n$ (see \cite{He1, Sp}, for example). Let $d(\pi)$ denote the number of parts in the partition $\pi$ and let $U(m)$ be the subvariety of $G$ consisting of all unipotent elements with an $m$-dimensional fixed space on $V$ (in other words, $U(m)$ is the union of the unipotent classes $C(\pi)$ with $d(\pi) = m$). It follows from the above discussion that $U(m)$ is irreducible and $C(\pi)$ is open in $U(m)$,  where $\pi$ is the partition $(n-m+1, 1^{m-1})$. Moreover, there is a unique partition $\pi'$ of $n$ such that $d(\pi')=m$ and $C(\pi')$ is contained in the closure of every conjugacy class contained in $U(m)$. This partition has at most two distinct part sizes (and if there are two, say $a$ and $b$ with $a>b$, then $a-b=1$). 

\subsubsection{Symplectic groups with $p \ne 2$}
Next assume $G = \Sp(V)$ and $p \ne 2$. Let $\pi$ be a partition of $n$ and write $C(\pi)$ for the corresponding class in $\SL(V)$ as above. Then $C(\pi) \cap G$ is nonempty if and only if the multiplicity of every odd part of $\pi$ has even multiplicity; if this condition holds, then $C_G(\pi) := C(\pi) \cap G$ is a conjugacy class of $G$. Moreover, the closure relation is the same as for $\SL(V)$ (for the admissible partitions). Set $U_G(m) = U(m) \cap G$ and note that $U_G(m)$ is irreducible and contained in the closure of the class $C_G(\pi)$, where 
\[
\pi = \left\{\begin{array}{ll}
(n-m+1, 1^{m-1}) & \mbox{if $m$ is odd} \\
(n-m, 2, 1^{m-2}) & \mbox{if $m$ is even.} 
\end{array}\right.
\]
As noted for $\SL(V)$, there is a unique unipotent class contained in the closure of any class in $U_G(m)$ (this is the same class as described for $\SL(V)$
and has the smallest dimension of any class in $U_G(m)$). Also note that if $m$ is even, then this class contains elements in a Levi subgroup of $G$, namely the stabilizer in $G$ of a pair of complementary totally isotropic subspaces of dimension $n/2$. Of course, if $m$ is odd, then no elements in $U_G(m)$ are contained in such a Levi subgroup.   

\subsubsection{Orthogonal groups with $p \ne 2$}\label{sss:orthpn2}
Next assume $G = \SO(V)$ and $p \ne 2$, with $n \geqs 5$. Set $C_G(\pi)= C(\pi) \cap G$, which is nonempty if and only if all even parts in $\pi$ occur with even multiplicity. In addition, we note that $C_G(\pi)$ is a single conjugacy class in the full orthogonal group $G.2=\OO(V)$, while  $C_G(\pi)$ splits into two $G$-classes if and only if all parts are even. As for $\Sp(V)$, the closure relation is the same as above, restricted to the admissible classes (for classes that split, the smallest classes in the respective closures are precisely the same). If $n$ is even then $d(\pi)$ is even for every partition $\pi$ corresponding to a class in $G$. Once again, there is a unique unipotent class contained in the closure of any class in $U_G(m) = U(m) \cap G$ and it has the smallest dimension of any class in $U_G(m)$. This class also contains elements in a Levi subgroup of $G$, which is the stabilizer of a pair of complementary totally singular subspaces of dimension $n/2$. Note that if $n/2$ is odd, then this Levi is unique up to conjugacy in $G$, whereas there are two $G$-classes of such Levi subgroups when $n/2$ is even, which are fused under the action of an involutory graph automorphism of $G$ (i.e. the two $G$-classes are fused in $G.2 = \OO(V)$).

\subsubsection{Symplectic groups with $p=2$}\label{sss:sppn2}
Now suppose $G = \Sp(V)$ and $p =2$. If $g \in G$ is unipotent, then we can write 
$V$ as an orthogonal direct sum of indecomposable $k\la g \ra$-modules (in the sense that a module is indecomposable if it cannot be decomposed as an orthogonal sum of two proper submodules).  The indecomposable summands that arise are labeled as follows in \cite[Lemma 6.2]{LS}:  
\begin{itemize}\addtolength{\itemsep}{0.2\baselineskip}
\item[{\rm (a)}] $V(2m)$, where $g$ acts as a single Jordan block of size $2m$; and
\item[{\rm (b)}] $W(\ell)$, where $g$ has two Jordan blocks of size
$\ell$, each corresponding to a submodule that is a totally isotropic space.
\end{itemize}
Then every unipotent element $g \in G$ yields an orthogonal direct sum decomposition
 of the form  
 \begin{equation}\label{e:decc}
 V = \sum_i  V(2m_i)^{a_i}  \perp  \sum_j W(\ell_j)^{b_j}
 \end{equation}
 with $0 \leqs a_i \leqs 2$ for each $i$, which is unique up to isomorphism. Spaltenstein \cite{Sp}  completely describes the closure relations, but here we only record what we need:
\begin{itemize}\addtolength{\itemsep}{0.2\baselineskip}
\item[{\rm (i)}] If $m_1 > m_2$, the closure of $V(2m_1) \perp V(2m_2)$ contains $V(2m_1 - 2) \perp V(2m_2 + 2)$. 
\item[{\rm (ii)}]  If $m_1 \geqs m_2$, the closure of  $V(2m_1) \perp V(2m_2)$ contains $W(m_1 + m_2)$.
\item[{\rm (iii)}]  We have $a_i=0$ for all $i$ if and only if $g$ is conjugate to an element in a 
Levi subgroup of $G$ arising as the stabilizer of a pair of complementary totally isotropic subspaces of $V$. For such an element $g$, the multiplicity of every part in the corresponding partition of $n$ is even. 
\item[{\rm (iv)}] The closure relation for unipotent elements with $a_i=0$ for all $i$ coincides with the usual ordering on partitions.   
\item[{\rm (v)}]  If $m$ is even, then there is a unique smallest class in $U_G(m) = U(m) \cap G$ and this class corresponds to a partition with at most two distinct sizes (if there are two, say $a>b$, then $a-b=1$).  
\end{itemize}

\subsubsection{Orthogonal groups with $p=2$}
Finally, let us assume $G= \SO(V)$ with $p=2$, where $n \geqs 6$ is even. Here it is  convenient to view $G$ as a subgroup of $J =  \Sp(V)$ and we observe that the description of the unipotent conjugacy classes in $G$ is (essentially) the same as for $J$. If $g \in J$ is unipotent, then $g$ is conjugate to an element of $G.2 = \OO(V)$ and two unipotent elements in $G.2$ are conjugate in $G.2$ if and only if they are conjugate in $J$. So we can use the same notation for the unipotent elements in $G.2$ corresponding to the decomposition in \eqref{e:decc}. Note that such an element $g \in G.2$ is contained in $G$ if and only if $\sum_ia_i$ is even (which is equivalent to the condition that $g$ has an even number of Jordan blocks on $V$). In addition, if $g \in G$ then the class $g^{G.2}$ splits into two $G$-classes if and only if $a_i=0$ and $\ell_j$ is even for all $i,j$ in \eqref{e:decc} (see \cite[Proposition 6.22]{LS}). The closure properties in $G$ are also inherited from $J$. In particular, if $m$ is even then the smallest unipotent class $g^G$ with $m$ Jordan blocks corresponds to the smallest class in a Levi subgroup $\GL(W)$ with $m/2$ Jordan blocks, where $W$ is a maximal totally singular subspace of $V$ (hence each Jordan block of $g$ on $V$ has even multiplicity).    

\begin{rem}\label{r:as}
As previously noted, if $G = \Sp(V)$ or $\SO(V)$ with $p=2$, then we will mainly be interested in the unipotent involutions in $G$. The conjugacy classes of unipotent involutions in simple algebraic groups (and the corresponding finite groups of Lie type) were studied in detail by Aschbacher and Seitz \cite{AS} and here we recall their notation. 

Let $g \in G$ be a unipotent involution with Jordan form $(J_2^s,J_1^{n-2s})$ on $V$, where $J_i$ denotes a standard unipotent Jordan block of size $i$. If $s$ is even, then $\Sp(V)$ and $\OO(V)$ both have two classes of such elements, with representatives denoted by $a_s$ and $c_s$ (here $g$ is of type $a_s$ if and only if $(v,gv) = 0$ for all $v \in V$, where $( \, , \,)$ is the corresponding alternating or symmetric form on $V$). On the other hand, if $s$ is odd then there is a unique class of such elements in $\Sp(V)$ and $\OO(V)$, represented by $b_s$ (for orthogonal groups, these elements are contained in $\OO(V) \setminus \SO(V)$). We also note that if $g \in \SO(V)$ is a unipotent involution, then $g^{\OO(V)} = g^{\SO(V)}$ unless $n \equiv 0 \imod{4}$ and $g$ is $\OO(V)$-conjugate to $a_{n/2}$, in which case the $\OO(V)$-class splits into two $\SO(V)$-classes. In view of the notation in \cite{AS}, we will refer to \emph{$x$-type} involutions in $G$, where $x$ is either $a$, $b$ or $c$. 

The correspondence between this notation and the decomposition in \eqref{e:decc} is as follows:
\begin{align*}
\mbox{$a_s$, $s$ even, $2 \leqs s \leqs n/2$:} & \;\; W(2)^{s/2} \perp W(1)^{n/2-s} \\
\mbox{$b_s$, $s$ odd, $1 \leqs s \leqs n/2$:} & \;\; V(2) \perp W(2)^{(s-1)/2} \perp W(1)^{n/2-s} \\
\mbox{$c_s$, $s$ even, $2 \leqs s \leqs n/2$:} & \;\;  V(2)^2 \perp W(2)^{s/2-1} \perp W(1)^{n/2-s}
\end{align*}
\end{rem}

\subsection{Tensor products}\label{ss:tensorproduct} 

In the proof of Theorem \ref{t:main2} we will need to consider the action of unipotent elements on tensor products and related spaces. As above, we write $J_i$ for a standard unipotent Jordan block of size $i$.

Let $J_a \in {\rm GL}_a(k) = \GL(W)$ be a regular unipotent element and let $J_a \otimes J_a$, $\wedge^2(J_a)$ and $\Sym^2(J_a)$ denote the action of $J_a$ on the tensor product $W \otimes W$, the exterior square $\wedge^2(W)$ and the symmetric square $\Sym^2(W)$, respectively. Similarly, we define $J_a \otimes J_b$. There are results in the literature giving the precise Jordan decomposition of these operators (see \cite{Barry}, for example), but we are only interested here in the number of Jordan blocks on the respective spaces. As explained below, this number is independent of the characteristic $p$, with the exception of the module $\Sym^2(W)$ for $p=2$.  

\begin{lem}\label{l:tensor} 
Let $a,b \geqs 2$ be integers.
\begin{itemize}\addtolength{\itemsep}{0.2\baselineskip}
\item[{\rm (i)}] $J_a \otimes J_b$ has $\min\{a,b\}$ Jordan blocks.
\item[{\rm (ii)}] $\wedge^2(J_a)$ has $\lfloor a/2 \rfloor$ Jordan blocks.
\item[{\rm (iii)}] $\Sym^2(J_a)$ has $\lceil a/2\rceil+\e$ Jordan blocks, where $\e=1$ if $a$ is even and $p=2$, otherwise $\e=0$.    
\end{itemize}
\end{lem} 

\begin{proof}  
All these results hold in characteristic $0$ by considering appropriate modules for $\SL_2(k)$ (see \cite[Section 6]{GG2}). Since the relevant operators are defined over $\mathbb{Z}$, it follows that the results in characteristic $0$ give lower bounds in the positive characteristic setting.

For the remainder, let us assume $p>0$. First consider (i) and assume $a \geqs b$. Then 
$J_a \otimes J_b$ has no more Jordan blocks than $J_a \otimes I_b$, which visibly has Jordan form $(J_a^b)$. Therefore (i) holds. Similarly, if $p \ne 2$ then 
\[
J_a \otimes J_a  = \wedge^2(J_a) \oplus \Sym^2(J_a)
\]
and thus (ii) and (iii) follow by combining part (i) with the result in characteristic $0$. For the remainder, we may assume $p=2$.

Consider (ii) and view $g= J_a \in \GL_a(k) < \SO_{2a}(k)=H$, where $\GL_a(k)$ is the stabilizer of a pair of complementary totally singular $a$-dimensional subspaces of the natural module for $H$ (so in particular, $g$ has Jordan form $(J_a^2)$ on this space). Then $\dim C_H(g) =   a  + 2c$, where $c$ is the number of Jordan blocks of 
$\wedge^2(J_a)$. By \cite[Chapter 4]{LS} or \cite{He},  the dimension of $C_H(g)$ is independent of the characteristic and the result follows. 

Finally, consider (iii) with $p=2$. Here $H < \Sp_{2a}(k) = K$ and $\dim C_K(g) = a + 2c'$, where $c'$ is the number of Jordan blocks of $\Sym^2(J_a)$. By
Lemma \ref{l:oinsp} below, we have $\dim C_K(g) = \dim C_H(g) + 2$, so $c'=c+1$ and thus (iii) follows from (ii).   
\end{proof} 

\subsection{Exterior squares}\label{ss:linalg}

Here we study the action of $G = \GL_n(k) = \GL(V)$ on $W = \wedge^2(V)$, where $n \geqs 2$ and $k$ is an algebraically closed field of characteristic $p \geqs 0$. Let $g \in G$ and recall that $\dim W = \binom{n}{2}$ and $W^g$ denotes the fixed space of $g$ on $W$. Let $\mathcal{E}(g)$ be the set of eigenvalues of $g$ on $V$ and let $\a(g)$ be the dimension of the largest eigenspace of $g$ on $V$. We will establish some useful bounds on $\dim W^g$ in terms of $\a(g)$. 

\begin{lem}\label{l:extbd}
Let $g \in G$ be a noncentral semisimple or unipotent element, and assume $g$ is an involution if $p=2$ and $g$ is unipotent. If $d=\a(g)$ then the following hold:
\begin{itemize}\addtolength{\itemsep}{0.2\baselineskip}
\item[{\rm (i)}] $\dim W^g \leqs d\lfloor n/2 \rfloor$.
\item[{\rm (ii)}] If $p \ne 2$, $g$ is semisimple and $\{\pm 1\} \subseteq \mathcal{E}(g)$, then $\dim W^g <  d(n-1)/2$.
\end{itemize}
\end{lem} 

\begin{proof}   
 First assume $g$ is semisimple and write $n = 2de + f$ with $0 \leqs f < 2d$.   It is straightforward to see that if $f=0$, then $\dim W^g$ is maximal when $g$ has $e$ pairs of distinct eigenvalues $\{\l,\l^{-1}\}$, each with multiplicity $d$.  In this case, $\dim W^g  = ed^2 = dn/2$. 

Now assume $f >0$. Here  the maximum still occurs when $g$ has $e$ pairs of distinct eigenvalues $\{\l,\l^{-1}\}$ with multiplicity $d$ and so we may write $g = g_1 \oplus g_2$ with respect to the decomposition $V=V_1 \oplus V_2$, where $\dim V_1 = 2de$, $\dim V_2=f$ and $g_1$ has the eigenvalues on $V_1$ as described above. Then 
\[
\dim W^g   = ed^2  + \dim \wedge^2(V_2)^{g_2}.
\]
If $0 < f   \leqs  d$, then we can assume $g_2$ is trivial and thus  
\[
\dim W^g   = ed^2  + \frac{1}{2}f(f-1)  \leqs  \frac{1}{2}d(2ed + f-1)  =  \frac{1}{2}d(n-1), 
\]
with equality only if $f=d$. On the other hand, if $f > d$ then we may assume $g_2$ has exactly two eigenvalues on $V_2$ and it is straightforward to show that $\dim W^g  < d(n-1)/2$.

To complete the proof for semisimple elements, let us assume $\{\pm 1\} \subseteq \mathcal{E}(g)$ and write $g=g_1 \oplus g_2$ with respect to the decomposition $V=U_1 \oplus U_2$, where $U_2$ is the kernel of $g^2-1$. Then $W^g = \wedge^2(U_1)^{g_1} \oplus \wedge^2(U_2)^{g_2}$ and the argument above gives 
\[
\dim \wedge^2(U_1)^{g_1} \leqs \frac{1}{2}d(n-l-1)
\]
with $l = \dim U_2$. In addition, if the two eigenspaces of $g_2$ on $U_2$ have dimensions $m$ and $l-m$, then 
\[
\dim \wedge^2(U_2)^{g_2} = \binom{m}{2}+\binom{l-m}{2} \leqs \frac{1}{2}d(l-2)
\]
and the result follows.

Finally, let us assume $g$ is unipotent. In view of Lemma \ref{l:complete} (also see Remark \ref{r:linear}), we may replace $g$ by any unipotent element in the closure of $g^G$ with the same number of Jordan blocks on $V$.

First assume $p \ne 2$. By the discussion in Section \ref{ss:classes}, we may assume $g$ has Jordan form $(J_a^{e},J_{a-1}^{d-e})$ for some $a \geqs 2$. 
By Lemma \ref{l:tensor} we see that $\wedge^2(J_m)$ and $J_b \otimes J_c$ (with $c \leqs b$) have $\lfloor m/2 \rfloor$ and $c$ Jordan blocks, respectively. This makes it easy to compute the number of Jordan blocks of $g$ on $W$ and the result follows. The case $p=2$ (with $g$ an involution) is entirely similar. 
\end{proof} 

We need to consider the case where $n$ is odd in a bit more detail (for example, see the proof of Theorem \ref{t:so2nodd}, which establishes Theorem \ref{t:main2} for orthogonal groups of the form $\SO_{2m}(k)$ with $m \geqs 5$ odd). There is a similar result for $n$ even, but the analysis is more complicated and we do not need it in this paper. For $n$ odd, we first observe that the proof of Lemma \ref{l:extbd} gives the following corollary.

\begin{cor} \label{c:nodd}  
Suppose $n = 2m+1$, $m \geqs 1$ and $g \in G$ has prime order modulo $\langle -I_n \rangle$. Set $d = \a(g)$. Then $\dim W^g \leqs dm$, with equality only if $d \leqs m+1$. In addition, if both bounds are attained then $g$ is unipotent. 
\end{cor} 

\begin{cor} \label{c:wedge2odd}   
Suppose $n = 2m+1$ and $m,r \geqs 2$. Let $g_1, \ldots, g_r$ be elements in $G$ of prime order modulo $\langle -I_n \rangle$ and set $d_i = \a(g_i)$ and $e_i = \dim W^{g_i}$. If $\sum_i d_i \leqs n(r-1)$, then one of the following holds (up to ordering and conjugacy):
\begin{itemize}\addtolength{\itemsep}{0.2\baselineskip}
\item[{\rm (i)}] $\sum_i e_i < (r-1)\dim W$. 
\item[{\rm (ii)}] $r=2$, $g_1 = (J_2^m,J_1)$ and either $g_2 = (\l I_{m}, \l^{-1}I_m,\mu I_1)$ with $\l \in k^{\times} \setminus \{\pm 1\}$ and $\mu \in k^{\times}\setminus \{\l^{\pm}\}$, or $p \ne 2$ and $g_2 = (J_3,J_2^{m-1})$.
\end{itemize} 
\end{cor}

\begin{proof}  
By Corollary \ref{c:nodd}, we have $\sum_i e_i  \leqs (r-1)\dim W$, with equality only if $\sum_i d_i = n(r-1)$ and $d_i \leqs m+1$ for all $i$.  Therefore, we may assume these conditions are satisfied, which immediately implies that $r=2$ (since $n \geqs 5$). Up to reordering, it follows that 
\[
d_1 = m+1, \; d_2 = m, \; e_1 = d_1m, \; e_2 = d_2m
\] 
and thus $g_1$ is unipotent by Corollary \ref{c:nodd}. 

If $p=2$ then $g_1$ is an involution and the condition $d_1=m+1$ forces it to have Jordan form $(J_2^m,J_1)$ as required. If $p \ne 2$, then the closure of any unipotent class in $G$ with $m+1$ Jordan blocks on $V$ contains the class of elements with Jordan form $(J_2^m,J_1)$. The next smallest class of unipotent elements with $m+1$ Jordan blocks contains elements with Jordan form $(J_3,J_2^{m-2},J_1^2)$. But it is straightforward to check that $e_1<d_1m$ if $g_1$ has this form, so this is not possible.    

Similarly, we find that if $d_2=m$ and $e_2 = d_2m$, then $g_2$ has the form described in (ii).  
\end{proof} 

\subsection{Subspace stabilizers}\label{ss:subspaces} 

In this final preliminary section we assume $G$ is one of the classical groups $\SL_n(k)$ ($n \geqs 2$), $\Sp_n(k)$ ($n \geqs 4$) or $\SO_n(k)$ (with $n \geqs 3$, $n \ne 4$). Recall that we may assume $p \ne 2$ if $G = \SO_n(k)$ and $n$ is odd. As before, let $V$ be the natural $kG$-module and set 
\[
X = C_1 \times \cdots \times C_r = x_1^G \times \cdots \times x_r^G
\]
as usual, where $r \geqs 2$ and each $x_i$ has prime order modulo $Z(G)$ (see Remark \ref{r:2}). For $g \in G$, let $\a(g)$ be the dimension of the largest eigenspace of $g$ on $V$ and set $d_i = \a(x_i)$ for $i=1, \ldots, r$. We define $\Delta$ and $X_H$ as in \eqref{e:delta} and \eqref{e:XH}, respectively, where $H$ is a closed subgroup of $G$.

As noted in Section \ref{s:intro}, if $\sum_id_i>n(r-1)$ then $G(x)$ acts reducibly on $V$ for all $x \in X$ and thus $\Delta$ is empty. The following result shows that if $\sum_id_i \leqs n(r-1)$ (or if $r \geqs 3$), then $G(x)$ is generically positive dimensional. In particular, in order to prove Theorem \ref{t:main2} we can ignore any tuples $x \in X$ such that $G(x)$ is finite.

\begin{lem}\label{l:infinite}  
Suppose $r \geqs 3$ or $d_1+d_2 \leqs n$. Then $G(x)$ is generically infinite.
\end{lem}

\begin{proof}  
If $r \geqs 3$ then the main theorem of \cite{GMT} implies that $C_1C_2C_3$ contains elements of arbitrarily large order (and indeed elements of infinite order if $k$ is not algebraic over a finite field). In addition, the same conclusion holds for $C_1C_2$, aside from a short list of classes $C_1$ and $C_2$ given in \cite[Theorem 1.1]{GMT} and in each of these cases one can check that $d_1 + d_2 > n$. Therefore, the closed subvariety $X_m :=\{ x \in X \,:\, |G(x)| \leqs m \}$ is proper for all $m \in \mathbb{N}$ and thus $\{x \in X \,:\, \mbox{$G(x)$ is infinite}\}$ contains the complement of a countable union of closed subvarieties and is therefore generic (and nonempty unless $k$ is algebraic over a finite field).
\end{proof} 

For the remainder of Section \ref{ss:subspaces}, we are mainly interested in the action of subgroups of $G$ of the form $G(x)$ on  varieties of appropriate $m$-dimensional subspaces of $V$ with $m = 1$ or $2$. Our first result on $1$-spaces can be viewed as an extension of  \cite[Lemma 2.15]{Ger}. 

\begin{lem} \label{l:1-spaces} 
Let $H$ be the stabilizer in $G$ of a $1$-dimensional subspace of $V$, which is either 
nondegenerate (if $p \ne 2$) or nonsingular (if $p=2$) when $G = \SO(V)$. If $\sum_i d_i \leqs n(r-1)$, then $X_H$ is contained in a proper closed subvariety of $X$. 
\end{lem} 

\begin{proof} 
Set $Y=G/H$ and observe that $\dim Y =  n -1$ and $\dim Y^g  \leqs d-1$ for all $g \in G$, where $d = \a(g)$.  Since $\sum_i d_i \leqs n(r-1)$, we deduce that 
$\sum_i (d_i-1) <  (r-1) \dim Y$ and the result follows via Lemma \ref{l:basic}.
\end{proof} 

\begin{rem} \label{r:linearalgebra} 
Let $x = (g_1, \ldots, g_r) \in X$ and let $U_i$ be a $d_i$-dimensional eigenspace of $g_i$ on $V$. Notice that if $\sum_i d_i > n(r-1)$ then $\bigcap_{i}U_i$ is nonzero
and thus $G(x)$ fixes a $1$-dimensional subspace of $V$. In particular, if $G=\SL(V)$ or $\Sp(V)$ then $X_H = X$ and the converse to Lemma \ref{l:1-spaces} holds. However, if $G = \SO(V)$ then $G(x)$ may fix a totally singular $1$-space and we cannot conclude that $X_H$ is dense in $X$.
\end{rem}

The next result handles the action of orthogonal groups on totally singular $1$-spaces. 

\begin{lem} \label{l:sing1}   
Let $G=\SO(V)$ and let $H$ be the stabilizer of a $1$-dimensional totally singular subspace of $V$. 
If $\sum_i d_i \leqs n(r-1)$, then either
\begin{itemize}\addtolength{\itemsep}{0.2\baselineskip}
\item[{\rm (i)}] $X_H$ is a proper closed subvariety of $X$; or
\item[{\rm (ii)}] $r=2$ and $x_1,x_2$ are quadratic. 
\end{itemize} 
\end{lem}  

\begin{proof}  
Set $Y=G/H$ and note that $X_H$ is closed by Lemma \ref{l:parabolic}, so we only need to show that $X_H$ is proper (unless $r=2$ and the $x_i$ are quadratic). We will apply Lemma \ref{l:basic} to do this, which means that we need to estimate $\dim Y^{x_i}$. Note that $\dim Y = n-2$.

First assume $x_i$ is unipotent and $p \ne 2$, so $d_i$ is equal to the number of Jordan blocks of $x_i$ on $V$. By replacing $C_i$ by a unipotent class $y_i^G$ in its closure with $\a(y_i) = \a(x_i) = d_i$, we may assume that every Jordan block of 
$x_i$ on $V$ has size $\ell$ or $\ell-1$ for some $\ell \geqs 2$ (with at least one Jordan block of size $\ell$); see Section \ref{sss:orthpn2}. If $\ell \geqs 3$, then there are no Jordan blocks of size $1$ and so the fixed space $V^{x_i}$ is totally singular and $\dim Y^{x_i} = d_i -1$. The same conclusion holds if $\ell=2$ and there are no Jordan blocks of size $1$. Finally, suppose $\ell=2$ and $x_i$ has a Jordan block 
of size $1$. Here $x_i$ fixes a nondegenerate $1$-space and we claim that $\dim Y^{x_i} = d_i -2$. To see this, first observe that $Y^{x_i}$ is precisely the subvariety of totally singular $1$-dimensional subspaces of $V^{x_i}$. Let $\mathbb{P}^1(V^{x_i})$ be the variety of $1$-dimensional subspaces of $V^{x_i}$, so $\dim \mathbb{P}^1(V^{x_i}) = d_i-1$. The nondegenerate $1$-spaces in $\mathbb{P}^1(V^{x_i})$ form a nonempty open subset and so the variety of totally singular $1$-spaces in $V^{x_i}$ has codimension
$1$ in $\mathbb{P}^1(V^{x_i})$. This justifies the claim. 

Now assume $x_i$ is semisimple and $p \ne 2$. If $x_i$ has a totally singular eigenspace of dimension $d_i$, then $\dim Y^{x_i} = d_i -1$. If not, then a $d_i$-dimensional eigenspace $W$ of $x_i$ on $V$ is nondegenerate (and corresponds to an eigenvalue $\pm 1)$; the largest irreducible component of $Y^{x_i}$ corresponds to the subvariety of totally singular $1$-spaces in $W$, which has dimension $d_i-2$.  

For both unipotent and semisimple $x_i$ (with $p \ne 2$), we observe that $d_i \leqs n/2$ whenever $\dim Y^{x_i} = d_i - 1$. First assume   
$\dim Y^{x_i} = d_i - 1$ for $i=1,2$. Then 
\[
\sum_{i=1}^{r} \dim Y^{x_i}  \leqs n - 2  + \sum_{i=3}^r \dim Y^{x_i},
\]
which is less than $(r-1) \dim Y = (r-1)(n-2)$ unless $r=2$ and each $x_i$ is quadratic. If we exclude the latter situation, the desired result follows from Lemma \ref{l:basic}. Similarly, if we assume 
$\dim Y^{x_i} = d_i - 1$ for only one $i$, then 
\[
\sum_{i=1}^r \dim Y^{x_i} =  1 + \sum_{i=1}^r (d_i-2) < (r-1)(n-2)
\]
and again the result holds. Finally, if $\dim Y^{x_i} < d_i - 1$ for all $i$, then 
\[
\sum_{i=1}^r \dim Y^{x_i} = \sum_{i=1}^r (d_i - 2)  \leqs (r-1)n - 2r < (r-1)(n-2)
\]
and once again the result follows.

To complete the proof, let us assume $p=2$. If $x_i$ is unipotent and some Jordan block has size $1$, then a generic fixed vector is nonsingular and the above argument goes through. Similarly, for semisimple classes we can repeat the argument given above. 
\end{proof}

For $2$-spaces we will need the following result.

\begin{lem} \label{l:2vs1}   
Let $G=\SL(V)$ with $n \geqs 3$ and let $H$ be the stabilizer of a $2$-dimensional subspace of $V$. If $\sum_i d_i \leqs n(r-1)$, then either
\begin{itemize}\addtolength{\itemsep}{0.2\baselineskip}
\item[{\rm (i)}]  $X_H$ is contained in a proper closed subvariety of $X$; or  
\item[{\rm (ii)}] $r=2$ and $x_1,x_2$ are quadratic. 
\end{itemize} 
\end{lem} 

\begin{proof}    
Set $Y = G/H$ and note that $\dim Y = 2(n-2)$. 
By \cite[Lemma 3.35 ]{GL}, we may assume that each $x_i$ is semisimple.  
Let $g \in G$ be a noncentral semisimple element with $\a(g) = d$ and let $d'$ be the dimension of the second largest eigenspace of $g$ on $V$. Then one of the following holds:
\begin{itemize}\addtolength{\itemsep}{0.2\baselineskip}
\item[{\rm (a)}] $d'=d$ and $\dim Y^g = 2(d-1)$.
\item[{\rm (b)}] $d'=d-1$ and $\dim Y^g = 2d-3$.
\item[{\rm (c)}] $d' \leqs d -2$ and $\dim Y^g = 2(d-2)$.
\end{itemize}

Let $d_i'$ be the dimension of the second largest eigenspace of $x_i$ on $V$. If $d_i' \leqs d_i -2$ for all but at most two $i$, then using the bound $\sum_i d_i \leqs n(r-1)$ we deduce that 
\[
\sum_{i=1}^r \dim Y^{x_i}  \leqs 4 +  \sum_{i=1}^r (2d_i - 4)  \leqs 2(r-1)(n-2) = (r-1) \dim Y.
\]
Moreover, we see that equality holds if and only if $r=2$ and $d_i = d_i'$ for $i=1,2$, in which case $n$ is even and $x_1,x_2$ are quadratic.

Now assume that $d_i ' \geqs d_i -1$ for $i=1,2,3$, so $\sum_{i=1}^3 \dim Y^{x_i} \leqs 2\sum_{i=1}^3d_i-6$. If $n$ is even, then $d_i \leqs n/2$ for $i=1,2,3$ and thus 
\[
\sum_{i=1}^r\dim Y^{x_i} \leqs 2\sum_{i=1}^rd_i - 2r \leqs 3n+2(r-3)(n-1)-2r < (r-1)\dim Y
\]
since $n \geqs 4$. Similarly, if $n$ is odd and $i \in \{1,2,3\}$ then either $d_i = (n+1)/2$ and $\dim Y^{x_i} = 2d_i - 3$, or $d_i  \leqs (n-1)/2$ and $\dim Y^{x_i} \leqs 2d_i -2$. By arguing as above, we deduce that  $\sum_{i} \dim Y^{x_i}  <  (r-1) \dim Y$ and the result follows.
\end{proof} 

By applying the previous lemma, we obtain the following result concerning the action of 
symplectic and orthogonal groups on nondegenerate $2$-spaces.  

\begin{lem} \label{l:nondeg2}   
Let $G=\Sp(V)$ or $\SO(V)$ with $n \geqs 3$, and let $H$ be the stabilizer of a $2$-dimensional nondegenerate   subspace of $V$. If $\sum_i d_i \leqs n(r-1)$, then either
\begin{itemize}\addtolength{\itemsep}{0.2\baselineskip}
\item[{\rm (i)}] $X_H$ is contained in a proper closed subvariety of $X$; or
\item[{\rm (ii)}] $r=2$ and $x_1,x_2$ are quadratic. 
\end{itemize} 
\end{lem}  

\begin{proof}  
First observe that $G$ acts transitively on $Y = G/H$, which is a dense
open subset of the variety $Z$ of all $2$-dimensional subspaces of $V$ and thus  
$\dim Y = \dim Z$. Since $\dim Y^g \leqs \dim Z^g$ for all $g \in G$, the proof of Lemma \ref{l:2vs1} implies that $\sum_i \dim Y^{x_i} < (r-1) \dim Y$ unless $r=2$ and $x_1,x_2$ are quadratic. Now apply Lemma \ref{l:basic}.
\end{proof} 

To close this section, we present Lemma \ref{l:oinsp} below on the action of $\Sp(V)$ on the homogeneous space $Y = \Sp(V)/\OO(V)$ when $p=2$ (as explained in the proof of the lemma, this can be viewed as a subspace action by identifying $\Sp_n(k)$ with the orthogonal group ${\rm O}_{n+1}(k)$). This justifies the comments in Remark \ref{r:orth} and it explains the extra condition in Theorem \ref{t:main2} when $G = \Sp(V)$ and $p=2$.  

In order to prove the lemma, we need the following well known fact about symplectic and orthogonal groups in characteristic $2$.   

\begin{lem} \label{l:serre}   
Let $G=\Sp(V)$ and $H = \OO(V)$, where $n \geqs 4$ and $p=2$. Then every element of $G$ is conjugate to an element of $H$.
\end{lem}

\begin{proof}  
Let $x \in G$ and write $x=su=us$, where $s$ is semisimple and $u$  is unipotent.  If $s$ has $3$ or more distinct eigenvalues on $V$, then $x$ preserves an orthogonal decomposition $V = V_1 \perp V_2$, where each $V_i$ is a nondegenerate subspace (with respect to the defining symplectic form on $V$), and the result follows by induction. If $s$ has exactly $2$ distinct eigenvalues, then $C_G(s)$ is the stabilizer of a pair of complementary totally isotropic spaces and this subgroup embeds in some conjugate of $H$. So we may assume that $s=1$ and $x$ is unipotent. We can argue as above if $x$ commutes with a nontrivial semisimple element, so we may assume $x$ is a distinguished unipotent element. As before, if $x$ preserves an orthogonal decomposition, then the result follows by induction. The only distinguished unipotent elements in $G$ that act indecomposably on $V$ are regular, which act on $V$ with a single Jordan block (see \cite[Chapter 6]{LS}). In this case, one can write down such an element in $H$, or one can appeal to the classification of conjugacy classes of unipotent elements in $H$, as described in \cite{LS}. 
\end{proof} 

\begin{lem}\label{l:oinsp}  
Let $G=\Sp(V)$ and $H = \OO(V)$, where $n \geqs 4$ and $p=2$. Define $x_i \in G$ as in \eqref{e:X} and set $e_i  = \dim V^{x_i}$. 
\begin{itemize}\addtolength{\itemsep}{0.2\baselineskip}
\item[{\rm (i)}] If $\sum_i e_i < n(r-1)$, then $X_H$ is contained in a proper closed subvariety of $X$.  
\item[{\rm (ii)}]  If $\sum_i e_i  \geqs  n(r-1)$, then $\Delta$ is empty. 
\end{itemize}
\end{lem}

\begin{proof} 
Set $Y=G/H$ (so $\dim Y = n$) and let $W$ be an indecomposable rational $kG$-module of 
dimension $n+1$ with socle $V$. Then $G$ acts transitively on the variety of $1$-dimensional subspaces of $W$ that are not contained in $V$ and the stabilizers
are just orthgogonal groups.  Thus we may identify this variety with $Y$. By Lemma \ref{l:serre}, each $g \in G$ fixes a complement to $V$ in $W$ and so we can identify the corresponding fixed spaces  $Y^g$ and $V^g$. In particular, $\dim Y^g  = \dim V^g$ and the bound in (i) implies that $\sum_i \dim Y^{x_i} < (r-1) \dim Y$. Now apply Lemma \ref{l:basic}.    

Now let us turn to (ii). Since each $g \in G$ fixes a complement to $V$ in $W$, it follows that $\dim W^g = \dim V^g + 1$.  Therefore, $\dim W^{x_i} = e_i +1$ and so the inequality in (ii) implies that $\bigcap_iW^{x_i}$ is nonzero. In other words, each $G(x)$ fixes a nonzero vector in $W$ and thus $G(x) \ne G$ since $W^G=0$. 
\end{proof} 

\section{Proof of Theorem \ref{t:main2}: Orthogonal groups}\label{s:ort}

In this section we prove Theorem \ref{t:main2} for the orthogonal groups $G = \SO_n(k)$. We partition the proof into two cases, according to the parity of $n$. We continue to define $X$ as in \eqref{e:X}, where each $x_i$ has prime order modulo $Z(G)$. We work over an algebraically closed field $k$ of characteristic $p \geqs 0$ that is not algebraic over a finite field. In addition, as explained in Section \ref{s:intro}, we may (and do)  assume that $k$ is uncountable, so $\Delta$ is nonempty if and only if it contains the intersection of countably many generic subsets of $X$.

\subsection{Even dimensional groups}\label{ss:soeven} 

We begin by assuming $G=\SO(V)$ with $\dim V = n=2m \geqs 6$. The cases $m \in \{3,4\}$ are excluded in the statement of Theorem \ref{t:main2} -- they require special attention and they will be handled at the end of this section (see Theorems \ref{t:so6}, \ref{t:so88} and \ref{t:so8}). So for now we will assume that $m \geqs 5$ and we make a distinction between cases according to the parity of $m$. 

\subsubsection{$m \geqs 5$ odd}

To begin with, we will assume $m \geqs 5$ is odd. We first consider the relevant cases with $X = C_1 \times C_2$ that appear in Table \ref{tab:main}. In order to state our first result, we define $x_1, x_2 \in G$ as follows:
\begin{nalign}\label{e:one}
x_1 & = (I_2, \l I_{m-1}, \l^{-1} I_{m-1}), \mbox{ or $p \ne 2$ and $x_1 = (J_3^2,J_2^{m-3})$} \\
x_2 & = (J_2^{m-1},J_1^2); \mbox{ type $a_{m-1}$ if $p=2$}
\end{nalign}
where $\l \in k^{\times}$ and $\l^2 \ne 1$ (for $p=2$, we adopt the notation for $x_2$ from \cite{AS} for unipotent involutions; see Remark \ref{r:as}). 

In the proof of Lemma \ref{l:ex:so} below we will use Gerhardt's result for $\GL_m(k)$ (see Theorem \ref{t:sl}), which we view as a Levi subgroup of the stabilizer in $G$ of a totally singular $m$-dimensional subspace of $V$. Note that such a Levi subgroup stabilizes exactly two totally singular $m$-spaces. Moreover, since $m$ is odd, these two spaces are in different $G$-orbits (recall that $G$ has two orbits on the set of totally singular $m$-spaces, with $U$ and $W$ in the same orbit if and only if $\dim U - \dim(U \cap W)$ is even; this allows us to refer to the \emph{type} of a totally singular $m$-space).

\begin{lem} \label{l:ex:so}   
Suppose $m \geqs 5$ is odd, $r=2$ and $x_1,x_2$ are defined as in \eqref{e:one}.
\begin{itemize}\addtolength{\itemsep}{0.2\baselineskip}
\item[{\rm (i)}] There exists a nonempty open subvariety $Y$ of $X$ such that for all $y \in Y$, $G(y)$ preserves a complementary pair of maximal totally singular subspaces of $V$.
\item[{\rm (ii)}] For all $x \in X$, $G(x)$ preserves maximal totally singular subspaces of $V$ of both types.  
\end{itemize}
In particular, $\Delta$ is empty.
\end{lem}

\begin{proof} 
We may view $x_1$ and $x_2$ as elements of $L= \GL_m(k)$, which is the stabilizer in $G$ of a pair of complementary maximal totally singular subspaces of $V$. More precisely, as an element of $L$ we take
\[
x_1 = (I_1,\l I_{(m-1)/2}, \l^{-1}I_{(m-1)/2}) \mbox{ or } (J_3,J_2^{(m-3)/2})
\]
and we note that the embedding of $x_2$ is unique up to conjugacy in $L$. Set $Y = D_1 \times D_2$, where $D_i = x_i^L$, and note that 
\[
\dim C_1 = m^2+m-2,\; \dim C_2 = m(m-1),
\]
\[
\dim D_1 = \frac{1}{2}(m^2 + 2m -3),\; \dim D_2 = \frac{1}{2}(m^2-1).
\]
In view of the eigenspace dimensions of $x_1$ and $x_2$ on the natural module for $L$, by applying Theorem \ref{t:sl} we deduce that $L(y)$ contains $L'=\SL_m(k)$ for generic $y \in Y$.

Consider the morphism 
\[
\phi:  D_1 \times D_2 \times G \rightarrow C_1 \times C_2,\;\; (d_1, d_2, g) \mapsto (d_1^g, d_2^g).
\] 
We claim that a generic fiber of $\phi$ has dimension $m^2$. To see this, first observe that if  $y = (d_1^g,d_2^g) \in {\rm im}(\phi)$ then
\[
\{ (d_1^{gh^{-1}},d_2^{gh^{-1}},h) \,:\, h \in L\} \subseteq \phi^{-1}(y)
\]
and thus $\dim \phi^{-1}(y) \geqs \dim L = m^2$. Therefore, it suffices to show there is a fiber of dimension $m^2$. Choose $y \in D_1 \times D_2 \subseteq C_1 \times C_2$ so
that $L' \leqs G(y) \leqs L$. Then $G(y)$ is not contained in any other conjugate of $L$, so $\phi(d_1, d_2, g) = y$ only if $g \in L$ and $(d_1^g, d_2^g) = y$.  In particular, the fiber $\phi^{-1}(y)$ is determined by $y$ and so it has dimension $m^2$. 

Therefore, in view of the dimensions of $C_i$ and $D_i$ given above, we deduce that $\phi$ is dominant and thus the image of $\phi$ contains a nonempty open subvariety of $X = C_1 \times C_2$.   If $x \in X$ is in the image of $\phi$, then $G(x)$ is conjugate to a subgroup of $L$ and thus (i) holds. Finally, let us observe that the set of totally singular $m$-spaces of a given type can be identified with the homogeneous space $G/P$ for some maximal parabolic subgroup $P$ of $G$. By Lemma \ref{l:parabolic}, $X_P$ is closed and thus (i) implies that $X_P=X$. This establishes part (ii).
\end{proof}
 
We can now establish Theorem \ref{t:main2} for $G = \SO(V)$, where $\dim V = 2m$ and 
$m \geqs 5$ is odd. Recall that $C_i = x_i^G$ and $d_i = \a(x_i)$ is the maximal dimension of an eigenspace of $x_i$ on $V$.
 
\begin{thm} \label{t:so2nodd}    
If $m \geqs 5$ is odd and $\sum_i  d_i \leqs n(r-1)$, then $\Delta$ is empty if and only if $r=2$ and either $x_1,x_2$ are quadratic, or defined as in \eqref{e:one} (up to ordering).
\end{thm} 

\begin{proof} 
If $r=2$ and $x_1,x_2$ are either quadratic or defined as in \eqref{e:one}, then $\Delta$ is empty by Lemmas \ref{l:quadratic} and \ref{l:ex:so}. Therefore, it remains to show that $\Delta$ is nonempty in all other cases. We partition the proof into two cases. In order to explain the case distinction, recall that if $x_i$ is unipotent then there is a unique unipotent conjugacy class $y^G$ of minimal dimension with $\a(y) = d_i$ (see Section \ref{ss:classes}). We will refer to $y^G$ as the \emph{smallest} unipotent class containing elements with $d_i$ Jordan blocks on $V$ and it will be useful (in Case 2) to note that $y^G$ is contained in the closure of $x_i^G$.

\vs

\noindent \emph{Case 1. If $x_i$ is unipotent, then $C_i=x_i^G$ is the smallest conjugacy class  in $G$ of unipotent elements with $d_i$ Jordan blocks on $V$.}

\vs

Suppose the given condition holds for all unipotent elements $x_i$ in \eqref{e:X}. Let $P=QL$ be the stabilizer in $G$ of a totally singular $m$-space $W$ of $V$, where $Q$ is the unipotent radical and $L = \GL(W) = \GL_m(k)$ is a Levi subgroup. Note that $L$ is the stabilizer of a decomposition $V = W \oplus W'$, where $W'$ is also a totally singular $m$-space (since $m$ is odd, $W$ and $W'$ represent the two $G$-orbits on the set of such spaces). We may identify $Q$ with the $kL$-module $\wedge^2(W)$.

Up to conjugacy, we may embed each $x_i$ in $L$. Indeed, this is clear if $x_i$ is semisimple since $L$ contains a maximal torus of $G$; for unipotent $x_i$, it follows from the hypothesis in Case 1 and the properties of unipotent classes  discussed in Section \ref{ss:classes} (specifically, the Jordan blocks of $x_i$ occur with even multiplicity). If $x_i$ is unipotent, then $C_i \cap L = x_i^L$. On the other hand, if $x_i$ is semisimple then we may assume that if $\l \ne \pm 1$ is an eigenvalue of $x_i$ on $V$, then the multiplicities of $\l$ on $W$ and $W'$ differ by at most $1$.

Let $d_i' = \lceil d_i/2\rceil$ be the maximal dimension of an eigenspace of $x_i$ on $W$. We claim that $\sum_i d_i' \leqs m(r-1)$. This is clear if each $d_i$ is even. Similarly, if exactly one $d_i$ is odd then $\sum_id_i \leqs 2m(r-1)-1$ and once again the claim follows. More generally, suppose $\ell \geqs 2$ of the $d_i$ are odd and note that $d_i \leqs m$ if $d_i$ is odd, otherwise $d_i \leqs 2m-2$. If $(\ell,r) \ne (2,2)$, then 
\[
\sum_{i=1}^r d_i \leqs \ell m +2(r-\ell)(m-1) \leqs 2m(r-1)-\ell
\]
and the result follows. Finally, if $\ell=r=2$ then we may assume $x_1$ or $x_2$ is non-quadratic, so $d_1+d_2 \leqs 2m-2$ and this justifies the claim. 

Set 
\[
Y = D_1 \times \cdots \times D_r = x_1^L \times \cdots \times x_r^L \subseteq X.
\] 
Since $\sum_i d_i' \leqs m(r-1)$, Theorem \ref{t:sl} implies that for generic $y \in Y$, either $L(y)$ contains $L'={\rm SL}_{m}(k)$, or $r=2$ and the $x_i$ are quadratic elements of $L$ with respect to $W$. In the latter situation, the $x_i$ also act quadratically on $V$, which is a case we have already handled.  

  By applying Lemma \ref{l:submodules-conn}, it follows that for generic $x \in X$, $G(x)^0$ is either irreducible on $V$, or it has exactly two composition factors of dimension $m$.  In addition, the rank of $G(x)$ is generically at least $m-1$ by Corollary \ref{c:toralrank}. Since $L'=\SL_m(k)$ contains regular semisimple elements with distinct eigenvalues on  $V$ (that is, $L'$ contains elements that are strongly regular on $V$; see Definition \ref{d:sr}), it follows that for generic $x \in X$, $G(x)^0$ contains strongly regular elements on $V$ (by Lemma \ref{l:sr}(ii)).   In particular, $G(x)^0$ does not generically preserve a nondegenerate $m$-space (since $m$ is odd, the stabilizer in $G$ of a nondegenerate $m$-space does not contain an element with distinct eigenvalues on $V$). As a consequence, either $G(x)^0$ is generically irreducible on $V$, or it has precisely two composition factors (both $m$-dimensional) and any proper invariant subspace of $V$ 
is totally singular.   
Since the totally singular $m$-spaces of a fixed type (i.e. in a given $G$-orbit) form an irreducible projective variety, by applying Lemma \ref{l:parabolic} we deduce that either 
\begin{itemize}\addtolength{\itemsep}{0.2\baselineskip}
\item[{\rm (a)}] $G(x)^0$ acts irreducibly on $V$ for generic $x \in X$; or
\item[{\rm (b)}] for all $x \in X$, $G(x)^0$ stabilizes a totally singular $m$-space (of a fixed  type).
\end{itemize}

If (a) holds, then Corollary \ref{c:sr} implies that $\Delta$ is nonempty. Therefore, to complete the argument in Case 1, we need to rule out (b).

Seeking a contradiction, suppose (b) holds.  Fix $z = (z_1, \ldots, z_r) \in X$ such that $L' \leqs  G(z)^0 \leqs G(z) \leqs L$. Consider the set 
\[
X_0 := z_1^Q \times \cdots \times z_r^Q \subseteq X
\]
and note that $G(y) \leqs P$ and $P' \leqs G(y)Q$ for all $y \in X_0$. We observe that if $y \in X_0$, then either $G(y)$ is contained in a complement to $Q$ in $P$, or $G(y)$ contains $Q$ (since $QL' \leqs QG(y)$). Moreover, if $y_1,y_2$ are distinct elements of $X_0$ and $Q \not\leqs G(y_i)$ for $i=1,2$, then $G(y_1)$ and $G(y_2)$ are contained in distinct complements to $Q$. Since $H^1(L',Q)=0$ by \cite{JP}, it follows that the space of complements to $Q$ in $P$ coincides with the space of $Q$-conjugates of $L$ and so has dimension $m(m-1)/2$ as a variety. On the other hand, aside from the special cases recorded in the statement of the theorem, we see that 
\[
\dim X_0  = r \dim Q - \sum_{i=1}^r\dim Q^{x_i} > \dim Q = \frac{1}{2}m(m-1)
\]
by Corollary \ref{c:wedge2odd} and thus $QL' \leqs G(y)^0$ for some $y \in X_0$. In particular, there exists $y \in X_0$ such that $G(y)^0$ fixes a unique totally singular $m$-space (namely, $W$) and it follows that the set of $y \in X$ such that $G(y)^0$ fixes a totally singular $m$-space in the other orbit is contained a proper closed subvariety of $X$. 
 
But by applying the same argument with respect to the opposite parabolic subgroup of $G$ (namely, the stabilizer of the totally singular $m$-space $W'$) we see that for some $y \in X$, $G(y)^0$  does not fix any totally singular $m$-space in the orbit of $W$.  This is a contradiction and the proof is complete in Case 1. 

\vs

\noindent \emph{Case 2. There exists a unipotent $x_i$ such that $C_i=x_i^G$ is not the smallest conjugacy class of unipotent elements in $G$ with $d_i$ Jordan blocks on $V$.}

\vs

To complete the proof of the theorem, we may assume that we are in the situation described in Case 2. Let $y_i^G$ be the smallest unipotent class containing elements with $d_i$ Jordan blocks on $V$ and recall that $y_i^G$ is contained in the closure of $x_i^G$. In view of Lemma \ref{l:closures}, we may assume that $G(x) \ne G$ for all $x \in \bar{X} \setminus{X}$, which implies (by our work in Case 1) that $r=2$ and either 
\begin{itemize}\addtolength{\itemsep}{0.2\baselineskip}
\item[{\rm (a)}] $y_1$ and $y_2$ are quadratic; or
\item[{\rm (b)}] $y_1$ and $y_2$ have the form given in \eqref{e:one}, up to ordering.
\end{itemize}

First assume $y_1$ and $y_2$ are both quadratic, so $d_1=d_2=m$ (since $d_1 + d_2 \leqs 2m$ and $d_i \geqs m$ for each $i$).  But since $m$ is odd, there are no quadratic
unipotent elements in $G$ with $d_i=m$, so both classes are semisimple and thus $x_1,x_2$ are both quadratic, which is one of the special cases appearing in the statement of the theorem.
For the remainder, we may assume $y_1,y_2$ have the form given in \eqref{e:one}, up to ordering, so $d_1 = m-1$ and $d_2=m+1$. We will consider separately the cases $p=2$ and $p \ne 2$.

Suppose $p=2$, so $x_1=y_1$ is semisimple, $y_2$ is a unipotent involution of type $a_{m-1}$ and the condition in Case 2 forces $x_2$ to be of type $c_{m-1}$. Set $Y = y_1^G \times y_2^G \subseteq \bar{X}$. As explained in the analysis of Case 1, there exists $y \in Y$ such that $\SL_m(k) \leqs G(y)$, where $\SL_m(k)$ fixes a decomposition $V = W \oplus W'$ into totally singular $m$-spaces. By Corollary \ref{c:toralrank}, the rank of $G(x)$ is at least $m-1$ for generic $x \in X$. Then by arguing as above, using Lemma \ref{l:submodules}, we deduce that $\Delta$ is nonempty if $G(x)^0$ is generically irreducible on $V$.  

To complete the proof for $p=2$, we may assume $G(x)^0$ is reducible on $V$ for all $x \in X$, with two $m$-dimensional composition factors for generic $x$.   
In particular, $G(x)^0$ generically fixes a totally singular $m$-space. Notice that $x_1$ and $x_2$ both commute with transvections and so the two classes $C_1$ and $C_2$ are invariant under $G.2 = \OO_{n}(k)$. Therefore, $G(x)^0$ generically preserves totally singular $m$-spaces of both types, which implies that $G(x)^0$ is generically contained in a Levi subgroup $\GL_m(k)$.   But this is a contradiction since every unipotent involution in a Levi subgroup of this form is of type $a$ in the notation of \cite{AS}; in particular, no conjugate of $x_2$ is contained in such a subgroup.

Finally, let us turn to the case $p \ne 2$.  By arguing as above for $p=2$, we see that either $\Delta$ is nonempty, or each $G(x)^0$ is contained in a Levi subgroup of the form $\GL_m(k)$ (indeed, $G(x)^0$ generically preserves totally singular $m$-spaces of both types, in which case Lemma \ref{l:parabolic} forces this property to hold for all $x \in X$). Seeking a contradiction, let us assume each $G(x)^0$ is contained in a Levi subgroup of the form $\GL_m(k)$. There are two cases to consider.

First assume $x_2$ and $y_2$ are not conjugate. By passing to the closures of $x_1^G$ and $x_2^G$, we may assume that $x_1$ and $y_1$ are conjugate and $\dim x_2^G$ is as small as possible (subject to the constraints). This means that we may assume $x_2$
has a Jordan block of size $3$ and $x_1$ acts nontrivially on a $3$-dimensional nondegenerate space. Since $\SO_3(k)$ can be topologically generated by conjugates of any two nontrivial elements other than involutions (see \cite[Theorem 4.5]{Ger}), there exists $x \in X$ such that $G(x)$ induces $\SO_3(k)$ on a nondegenerate $3$-space. For such an element $x$, $G(x)^0$ does not preserve a totally singular $m$-space (since such a space would have to be contained in the orthogonal complement of
  the nondegenerate $3$-space).  Therefore, by applying Lemma \ref{l:parabolic} we deduce that the set of $x \in X$ such that $G(x)$ preserves a totally singular $m$-space (of either type) is a proper closed subvariety of $X$.  But this is incompatible with the fact that each $G(x)^0$ is contained in a Levi subgroup of the form $\GL_m(k)$ and we have reached a contradiction.

Now assume $x_2$ and $y_2$ are conjugate. Given the assumption in Case 2, it follows that $x_1$ is unipotent and not conjugate to $y_1 = (J_3^2,J_2^{m-3})$. If $x_1$ is not contained in a Levi subgroup of the form $\GL_m(k)$, then we can repeat the argument above for $p=2$ to  obtain a contradiction. So we may assume that the multiplicity of each Jordan block of $x_1$ is even. Moreover, by passing to closures, we may assume that $x_1$ either
has a Jordan block of size $4$, or at least four Jordan blocks of size $3$ (with $d_1= m -1)$.   Let $\Omega$ be the variety of totally singular $m$-spaces of a fixed type. By arguing as in the proof of Lemma \ref{l:ex:so}, it follows that $\dim \Omega^{x_1} < \dim \Omega^{y_1}$ and thus
\[
\dim \Omega^{x_1} + \dim \Omega^{x_2}  < \dim \Omega^{y_1} + \dim \Omega^{y_2}
= \dim \Omega.
\] 
Then by applying Lemma \ref{l:basic}, we deduce that $G(x)$ does not generically fix an $m$-space in $\Omega$ and this final contradiction completes the proof of the theorem. 
 \end{proof}   

\subsubsection{$m \geqs 6$ even}

Now let us assume $m \geqs 6$ is even. In the following lemma we define $C_i = x_i^G$ for  $i=1,2$ as follows:
\begin{nalign}\label{e:two}
x_1 & = (I_2, \l I_{m-1}, \l^{-1} I_{m-1}), \mbox{ or $p \ne 2$ and $x_1 = (J_3^2,J_2^{m-4}, J_1^2)$} \\
x_2 & = (J_2^{m}); \mbox{ type $a_{m}$ or $a_m'$ if $p=2$}
\end{nalign}
where $\l \in k^{\times}$ and $\l^2 \ne 1$ (note that  $d_1 = m-1$ if $x_1$ is semisimple and $d_1= m$ if $x_1$ is unipotent). Recall that if $p=2$ then there are three $G$-classes of unipotent involutions with Jordan form $(J_2^m)$, with representatives labelled $a_m$, $a_{m}'$ and $c_m$ in \cite{AS} -- the involutions of type $a_m$ and $a_m'$ are conjugate in $G.2 = \OO(V)$.

\begin{lem} \label{l:ex:so2}   
Suppose $m \geqs 6$ is even, $r=2$ and $x_1,x_2$ are defined as in \eqref{e:two}. \begin{itemize}\addtolength{\itemsep}{0.2\baselineskip}
\item[{\rm (i)}] There exists a nonempty open subvariety $Y$ of $X$ such that for all $y \in Y$, $G(y)$ preserves a complementary pair of maximal totally singular subspaces of $V$.
\item[{\rm (ii)}] For all $x \in X$, $G(x)$ preserves maximal totally singular subspaces of $V$. 
\end{itemize}
In particular, $\Delta$ is empty.
\end{lem}
 
\begin{proof}  
First observe that $x_1^G = x_1^{G.2}$, which implies that $x_1$ preserves totally singular $m$-spaces of both types. On the other hand, $x_2$ preserves a totally singular $m$-space of a fixed type. With this observation in hand, the proof of the lemma is essentially identical to that  of Lemma \ref{l:ex:so} and we omit the details.
\end{proof} 

\begin{thm} \label{t:so2neven}     
If $m \geqs 6$ is even and $\sum_i  d_i \leqs n(r-1)$, then $\Delta$ is empty if and only if $r=2$ and either
\begin{itemize}\addtolength{\itemsep}{0.2\baselineskip} 
\item[{\rm (i)}] the $x_i$ are quadratic, or defined as in \eqref{e:two}, up to ordering; or
\item[{\rm (ii)}] $p \ne 2$ and $x_1 = (J_3,J_2^{m-2},J_1)$, $x_2 = (J_2^m)$, up to ordering.
\end{itemize}
\end{thm} 

\begin{proof}   
First observe that $\Delta$ is empty in (i) and (ii). This is clear if the $x_i$ are quadratic and it follows from Lemma \ref{l:ex:so2} if $x_1$ and $x_2$ are defined as in \eqref{e:two}. If (ii) holds, then $C_1$ is in the closure of the corresponding unipotent class in \eqref{e:two}, so once again Lemma \ref{l:ex:so2} implies that $\Delta$ is empty. It remains to show that $\Delta$ is nonempty in all other cases.

Let $P=QL$ be the stabilizer in $G$ of a totally singular $1$-space $\la v \ra$, where $Q$ is the unipotent radical and $L$ is a Levi subgroup. Set $W=v^{\perp}/\langle v \rangle$, which is a nondegenerate space of dimension $n-2$, and note that we may identify $L'$ with $\SO(W)$.   Let $D_i$ be the closure of $C_i$.  In terms of this notation, we make the following claim.

\vs

\noindent \emph{Claim. $\Delta$ is nonempty if there exists $g_i \in D_i \cap L$ such that  
the Zariski closure of $\langle g_1, \ldots, g_r \rangle$ contains $L'$.}

\vs

To see this, suppose we can find elements $g_i$ with this property. Then Lemma \ref{l:submodules} implies that for generic $x \in \bar{X}$ (and hence also for generic $x \in X$),  $G(x)^0$ has a composition factor on $V$ of dimension at least $n-2$. In particular, $G(x)$ is not generically an irreducible imprimitive subgroup of $G$ with respect to the natural module $V$.  In addition, by combining the bound $\sum_i  d_i \leqs n(r-1)$ with our results in Section \ref{ss:subspaces} on subspace stabilizers, we deduce that $G(x)$ does not generically fix a $1$-space nor a nondegenerate $2$-space. Therefore, $G(x)$ is generically irreducible and primitive on $V$ (which implies that $G(x)^0$ is also generically irreducible).  By applying Corollary \ref{c:toralrank}, we deduce that $G(x)$ has rank $m - 1$ or $m$ for generic $x \in X$. But since $m \geqs 6$, Lemma \ref{l:largeranksubs} implies that $G$ does not have any proper connected irreducible subgroups of rank $m-1$ or $m$. Therefore, $G(x) = G$ for generic $x \in X$ and this justifies the claim.

It remains to establish the existence of the $g_i$. First assume that $x_i$ is semisimple and fix a scalar $\l \in k^{\times}$ such that the $\l$-eigenspace of $x_i$ has dimension $d_i$. Note that $D_i = C_i$ and choose $g_i \in D_i \cap L$ with $g_i v = \l v$.  Let $d_i'$ be the dimension of the largest eigenspace of $g_i$ on $W$. If we can choose $\l \ne \pm 1$, then either $d_i \leqs m$ and $d_i'=d_i-1$, or  $d_i \leqs 2m/3$ and  $d_i'=d_i$ (in the latter case, $x_i$ has at least three distinct eigenvalues with $d_i$-dimensional eigenspaces on $V$). On the other hand, if the only $d_i$-dimensional eigenspace corresponds to an eigenvalue $\pm 1$, then either $d_i'=d_i-2$, or $x_i$ is an involution and $d_i' = d_i=m$.

Now assume $x_i$ is unipotent. If $d_i > m$, then $x_i$ has at least two Jordan blocks of size $1$ (and if $p=2$, at least four such blocks) since the total number of Jordan blocks is even.  Therefore, in this situation we can choose $g_i \in C_i \cap L$ with $d_i'=d_i-2$ (that is, $g_i$ has $d_i-2$ Jordan blocks on $W$). Now assume  $d_i \leqs m$ and consider a Jordan block of $x_i$ of size $e > 1$.  If $e$ is odd, then the closure of $C_i$ contains an element with two Jordan blocks of size $1$ (and one of size $e-2$), so in this case we can choose $g_i \in D_i \cap L$ with $d_i'=d_i$. Similarly, if $e$ is even, then $x_i$ has at least two Jordan blocks of size $e$ and the closure of $C_i$ contains an element with two Jordan blocks of size $1$ (and two of size $e-1$), so once again we can choose $g_i \in D_i \cap L$ with $d_i'=d_i$.  

For $r \geqs 3$, if we choose $g_i \in D_i \cap L$ as above then $\sum_i d_i' \leqs (n-2)(r-1)$ and by applying Theorem \ref{t:so2nodd} we deduce that the closure of $\la g_1, \ldots, g_r \ra$ contains $L' = \SO(W)$ as required. 

Finally, suppose $r =2$, $d_1 + d_2 \leqs n$ and we are not in cases (i) or (ii) in the statement of the theorem. Then the previous argument goes through unless the chosen $g_i \in L$ are among the special cases arising in the statement of Theorem \ref{t:so2nodd}. It just remains to handle these special cases.

If $g_1$ and $g_2$ are both quadratic on $W$, then the condition $d_1 + d_2 \leqs n$ implies that $x_1$ and $x_2$ are both quadratic on $V$, as in part (i) of the theorem. 
So we may assume that $g_1$ and $g_2$ are as in \eqref{e:one}, up to reordering.  In particular, $g_2 = (J_2^{m-2},J_1^2)$, which is of type $a_{m-2}$ if $p=2$.  Given the above construction of $g_2$ from $x_2$, it follows that the Jordan form of $x_2$ is one of the following:
\begin{equation}\label{e:list}
(J_2^m), \; (J_2^{m-2},J_1^4), \; (J_3,J_2^{m-2},J_1),\; (J_3^2,J_2^{m-4},J_1^2).
\end{equation}

First assume $g_1 = (I_2, \l I_{m-2}, \l^{-1}I_{m-2})$ on $W$, so $x_1 = (I_2, \l I_{m-1}, \l^{-1}I_{m-1})$ and $d_1 = m-1$. In turn, this implies that $d_2 \leqs m + 1$, ruling out the second possibility for $x_2$ in \eqref{e:list}. We now consider the cases $p=2$ and $p \ne 2$ separately.

Suppose $p=2$. Here $x_2 = (J_2^m)$ is the only option and we may assume $x_2$ is of type $c_m$ (if $x_2$ has type $a_m$ or $a_m'$ then we are in one of the cases recorded in \eqref{e:two}). To handle this case, we switch parabolics and work in the stabilizer
of a totally singular $m$-space $U$. By replacing $x_2$ by an element of type $a_m$ in the closure of $C_2$, we see that for some $x \in \bar{X}$, $G(x)$ contains the derived subgroup $\SL_m(k)$ of a Levi subgroup of the stabilizer of $U$. Similarly, if we replace $x_2$ by an involution of type $a_m'$ then there exists $y \in \bar{X}$ such that $G(y)$ contains the corresponding subgroup in the stabilizer of $U'$, where $U$ and $U'$ represent the two $G$-orbits on the set of all totally singular $m$-spaces. In the usual way, this shows that either 
\begin{itemize}\addtolength{\itemsep}{0.2\baselineskip} 
\item[{\rm (a)}] $G(x)^0$ is generically irreducible with rank at least $m-1$; or 
\item[{\rm (b)}] $G(x)$ generically preserves totally singular $m$-spaces of both types
and the smallest composition factor of $G(x)$ on $V$ is $m$-dimensional for generic $x \in X$. 
\end{itemize}
If (a) holds then $\Delta$ is nonempty. To eliminate (b), observe that the intersection of two totally singular $m$-spaces of different types is nontrivial since $m$ is even. In particular, if (b) holds then $G(x)$ generically preserves a space of dimension less than $m$, which is a contradiction.

Now let us assume that $p \ne 2$ (we are continuing to assume that $g_1 = (I_2, \l I_{m-2}, \l^{-1}I_{m-2})$ on $W$). We consider the possibilities for $x_2$ recorded in \eqref{e:list}. If $x_2 = (J_2^m)$ then we are in one of the cases in \eqref{e:two}, so we may assume $x_2$ is one of the final two possibilities in \eqref{e:list}. In both cases, the closure of $C_2$ contains an element with Jordan form $(J_2^m)$ and we note that there are two such $G$-classes, which are fused in $G.2 = \OO(V)$. We can now argue as in the $p=2$ case, working with the stabilizers of totally singular $m$-spaces of both types.

To complete the proof, we may assume $p \ne 2$ and $g_1 = (J_3^2, J_2^{m-4})$. From the construction of $g_2$ given above, this forces $x_1 = (J_3^2, J_2^{m-4},J_1^2)$ and once again we need to inspect the possibilities for $x_2$ given in \eqref{e:list}. As above, the bound
$d_1 + d_2 \leqs n$ rules out the second possibility, while $x_2=(J_2^m)$ gives one of the cases in \eqref{e:two}. In the final two cases, we can argue as above: there exists $x \in \bar{X}$ such that $G(x)^0$ contains the derived subgroup of a Levi subgroup of the stabilizer of a totally singular $m$-space of either type and we conclude that $G(x)=G$ for generic
$x \in X$.  
\end{proof} 
 
\subsubsection{$m \in \{3,4\}$}

Here we consider the groups $\SO_6(k)$ and $\SO_8(k)$.    Since $\SO_6(k)$ is isogenous to $\SL_4(k)$, we can use Theorem \ref{t:sl} to state a result in terms of the $6$-dimensional orthogonal module $V$ and the $4$-dimensional linear module $W$ (note that $V = \wedge^2(W)$ as a module for $\SL_4(k)$). As before, we set $d_i = \a(x_i)$ with respect to the action of $x_i$ on $V$. In parts (i) and (iii), we write $\l$ for a scalar in $k^{\times}$ with $\l^2 \ne 1$. 

\begin{thm} \label{t:so6}  
If $m=3$ and $\sum_i  d_i \leqs 6(r-1)$, then $\Delta$ is empty if and only if one of the following holds:
\begin{itemize}\addtolength{\itemsep}{0.2\baselineskip}
\item[{\rm (i)}] $r=3$ and each $x_i$ is of the form $(\l I_3, \l^{-1}I_3)$ or $(J_2^2,J_1^2)$.
\item[{\rm (ii)}] $r=2$ and $x_1, x_2$ are both quadratic on $W$.  
\item[{\rm (iii)}] $r=2$, $x_1 = (\l I_3, \l^{-1}I_3)$ or $(J_2^2,J_1^2)$, and $x_2$ is nonregular (up to ordering).
\end{itemize}
\end{thm}

The result for $\SO_8(k)$ is necessarily more complicated because there are three restricted irreducible $8$-dimensional modules (each a twist of the other by a triality graph automorphism). Moreover, the dimensions of the eigenspaces on the three modules can differ for a given element. For example, if $p \ne 2$ and $x$ has Jordan form $(J_3,J_1^5)$ on one of the $8$-dimensional modules, then it has Jordan form $(J_2^4)$ on the other two.

We will work with the simply connected group $G= \Spin_8(k)$ and we use the standard high weight notation to denote the three modules of interest, namely $V_j = L(\omega_j)$ for $j=1,3$ and $4$. For $g \in G$, let $\alpha_j(g)$ be the dimension of the largest eigenspace of $g$ on $V_j$. In particular, for the elements $x_i$ in \eqref{e:X}, set $d_{ij} = \alpha_j(x_i)$.  

\begin{thm} \label{t:so88}  
If $G = \Spin_8(k)$, $r \geqs 4$ and the $x_i$ in \eqref{e:X} are noncentral, then  
$\Delta$ is nonempty. 
\end{thm}  

\begin{proof}   
Let $V = L(\omega_j)$ for $j=1,3$ or $4$, and let $d_i$ be the maximal dimension of an eigenspace of $x_i$ on $V$. Note that $d_i \leqs 6$ and thus $\sum_i d_i \leqs 8(r-1)$.
In view of Lemma \ref{l:closures}, it suffices to show that $G(x) = G$ for some $x \in \bar{X}$. Since the closure of $x_i^G$ contains the semisimple part of $x_i$ (see \cite[p.92]{Ste}, for example), we may assume that each $x_i$ is either semisimple or unipotent. In fact, by the same argument, we may assume that each $x_i$ is either semisimple of prime order or a long root element (that is, a unipotent element with Jordan form $(J_2^2,J_1^4)$ on $V$).

Let $W$ be a totally singular $4$-dimensional subspace of $V$ and let $P=QL$ be the stabilizer of $W$ in $G$, where $Q$ is the unipotent radical and $L = \GL(W)$ is a Levi subgroup. By applying Theorem \ref{t:sl}, we deduce that there exists $x \in \bar{X}$ such that $G(x)^0$ contains $L' = \SL_4(k)$. Similarly, by applying a triality graph automorphism, there exists $y \in \bar{X}$ such that $G(y)^0$ contains $\SL(W') = \SL_4(k)$, where $W$ and $W'$ represent the two $G$-orbits on totally singular $4$-spaces. We can also find $z \in \bar{X}$ such that $G(z)^0$ contains $\SO_6(k)$, which is the derived subgroup of a Levi subgroup of the stabilizer of a totally singular $1$-space (recall that $\SO_6(k)$ and $\SL_4(k)$ are isogenous, so the latter claim also follows from Theorem \ref{t:sl}). 

These observations imply that for generic $x \in X$, the smallest composition factor of $G(x)^0$ on $V$ is at least $4$-dimensional and the largest is at least 
$6$-dimensional, whence $G(x)^0$ is generically irreducible on $V$ and has rank $3$ or $4$.     Moreover, $G(x)^0$ generically contains semisimple elements with distinct eigenvalues on $V$ (by Lemma \ref{l:sr}) and so either $\Delta$ is nonempty, or $G(x)^0$ is generically contained in a proper maximal
rank subgroup of $G$ (cf. Corollary \ref{c:sr}). But $G$ has no proper maximal rank irreducible connected subgroups and thus $\Delta$ is nonempty. 
\end{proof}

Our main result for $8$-dimensional orthogonal groups is the following. Note that in the statement of this theorem we return to assuming that the $x_i$ in \eqref{e:X} have prime order modulo $Z(G)$. Also note that if $x_i$ is an involution as described in part (i) or (ii) of the theorem, then $C_i = x_i^G$ is ${\rm Aut}(G)$-invariant and so $x_i$ acts the same way on all three $8$-dimensional modules.

\begin{thm} \label{t:so8}   
If $G=\Spin_8(k)$ and $\sum_i d_{ij}  \leqs  8(r-1)$ for all $j$, then $\Delta$ is empty if and only if $r=2$ and either 
\begin{itemize}\addtolength{\itemsep}{0.2\baselineskip}
\item[{\rm (i)}] $p \ne 2$ and $x_1 = x_2 = (-I_4,I_4)$; or 
\item[{\rm (ii)}] $p=2$ and $x_1=x_2$ are involutions of type $c_4$. 
\end{itemize}
\end{thm} 

\begin{proof}  
To begin with, let us assume $r=2$. Clearly, if (i) or (ii) holds then $x_1$ and $x_2$ are quadratic and thus $\Delta$ is empty by Lemma \ref{l:quadratic}. It remains to show that $\Delta$ is nonempty in all other cases. Set $V=V_1$ and $d_i = d_{i1}$ for $i=1,2$. Note that if $x_1$ and $x_2$ are both quadratic on $V$, then either $d_{1j}+d_{2j}>8$ for some $j$ in $\{1,3,4\}$, which is incompatible with the hypothesis of the theorem, or we are in one of the special cases (i) or (ii) in the statement of the theorem. Therefore, we may assume that $x_1$ is not quadratic. There are several cases to consider.

First assume $x_1$ and $x_2$ are both unipotent, so $p \ne 2$ since $x_1$ is non-quadratic. Suppose $d_1 < 4$, in which case $p \ne 3$ since $x_1$ has order $p$ modulo $Z(G)$. Then by passing to closures, we may assume that $x_1= (J_4^2)$ and $x_2 = (J_2^2,J_1^4)$ is a long root element. By Theorem \ref{t:sl},
we can choose $y \in \bar{X}$ so that $G(y)^0$ induces $\SL_4(k)$ on a $4$-dimensional totally singular subspace of $V$. Then for generic $x \in X$, the smallest composition factor of $G(x)^0$ on $V$ is at least $4$-dimensional. By applying the same argument to $V_3$ and 
$V_4$, we see that $G(x)^0$ generically has a composition factor on $V$ of dimension at least $6$ (since a Levi subgroup of the stabilizer of a totally singular $1$-space is conjugate via triality to a Levi of the stabilizer of a totally singular $4$-space). It follows that $G(x)^0$ is generically irreducible on $V$, with rank at least $3$ and it contains elements with distinct eigenvalues on $V$. But there are no proper connected subgroups of $G$ with these properties, whence $G(x) = G$ for generic $x \in X$.   If $d_2 < 4$, the result follows by interchanging $x_1$ and $x_2$. 

Now assume $d_i \geqs 4$ for $i=1,2 $ (we are continuing to assume that $r=2$ and the $x_i$ are unipotent, with $x_1$ non-quadratic). Then $d_1=d_2=4$ and by passing to closures, if necessary, we may assume that $x_1=(J_3^2, J_1^2)$ and $x_2 = (J_2^4)$.  Note that $x_1$ is conjugate to an element in a Levi subgroup ${\rm GL}(W)$ of the stabilizer in $G$ of a totally singular $4$-space $W$. It is also conjugate to an element in a Levi subgroup ${\rm GL}(W')$, where $W$ and $W'$ represent the two $G$-orbits on the set of totally singular $4$-spaces in $V$. On the other hand, $x_2$ is conjugate to an element in ${\rm GL}(W)$ or ${\rm GL}(W')$, but not both (in other words, $x_2$ only fixes a pair of complementary totally singular $4$-spaces in one of the two $G$-orbits). 
Therefore, Theorem \ref{t:sl} implies that we can find $y \in \bar{X}$ such that $G(y)$ induces $\SL_4(k)$ on a totally singular $4$-space and we can now repeat the argument presented in the previous paragraph.

Next assume $r=2$ and $x_1$ is semisimple (and non-quadratic), so $x_1$ is conjugate to elements in both Levi subgroups ${\rm GL}(W)$ and ${\rm GL}(W')$ described above. The same conclusion holds if $x_2$ is semisimple. On the other hand, if $x_2$ is unipotent then there is a unipotent element $y$ in the closure of $C_2$ such that $\a(y) = d_2$ (with respect to $V$) and some conjugate of $y$ is contained in ${\rm GL}(W)$ or $\GL(W')$.  We can now argue as above.

To complete the proof, it remains to show that $\Delta$ is nonempty when $r=3$ (since this immediately gives the result for all $r \geqs 3$). Without loss of generality, we may assume that $d_1 \leqs 4$.  Note that if $x_2$ and $x_3$ are unipotent, then  by passing to closures we may assume that $x_2=x_3 = (J_2^2,J_1^4)$ are long root elements. In all cases, by arguing as above for $r=2$, we see that there exists $x,y \in \bar{X}$ such that  
$G(x)$ induces a subgroup containing $\SL_4(k)$ on a totally singular $4$-space and $G(y)$ has a $6$-dimensional composition factor on $V$. The result now follows as above.
\end{proof}  

\subsection{Odd dimensional groups}\label{ss:odddim}  
 
To complete the proof of Theorem \ref{t:main2} for orthogonal groups, we may assume $G = \SO_n(k)$, where $n = 2m+1$, $m \geqs 1$ and $p \ne 2$. We continue to adopt the notation of the previous section (in particular, note that $Z(G)=1$ and the $x_i$ in \eqref{e:X} have prime order). We begin by handling a special case.

\begin{lem} \label{l:3222}   
Suppose $m$ is odd and $r=2$, where $x_1$ is unipotent, $d_1 = m$ and $d_1+d_2 \leqs n$. Then $\Delta$ is nonempty.
\end{lem}

\begin{proof}   
We use induction on $m$, noting that the case $m=1$ follows by applying \cite[Theorem 4.5]{Ger} with respect to the isogenous group $\SL_2(k)$ (note that $x_1$ is a regular unipotent element). 

For the remainder, let us assume that $m \geqs 3$.  By considering the closure of $x_1^G$ and appealing to the information on unipotent classes in Section \ref{ss:classes}, we may assume that $x_1 = (J_3,J_2^{m-1})$. Note that the $m$-dimensional fixed space $V^{x_1}$ is totally singular, so $x_1$ fixes a $2$-dimensional totally singular subspace of the natural module $V$. Also note that $d_2 \leqs m+1$.  

\vs

\noindent \emph{Case 1. $x_2$ is unipotent.}

\vs

Here $d_2$ is odd, so $d_2 \leqs m$. Since $C_1$ is contained in the closure of any unipotent class containing elements with at most $m$ Jordan blocks on $V$, by passing to closures we may assume that $x_1=x_2$.     

Let $P = QL$ be the stabilizer in $G$ of a totally singular $2$-dimensional subspace $W$ of $V$, where $Q$ is the unipotent radical and $L$ is a Levi subgroup. As observed above, we may assume that $x_i \in P$. Note that we may identify the $kL'$-module $Q/Q'$ with the tensor product $U \otimes U'$, where $U$ and $U'$ are the respective natural modules of the components of $L' = \SL_2(k) \times \SO_{2m-3}(k)$. Set $Y=D_1 \times D_2$, where $D_i$ is the set of elements in $C_i \cap L$ acting nontrivially on $W$. Then each $g_i \in D_i$ has Jordan form $(J_3,J_2^{m-3})$ on the nondegenerate $(2m-3)$-space preserved by $L$, so by induction and the result for $\SL_2(k)$ (see \cite[Theorem 4.5]{Ger}), we deduce that $G(y)$ contains $L'$ for generic $y \in Y$.

The K\"unneth formula gives $H^1(L', Q/Q')=0$ and one checks that  
\[
\dim \, [g_i, Q/Q'] = 2m-2 > 2m-3 = \frac{1}{2}\dim Q/Q'
\]
for all $(g_1,g_2) \in Y$. Therefore, if we fix $(g_1,g_2) \in Y$ then there exists $q_1,q_2 \in Q$ such that  $\langle g_1^{q_1}, g_2^{q_2} \rangle$ is Zariski dense
in $P'$.   By Lemma \ref{l:submodules-conn}, this implies that for generic $x \in X$,  $G(x)^0$ has a composition factor of dimension at least $2m-3$
and does not fix a $1$-space. If $m \geqs 5$ then $2m - 3  > m$ and so $G(x)$ cannot generically be imprimitive and irreducible on $V$. On the other hand, if $m=3$ then
$\dim V = 7$ is a prime and $G$ has no positive dimensional imprimitive subgroups. So in all cases, $G(x)$ is not generically imprimitive and irreducible on $V$. 
Clearly, no element of $C_i$ acts nontrivially on a nondegenerate $2$-space and we also note that $G(x)$ cannot act irreducibly on a $4$-space (each element in $C_i$ would have Jordan form $(J_2^2)$ on such a space and so the action of $G(x)$ would be  reducible by Lemma \ref{l:quadratic}).  Therefore, we see that for generic $x \in X$, either $G(x)^0$ acts irreducibly on $V$, or $G(x)$ preserves a totally singular $2$-space. Moreover, $G(x)^0$ generically contains elements with distinct eigenvalues on $V$ (that is, elements that are strongly regular on $V$). 

Let $\Omega = G/P$ be the variety of $2$-dimensional totally singular subspaces of $V$. We need to compute $\dim \Omega^{x_1}$. Suppose $W \in \O^{x_1}$. If $x_1$ acts trivially on $W$, then $W$ is contained in the $1$-eigenspace $V^{x_1}$, which as noted above is a totally singular 
$m$-space. The variety of $2$-dimensional subspaces of $V^{x_1}$ has dimension $2(m-2)$. On the other hand, if $x_1$ is nontrivial on $W$, then $W$ contains a nonzero vector
in the hyperplane 
\[
V_0 =\{v \in V \, :\,  (x_1-I_n)^2v=0\}.
\]
Let $V_0'$ denote the set of singular vectors in $V_0$.  This is a hypersurface in $V_0$ and so it has dimension $n-2 = 2m-1$.
Let $V_0'' = V_0' \setminus {V^{x_1}}$ and consider the map from $V_0''$ to $\Omega^{x_1}$ given by $v \mapsto \langle v,  x_1v \rangle$.   This is a surjection and all fibers are $2$-dimensional  (the fiber of $\langle v,  x_1v \rangle$ consists of the vectors $av + bx_1v$ with $a \ne 0$). Therefore $\dim \Omega^{x_1} = 2m-3$ and thus
\[
\dim \Omega^{x_1} + \dim \Omega^{x_2} = 4m - 6 < \dim \Omega = 4m - 5,
\]
so for generic $x \in X$,  $G(x)$ does not fix a totally singular $2$-space.  

We conclude that  $G(x)$ is generically irreducible on $V$ (and also primitive). 
Finally, since $G(x)$ generically contains elements that are strongly regular on $V$, we deduce that either $G(x)$ is generically contained in a proper maximal rank subgroup of $G$, or $\Delta$ is nonempty. But $G$ does not have a proper primitive irreducible maximal rank subgroup, whence $\Delta$ is nonempty. 

\vs

\noindent \emph{Case 2. $x_2$ is semisimple.}

\vs

To complete the proof, we may assume $x_2$ is semisimple. First suppose $x_2$ is an involution, so  $x_2 = (-I_{m+1},I_m)$ since $d_2 \leqs m+1$. Define $P = QL$ and $\O = G/P$ as in Case 1. Note that we may embed $x_2$ in $L$ so that it has distinct eigenvalues on the $2$-dimensional totally singular subspace $W$ preserved by $P$. Visibly we have $\dim \Omega^{x_2} = 2m-3$  (the largest component arises by choosing totally singular $1$-spaces from each eigenspace of $x_2$) and the result now follows by repeating the argument in Case 1 for $C_1 = C_2$.

Finally, let us assume $x_2$ is semisimple of odd prime order. If $\dim V^{x_2}=d_2$ then $d_2$ is odd and thus $d_2 \leqs m$. If not, then since each eigenvalue $\l \in k^{\times}\setminus \{\pm 1\}$ has the same multiplicity as $\l^{-1}$, we still deduce that $d_2 \leqs m$. As above, by induction we may assume that $x_2 \in L$ and $L' \leqs G(x)$ for generic $x \in X$. Then a straightforward calculation shows that $\dim \Omega^{x_2} \leqs 2m-3$ and we can now repeat the argument in Case 1.
\end{proof}   

We will also need the following technical lemma on fixed spaces for the action of $G = \SO_9(k)$ on totally singular $4$-spaces.

\begin{lem} \label{l:so9}   
Suppose that $m=4$ and $\Omega$ is the variety of $4$-dimensional totally singular subspaces of the natural module $V$.  
\begin{itemize}\addtolength{\itemsep}{0.2\baselineskip}
\item[{\rm (i)}] If  $g = (J_3^3)$ or   $(I_3, \l I_3, \l^{-1}I_3)$ for some $\l \in k^{\times}\setminus \{\pm 1\}$, then $\dim \Omega^g=3$.
\item[{\rm (ii)}] If $g =(J_2^4, J_1)$, then $\dim \Omega^g=6$.  
\end{itemize}
\end{lem}

\begin{proof} 
First observe that $\dim \Omega = 10$. 
The result for $g= (I_3, \l I_3, \l^{-1}I_3)$ is clear since each $W$ in $\O^g$ intersects the nondegenerate $1$-eigenspace of $g$ in a totally singular $1$-space, while $\dim (W \cap U) \leqs 3$ if $U$ is the $\l$-eigenspace.

Now assume $g = (J_3^3)$ is unipotent. Let $\Omega_1$ be the set of spaces in $\Omega^g$
on which $g$ acts quadratically and let $U$ be a space in $\Omega_1$. Then $U$ is a subspace of $W= \ker(g-I)^2$, which is a $6$-dimensional space with a $3$-dimensional radical $W_1$ (the fixed space of $g$ on $V$).  Note that $W/W_1$ is a nondegenerate $3$-space and so every maximal totally singular subspace is $1$-dimensional. It follows that $U$ must contain $W_1$ and the map $U \mapsto U/W_1$ from $\Omega_1$ to the set of $1$-dimensional totally singular subspaces of $W/W_1$ is an isomorphism of varieties. Therefore $\dim \Omega_1 =1$.   

So it suffices to show that $\dim \Omega_0 = 3$, where $\Omega_0$ is the set of spaces $U$ in $\Omega^g$ such that $g$ acts on $U$ with a Jordan block of size $3$. Let $U$ be a space in $\Omega_0$ and set $U_1 = U \cap V^g$, which is a $2$-dimensional space. Then  
$U/U_1$ is a $g$-invariant totally singular $2$-dimensional subspace of $U_1^{\perp}/U_1$, which is $5$-dimensional and nondegenerate. Moreover, $g$ has Jordan form $(J_3,J_1^2)$ on this $5$-space. Let $R$ be the variety of $2$-dimensional totally singular subspaces in $U_1^{\perp}/U_1$, so $R$ is irreducible and $\dim R = 3$.  We can identify $R$ with the variety of $1$-dimensional subspaces in a $4$-dimensional symplectic space. Under this identification, since $g$ has Jordan form $(J_2^2)$ on the symplectic $4$-space, we see that 
$\dim R^g=1$. Therefore, the variety of $g$-invariant totally singular $4$-spaces whose intersection is a fixed hyperplane in $V^g$ is $1$-dimensional. Let $f$ be the morphism from $\Omega_0$ to the variety of hyperplanes in $V^g$ sending $U$ to $U \cap V^g$.  The image of $f$ is $2$-dimensional and we have shown that every fiber is $1$-dimensional, whence $\dim \Omega_0=3$ as required. 

To complete the proof, let us assume $g=(J_2^4, J_1)$. Let $S = V^g$, so $\dim S =5$ and the radical $R = {\rm im}(g-1)$ of $S$ is $4$-dimensional.
Let $\Omega_i$ be the set of spaces $U$ in $\Omega^g$ with $\dim(U \cap S) = i$. We claim that $\dim \Omega_i \leqs 6$ for each $i$, with equality when $i=2$. Fix $U$ in $\Omega^g$ and observe that $\dim (U \cap S) \geqs 2$ since $g$ is quadratic.
 
If $\dim (U \cap S) = 4$, then $U = R$ and thus $\dim \Omega_4=0$. Next assume $\dim(U \cap S)=3$. Here $U/(U \cap S)$ is a $1$-dimensional totally singular subspace of
the $3$-dimensional nondegenerate space $(U \cap S)^{\perp}/(U \cap S)$. Now the variety of $1$-dimensional totally singular subspaces of a $3$-dimensional orthogonal space has dimension $1$ and we deduce that $\dim \Omega_3 = 4$ since the variety of hyperplanes in $S$ is $3$-dimensional. Finally, suppose $\dim (U \cap S)=2$. Here $U/(U \cap S)$ is a totally singular $g$-invariant subspace of the nondegenerate $5$-space $(U \cap S)^{\perp}/(U \cap S)$. Since $g$ acts nontrivially on this space, it must correspond to a long root element in $\SO_5(k)$ (that is, it must have Jordan form $(J_2^2,J_1)$ on this $5$-space). We can identify the action of $g$ on the variety of totally singular $2$-spaces in this orthogonal $5$-space with the action of $(J_2,J_1^2)$ on $1$-dimensional subspaces of the corresponding $4$-dimensional symplectic space. It follows that the fixed space of $g$ on the variety of $2$-dimensional totally singular subspaces of $(U \cap S)^{\perp}/(U \cap S)$ is $2$-dimensional. Finally, since the variety of $2$-dimensional totally singular subspaces of $S$ is $4$-dimensional, we conclude that $\dim \Omega_2= 6$. This justifies the claim and the proof of the lemma is complete.
\end{proof}  
 
We are now ready to prove Theorem \ref{t:main2} for odd-dimensional orthogonal groups.  Note that if $r=2$ and $x_1, x_2$ are quadratic then $d_1 + d_2 > n$ and so this case does not arise in the following statement.  
 
\begin{thm}\label{t:oddorthogonal}  
Suppose $G = \SO_{n}(k)$, where $n = 2m+1$, $m \geqs 1$ and the $x_i$ in \eqref{e:X} have prime order. If $\sum_i  d_i \leqs n(r-1)$ then $\Delta$ is empty if and only if one of the following holds:
\begin{itemize}\addtolength{\itemsep}{0.2\baselineskip} 
\item[{\rm (i)}] $r=3$, $m=2$ and $x_i = (J_2^2,J_1)$ for all $i$.
\item[{\rm (ii)}] $r=2$, $m \geqs 2$ is even, $x_1 = (J_2^m,J_1)$ and $x_2 = (I_1,\l I_m, \l^{-1}I_m)$ for some $\l \in k^{\times} \setminus \{\pm 1\}$, up to ordering.
\end{itemize}
\end{thm}
 
\begin{proof}   
First we show that $\Delta$ is empty in cases (i) and (ii). Suppose $m=2$. Here we work in the isogenous group $\Sp(W) = \Sp_4(k)$, in which case $x_1$ has Jordan form $(J_2,J_1^2)$ on $W$ and thus any three conjugates of $x_1$ fix a nonzero vector in $W$. This gives the desired conclusion in (i). In (ii) we calculate that 
$\dim C_1 = m^2$ and $\dim C_2 = m^2+m$, so $\dim X = \dim G$ and thus Corollary \ref{c:adjreducible} implies that $G(x)$ acts reducibly on $V$ for all $x \in X$. In particular, $\Delta$ is empty. 
  
To complete the proof of the theorem, it remains to show that $\Delta$ is nonempty in all other cases. We proceed by induction on $m$, noting that $\SO_3(k) \cong {\rm PSL}_2(k)$ and so the result holds for $m=1$ by \cite[Theorem 4.5]{Ger}.
 
Now assume $m \geqs 2$. In view of Lemmas \ref{l:1-spaces}, \ref{l:sing1} and \ref{l:nondeg2}, we see that the bound $\sum_i  d_i \leqs n(r-1)$ implies that for generic $x \in X$, $G(x)$ does not fix a $1$-space (of any type) nor a nondegenerate $2$-space. The case $m=2$ requires special attention.
  
\vs

\noindent \emph{Case 1. $m=2$.}

\vs  
  
For $m=2$ we claim that it suffices to prove that $G(x)$ is generically irreducible on $V$ (recall  that $G(x)$ is generically positive dimensional
by Lemma \ref{l:infinite}). To justify the claim, first observe that the only proper positive dimensional irreducible subgroup of $G$ is $H={\rm PSL}_2(k)$, up to conjugacy, with $p \ne 3$. So let us assume $p \ne 3$ and set $Y = G/H$. Let $g \in H$ be a nontrivial element. Since $\dim (g^G \cap  H)  \leqs 2$ and $\dim g^G \geqs 6$, it follows that $\dim Y^g < \frac{1}{2}\dim Y$ (see \eqref{e:fpr}) and thus Lemma \ref{l:basic} implies that $G(x)$ is not generically contained in a conjugate of $H$.  This justifies the claim.
 
As noted above, $G(x)$ does not generically preserve a $1$-dimensional subspace of $V$ nor a nondegenerate $2$-space. Therefore, $G(x)$ is either generically irreducible (in which case $\Delta$ is nonempty, as explained above), or every $G(x)$ fixes a totally singular $2$-space. The stabilizer of a totally singular $2$-space in $G$ corresponds to the stabilizer in $\Sp_4(k)$ of a $1$-dimensional subspace in the $4$-dimensional symplectic module $W$. Therefore, we just need to determine when the standard inequality $\sum_i d_i' \leqs 4(r-1)$ holds, where $d_i'$ is the maximal dimension of an eigenspace of $x_i$ on $W$. 

This clearly holds if $r \geqs 4$. For $r=3$, the inequality fails if and only if each $x_i$ is a transvection on $W$, which gives the case recorded in part (i) of the theorem. Finally, suppose $r=2$ and the inequality does not hold. Then up to reordering, noting that $d_1+d_2 \leqs 5$, we may assume $x_1$ is a transvection on $W$ and $x_2$ is not regular (that is, $x_2$ has a $2$-dimensional eigenspace on $W$). If $x_2$ is unipotent or an involution, then $d_1+d_2 > 5$, which is a contradiction. The remaining case is given in (ii).  

\vs

\noindent \emph{Case 2. $m \geqs 3$.}

\vs 

Now assume $m \geqs 3$. Let $P=QL$ be the stabilizer in $G$ of a totally singular $1$-space  spanned by $v \in V$, where $Q$ is the unipotent radical and $L$ is a Levi subgroup. Choose $g_i \in C_i \cap P$ and let $d_i'$ be the dimension of the largest eigenspace of $g_i$ acting on the nondegenerate $(n-2)$-dimensional space $U = v^{\perp}/\langle v \rangle$. We embed each $g_i$ in $P$ so that $d_i'$ is as small as possible.  
 
\vs

\noindent \emph{Case 2.1. $\sum_i d_i' \leqs (r-1)(n-2)$.}

\vs 
 
Suppose $\sum_i d_i' \leqs (r-1)(n-2)$. To begin with, we will assume we are not in one of the special cases recorded in parts (i) and (ii)  (with respect to the action of $g_i$ on $U$).  Then by induction, $G(x)^0$ generically has a composition factor on $V$ of dimension at least $n-2$. Since $G(x)$ does not generically fix a nondegenerate $2$-space nor any $1$-space, it follows that $G(x)$ is generically irreducible on $V$ with rank $m-1$ or $m$. If $m \geqs 4$, then Lemma \ref{l:largeranksubs} implies that there is no proper subgroup of $G$ with these properties, whence $G(x) = G$ for generic $x \in X$.   

Now assume $m=3$. If $G(x)$ has rank $3$ for any $x$ (and so for generic $x$), then $G(x)$ is generically a rank $3$ irreducible subgroup. By inspection, we see that there is no proper subgroup of $G$ with this property and thus $\Delta$ is nonempty. Therefore, we may assume $G(x)$ has rank $2$ for generic $x \in X$. Recall that there exists $x \in X$ such that $QL' \leqs G(x)$. This implies that $f(G(x))$ is contained in the closure of $f(L')$, where $f: G \to \mathcal{M}_7(\textbf{x})$ is the morphism in \eqref{e:char}, sending $g \in G$ to its characteristic polynomial on $V$, which is contained in the variety 
$\mathcal{M}_7(\textbf{x})$ of monic polynomials in $k[\textbf{x}]$ of degree $7$. In particular, this implies that $G(x)$ does not contain elements with distinct eigenvalues on $V$ (since the $1$-eigenspace of any element in $QL'$ is at least $3$-dimensional). Now the only connected irreducible rank $2$ subgroups of $G$ are $G_2$ and $A_2$ (the latter occurring only for $p=3$). But the weight spaces on $V$ for the maximal tori of these subgroups are all $1$-dimensional, so they both contain regular semisimple elements and we have reached a contradiction.
 
To complete the argument in Case 2.1, we may assume that we are in one of the special cases recorded in parts (i) and (ii) of the theorem (in terms of the action of $g_i$ on the $(n-2)$-space $U$). In particular, $r \leqs 3$ and $m$ is odd.

First assume $m=3$. If $x_i$ is unipotent and not of the form $(J_2^2,J_1^3)$, then $x_i$ has a Jordan block of size at least $3$ and clearly we may choose $g_i \in C_i \cap P$ so that it does not have Jordan form $(J_2^2,J_1)$ on $U$. It follows that if we are forced to be in one of the special cases recorded in (i) and (ii) (with respect to the action on $U$), then the original elements $x_i \in G$ must also be of the form given in one of these special cases (for the action on $V$). This is a contradiction.

Now assume $m \geqs 5$.  Here $r=2$ and we may assume that $x_1$ is unipotent and $x_2 = (I_1, \l I_m, \l^{-1}I_m)$ is semisimple. The condition $d_1+d_2 \leqs n$ implies that $d_1 \leqs m+1$ and thus $d_1 \leqs m$ since $m$ is odd. In particular, $x_1$ must have a Jordan block of size at least $3$ and as noted above we may choose $g_1 \in C_i \cap P$ so that it does not have Jordan form $(J_2^{m-1},J_1)$ on $U$. Once again, we have reached a contradiction. 

\vs

\noindent \emph{Case 2.2. $\sum_i d_i' > (r-1)(n-2)$.}

\vs 

For the remainder of the proof we may assume that $\sum d_i ' > (r-1)(n -2)$. Note that if $x_i$ is semsimple, then either $d_i'=d_i-2$ or one of the following holds:
\begin{itemize}\addtolength{\itemsep}{0.2\baselineskip} 
\item[{\rm (a)}] $x_i$ is an involution with $d_i= m+1$ and $d_i'=m$; 
\item[{\rm (b)}] $x_i$ has odd order, $d_i \leqs m$ and $d_i'=d_i-1$;  or
\item[{\rm (c)}] $x_i$ has odd order and $d_i = d_i' \leqs n/3$.
\end{itemize}
Similarly, if $x_i$ is unipotent, then either $d_i'=d_i-2$, or $x_i$ has at most one Jordan
block of size $1$ and $d_i = d_i' \leqs m+\e$, where $\e=1$ if $m$ is even, otherwise $\e=0$ (note that $d_i$ is always odd if $x_i$ is unipotent).   

Next observe that if $d_i' = d_i - 2$ for all but at most one $i$, then $\sum_i d_i' \leqs (r-1)(n-2)$, which is a contradiction. Similarly, the above observations imply that if 
$d_i' \ne d_i - 2$ for three distinct $i$, say $i=1,2,3$, then 
$\sum_{i=1}^3  d_i' \leqs 2(2n-1)$ and thus $\sum_i d_i' \leqs (r-1)(n-2)$. Therefore, up to reordering the $C_i$, we may assume that $d_i' = d_i - 2$ if and only if $i \geqs 3$. Notice that if $x_1$ and $x_2$ are semisimple, then either $d_i'=d_i-1$ or $d_i = d_i' \leqs n/3$ for $i=1,2$, which implies that $\sum_i d_i' \leqs (r-1)(n-2)$. Therefore, we may assume $x_1$ is unipotent with $d_1 = d_1' \leqs m+\e$.  In particular, $d_1$ is odd. 
 
First assume $m \geqs 3$ is odd, so $d_1 \leqs m$ (since $d_1$ is odd). Suppose $x_2$ is  unipotent. If $d_2 < m$, then $d_2 \leqs m-2$ and this is incompatible with the bound on $\sum_i d_i'$. On the other hand, if $d_2 = m$ then by passing to the closures of $C_1$ and
 $C_2$, we may assume that $x_1=x_2 = (J_3,J_2^{m-1})$. Let us also observe that if $x_2$ is semisimple, then $d_2 \leqs m+1$ and $d_1 = m$ (indeed, if $d_1 \leqs m-2$ then it is easy to check that $\sum_id_i' \leqs (r-1)(n-2)$). Therefore, in both cases Lemma \ref{l:3222} implies that $\la y_1,y_2 \ra$ is Zariski dense in $G$ for generic $(y_1,y_2) \in C_1 \times C_2$ and the result follows.
 
To complete the proof, we may assume $m \geqs 4$ is even and $x_1$ is unipotent with $d_1 = d_1' \leqs m+\e$. We partition the analysis into three subcases.

\vs

\noindent \emph{Case 2.2.1. $m \geqs 4$ even, $x_2$ is unipotent, $r \geqs 3$.}

\vs 

Here we assume $r \geqs 3$ and $x_2$ is unipotent, so $d_1 + d_2 = d_1' + d_2' > n-2$ and $d_i \leqs m + 1$ for $i=1,2$. By passing to closures, we may assume that $x_1=x_2 = (J_2^m,J_1)$.  Note that for any class $C_3$  we have $\sum_{i=1}^3 d_i \leqs 2n$ and so we may assume that $r=3$.  Since $\sum_id_i'>2(n-2)$, it follows that $d_3' \geqs n-4$ and $d_3 = d_3'+2 \geqs n-2$, so $x_3$ is one of the following:
\[
(J_3,J_1^{n-3}), \; (J_2^2,J_1^{n-4}),\; (-I_{n-1},I_1).
\]

First assume $x_3$ is unipotent. Then by passing to closures once again, we may assume that $x_3 = (J_2^2,J_1^{n-4})$. By applying \cite[Theorem 4.5]{Ger}, we see that there exists $(y_1,y_2) \in C_1 \times C_2$ such that $J = \SL_2(k)^{m/2}$ is the Zariski closure of $\la y_1, y_2 \ra$ (as a $kJ$-module, $V$ is a direct sum of $m$ 
$2$-dimensional spaces and a copy of the trivial module). In particular, there exists a conjugacy class $D  = z^G \subseteq C_1C_2$ of prime order semisimple elements such that each eigenspace of $z$ on $V$ is at most  $2$-dimensional.  If $d'$ denotes the dimension of the largest eigenspace of $z$ on $U$, then $d'+d_3'  = n-2$. Therefore, if we set $X' = D \times C_3$ then our earlier work in Case 2.1 implies that $G = G(y)$ for generic $y \in X'$ and we conclude by applying Lemma \ref{l:products}.  

To complete the analysis of Case 2.2.1, we may assume $r=3$ and $x_3 = (-I_{n-1},I_1)$ is a pseudoreflection. 
Fix $(y_1,y_2) \in C_1 \times C_2$ such that the Zariski closure of $\langle y_1, y_2 \rangle$ is the subgroup $J = \SL_2(k)^{m/2}$ described above. Note that $J$ fixes only finitely many nondegenerate subspaces of $V$ and so there is a nonempty open subset of $C_3$ such that no element in this open set fixes a proper nondegenerate space fixed by $J$.     

Next observe that the variety of $2$-dimensional totally singular subspaces fixed by $J$ is $1$-dimensional. Let us also note that for $y_3 \in C_3$, 
the variety of $y_3$-invariant totally singular $2$-spaces coincides with the variety of all $2$-dimensional totally singular subspaces of the $(-1)$-eigenspace of $y_3$. The latter variety has codimension at least $2$ in the variety of all totally singular $2$-dimensional subspaces of $V$. Therefore, we may choose $y_3 \in C_3$ so that it does not preserve any $J$-invariant totally singular $2$-space (nor any $J$-invariant proper nondegenerate subspace). Given $x = (y_1,y_2,y_3) \in X$ with the above properties, we see that $G(x)$ is either irreducible, or it must preserve a totally singular subspace. But $y_3$ preserves such a space if and only if it is contained in its $(-1)$-eigenspace.  So if $G(x)$ fixes a totally singular space $W$, then any $J$-invariant subspace of $W$ is $G(x)$-invariant as well.   Since
every irreducible $kJ$-submodule of $V$ is $2$-dimensional (or trivial), this would imply that $G(x)$ fixes a totally singular $2$-space, contrary to the choice of
$y_3$.   Thus for generic $y_3 \in C_3$,  $G(x)$ is irreducible.  

Finally, we claim that $G(x)$ acts primitively on $V$ (for $x = (y_1,y_2,y_3) \in X$ as above).
Suppose $G(x)$ is imprimitive, so it preserves a decomposition $V = V_1 \oplus \cdots \oplus V_t$, where $t \geqs 3$ is odd and $G(x)$ acts transitively on the set of summands (since $G(x)$ is irreducible). Since $J<G(x)$ is connected, it must fix each summand in this decomposition. But then $\la y_3 \ra$ must act transitively on the summands, which is a contradiction since $t \geqs 3$ and $y_3$ is an involution. Therefore, $G(x)$ is a primitive irreducible group containing psuedoreflections and it is well known that this implies that  $G(x)=G$ (see \cite[Theorem 8.3]{GS} for a much more general result).  
   
\vs

\noindent \emph{Case 2.2.2. $m \geqs 4$ even, $x_2$ is unipotent, $r=2$.}

\vs 

Now assume $r = 2$ and $x_2$ is unipotent, so $d_i = d_i' \leqs m+1$ is odd and we have $2m \leqs d_1+d_2 \leqs 2m+1$. Up to reordering, we may assume that $d_1= m+1$ and $d_2 = m-1$. Then by passing to closures, we may assume that $x_1 = (J_2^m,J_1)$ and $x_2 = (J_3^3,J_2^{m-4})$. Let $P=QL$ be the stabilizer in $G$ of a totally singular $m$-space $W$, where $Q$ is the unipotent radical and $L$ is a Levi subgroup.   

First assume that $m=4$. Here we can choose $y_1 \in C_1$ so that it has Jordan form $(J_3,J_1)$ on $W$. By Theorem \ref{t:sl}, there exists $x \in X$ so that $G(x)=G(x)^0$ acts as $\SL_4(k)$ on $W$ and acts uniserially on $V$. For generic $x \in X$ it follows that $G(x)$ has rank at least $3$, the smallest nonzero $G(x)$-invariant subspace is at least $4$-dimensional and $G(x)$ is generically primitive. Since $x_2$ does not preserve a $4$-dimensional nondegenerate space, we see that $G(x)$ is either generically irreducible and primitive, or it generically preserves a totally singular $4$-space.  Let $\Omega$ be the homogeneous variety of totally singular $4$-spaces.  By Lemma \ref{l:so9} we have $\dim \Omega^{x_1} =3$ and 
$\dim \Omega^{x_2}=6$. Therefore, $\dim \Omega^{x_1} + \dim \Omega^{x_2} < 
\dim \Omega = 10$ and so for generic $x$, $G(x)$ does not preserve a totally singular $4$-space. In view of Lemma \ref{l:largeranksubs}, we conclude that $\Delta$ is nonempty.
 
Now assume $m \geqs 6$.   We can choose $y \in X$ such that $G(y)=QL'$  (here we are using Theorem \ref{t:sl}, working with elements in $C_1$ and $C_2$ that stabilize a totally singular $m$-space).   Note that $\mathrm{End}_{QL'} (V) = k$ and so for generic $x \in X$, $\dim \mathrm{End}_{G(x)^0}(V) = 1$ and thus $V$ is an indecomposable $kG(x)^0$-module.   By induction, we can choose $x \in X$ so that $G(x) = \SL_2(k) \times \SO_{n-4}(k)$, which is the derived
subgroup of a Levi subgroup of the stabilizer of a totally singular $2$-space. Therefore, for generic $x \in X$ we observe that the smallest nonzero $G(x)^0$-invariant subspace has dimension at least $m$, and $G(x)^0$ also has a composition factor of dimension at least $n-4$.  This forces $G(x)^0$ to be generically irreducible of rank at least
$m-1$ and then Lemma \ref{l:largeranksubs} implies that $G(x)=G$ for generic $x$. 

\vs

\noindent \emph{Case 2.2.3. $m \geqs 4$ even, $x_2$ is semisimple.}

\vs 
  
Finally, to complete the proof of the theorem we may assume $m \geqs 4$ is even and $x_2$ is semisimple. Recall that $x_1$ is unipotent with $d_1 = d_1'$ and we may assume that either 
$d_2'=d_2 -1$, or $m=4$ and $d_2=d_2'=3$. The condition $\sum_id_i'>(r-1)(n-2)$ implies that 
$d_1' + d_2' \geqs 2m$, so $x_1 = (J_2^m,J_1)$ and either $x_2 = (I_1,\l I_m, \l^{-1}I_m)$, or $x_2 = (-I_m,I_{m+1})$, or $m=4$ and $d_2=d_2'=3$.     
 
First assume $x_2 = (I_1,\l I_m, \l^{-1}I_m)$, so $d_1' + d_2' =2m$. If $r \geqs 3$, then 
\[
\sum_{i=1}^rd_i'  \leqs 2m+(r-2)(n-3) \leqs (r-1)(n-2),
\]
which is incompatible with the defining condition of Case 2.2. On the other hand, if $r=2$ then we are in the special case identified in part (ii) of the theorem. 

Next assume $x_2 = (-I_m, I_{m+1})$. Here $d_1 = d_1'=d_2=m+1$ and $d_2'=m$, so the condition $\sum_id_i \leqs n(r-1)$ implies that $r \geqs 3$. In addition, the inequality 
$\sum d_i' > (r-1)(n-2)$ implies that $r=3$ and $x_3 = (-I_{2m}, I_1)$. As in Case 2.2.1, we can choose $y_i \in C_i$, $i =1,2$ such that $J = \SL_2(k)^{m/2}$ is the Zariski closure of $\langle y_1, y_2 \rangle$. Then by repeating the argument in Case 2.2.1, we can 
find $y_3 \in C_3$ such that $G(x)=G$ for $x = (y_1,y_2,y_3) \in X$.

Finally, let us assume $m=4$ and $d_2=d_2'=3$.  First observe that we can choose $a \in X$ such that $G(a)$ contains $L' = \SL_4(k)$, where $L$ is a Levi subgroup of the stabilizer in $G$ of a totally singular $4$-space. In particular, the rank of $G(x)$ is generically at least $3$ and as above it suffices to show that $G(x)$ is generically irreducible on $V$. 

Since $G(a)$ does not preserve any $2$-dimensional or $3$-dimensional spaces, and $G(x)$ does not generically fix a totally singular $4$-space by Lemma \ref{l:so9}, it follows that if $G(x)$ is not generically irreducible, then $G(x)$ must preserve a nondegenerate $5$-space for generic $x \in X$. On such a $5$-space, an element of $C_2$ either has a $3$-dimensional $1$-eigenspace, or two $2$-dimensional eigenspaces. In the first case, we see that $G(x)$ fixes a $1$-space, while $G(x)$ fixes a $2$-space in the latter (this is the exception for $m=2$). As noted above, neither possibility can occur, so $G(x)$ does not generically fix a nondegenerate $5$-space and the proof of the theorem is complete. 
\end{proof} 
 
\section{Proof of Theorem \ref{t:main2}: Symplectic groups}\label{s:symp}

In this section we complete the proof of Theorem \ref{t:main2} by handling the symplectic groups $G = \Sp(V) = \Sp_n(k)$, where $n=2m$ with $m \geqs 2$ and $k$ is an uncountable algebraically closed field of characteristic $p \geqs 0$. We continue to define $X$ as in \eqref{e:X}, where each $x_i$ has prime order modulo $Z(G)$. We consider separately the cases where $p \ne 2$ and $p=2$. 

 \subsection{Odd characteristic}\label{ss:odd}
   
In this section we assume $p \ne 2$, so $Z = Z(G) = \la -I_n \ra$. We begin by considering the    special case $m=2$.

\begin{thm} \label{t:sp4odd} 
Suppose $m=2$ and $p \ne 2$. If $\sum_i d_i \leqs 4(r-1)$ then $\Delta$ is empty if and only if one of the following holds (up to ordering and conjugacy):
\begin{itemize}\addtolength{\itemsep}{0.2\baselineskip}
\item[{\rm (i)}] $r=2$ and $x_1,x_2$ are quadratic. 
\item[{\rm (ii)}] $r=2$, $x_1 = (-I_2,I_2)$ and $x_2$ is non-regular.  
\item[{\rm (iii)}] $r=3$, $x_1 = x_2 = (-I_2,I_2)$ and $x_3$ is quadratic.  
\item[{\rm (iv)}] $r=4$ and $x_i = (-I_2,I_2)$ for all $i$. 
\end{itemize}
\end{thm} 

\begin{proof}  
First assume that we are in one of the cases labelled (i)-(iv). If $y = (y_1, \ldots, y_r) \in X$ then in each case it is easy to check that the $y_i$ preserve a common $1$-dimensional subspace of the $5$-dimensional orthogonal module $kG$-module $W$. In particular, $\Delta$ is empty. To complete the proof, we need to show that $\Delta$ is nonempty in all the remaining cases.

This essentially follows from the corresponding result for ${\rm Spin}_5(k)$ in Theorem \ref{t:oddorthogonal}. For each $g \in G$, let $\alpha(g)$ (respectively, $\b(g)$) be the dimension of the largest eigenspace of $g$ on $V$ (respectively, $W$). Since $W \oplus k \cong \wedge^2(V)$, we have the following relationship between $\alpha$ and $\beta$ (here $\l \in k^{\times}$ is a scalar with  $\l^2 \ne 1$):

\begin{itemize}\addtolength{\itemsep}{0.2\baselineskip}
\item[{\rm (a)}] If $g = (-I_2,I_2)$ then $\alpha(g) = 2$ and $\beta(g)=4$.
\item[{\rm (b)}] If $g$ is semisimple, $\a(g) = 2$ and $g$ is not an involution, then either $g = (\l I_2, \l^{-1}I_2)$ and $\beta(g) =3$,
or $g = (I_2, \l I_1, \l^{-1}I_1)$ and $\beta(g) = 2$.
\item[{\rm (c)}] If $g$ is regular, then $\alpha(g)= \beta(g)=1$.  
\item[{\rm (d)}] If $g = (J_2, J_1^2)$ is a long root element, then $\alpha(g)= \beta(g) = 3$.
\item[{\rm (e)}] If $g = (J_2^2)$ is a short root element, then $\alpha(g)=2$ and $\beta(g)=3$.
\end{itemize} 
 
So by applying Theorem \ref{t:oddorthogonal}, we deduce that if $r \geqs 4$ then $\Delta$ is nonempty unless $r=4$ and each $x_i$ is an involution acting as $(-I_4,I_1)$ on $W$. This corresponds to the case recorded in part (iv). Similarly, the exceptional cases for $r=2,3$ are easily determined from the above information in (a)-(e).
\end{proof}

\begin{rem}  
The previous result implies that if $G = \Sp_4(k)$ and $p \ne 2$, then $\Delta$ is nonempty if and only if there exists $x \in X$ such that $G(x)^0$ is irreducible on both the symplectic and orthogonal $kG$-modules. 
\end{rem} 
      
For the remainder of Section \ref{ss:odd}, we may assume that $m \geqs 3$. It turns out that the cases $m \in \{3,4\}$ also require special attention. In particular, we will need the following technical lemma on fixed point spaces for certain actions of $G = \Sp_6(k)$. Recall that $\a(g)$ denotes the dimension of the largest eigenspace of $g \in G$ on the natural module $V$.
   
   \begin{lem} \label{l:fpsp6} 
   Suppose $m=3$, $p \ne 2$ and let $P = QL$ be the stabilizer in $G$ of a totally isotropic $3$-space, where $Q$ is the unipotent radical and $L$ is a Levi factor. Set $X_1 = G/P$ and $X_2 = G/N$, where $N=N_G(L) = L.2$, so $\dim X_1=6$ and $\dim X_2 = 12$. Let $g \in G$ be an element of prime order modulo $Z(G)$.
   \begin{itemize}\addtolength{\itemsep}{0.2\baselineskip}
   \item[{\rm (i)}]  If $g = (-I_2,I_4)$ then $\dim X_1^g = 4$ and $\dim X_2^g = 8$.
   \item[{\rm (ii)}] If $g = (J_2,J_1^4)$ or $(J_2^3)$ then $\dim X_1^g= 3$ and $g$ has no fixed points on $X_2$.
   \item[{\rm (iii)}]  If $g = (J_2^2,J_1^2)$ or $g = (I_4, \l I_1, \l^{-1}I_1)$ for some $\l \in k^{\times}\setminus \{\pm 1\}$, then $\dim X_1^g = 3$ and $\dim X_2^g = 6$.
   \item[{\rm (iv)}]  If $g = (\l I_3, \l^{-1} I_3)$ for some $\l \in k^{\times}\setminus \{\pm 1\}$, then 
   $\dim X_1^g = 2$ and $\dim X_2^g = 4+\e$, where $\e=2$ if $\l^2 = -1$, otherwise 
   $\e=0$.  
   \item[{\rm (v)}] If $g = (J_3^2)$ then
   $\dim X_1^g =2$ and $\dim X_2^g = 4$. 
   \item[{\rm (vi)}] If $g$ is semisimple of odd order and $\a(g) \leqs 2$, then $\dim X_1^g \leqs 2$ and $\dim X_2^g = 1+\e$, where $\e=3$ if $1$ is an eigenvalue, otherwise $\e=0$.  
   \end{itemize}
   \end{lem}
   
   \begin{proof}  
   This is a straightforward computation using \eqref{e:fpr}, which states that if $Y = G/H$ is a homogeneous space, then 
   \[
   \dim Y - \dim Y^g = \dim g^G - \dim (g^G \cap H)
   \]
   for all $g \in H$. We omit the calculations. Note that if $g \in N$, then either $g \in L$, or $g^2$ is the central involution in $G$.
   \end{proof} 
   
We can now prove the main result for $G = \Sp_6(k)$ with $p \ne 2$ (note that the elements in parts (ii) and (iii) are described up to conjugacy and multiplication by $-1$).

   \begin{thm} \label{l:sp6odd} 
   Suppose $m=3$, $p \ne 2$ and the $x_i$ in \eqref{e:X} have prime order modulo $Z(G)$. If 
$\sum_i  d_i \leqs 6(r-1)$ then $\Delta$ is empty if and only if one of the following holds (up to ordering):
\begin{itemize}\addtolength{\itemsep}{0.2\baselineskip}
\item[{\rm (i)}] $r=2$ and $x_1,x_2$ are quadratic. 
   \item[{\rm (ii)}] $r=2$, $x_1 = (-I_2,I_4)$ and $x_2 = (J_3^2)$ or $(I_2, \l I_2, \l^{-1}I_2)$ for some $\l \in k^{\times} \setminus \{\pm 1\}$. 
   \item[{\rm (iii)}] $r=3$ and $x_i = (-I_2,I_4)$ for all $i$. 
   \end{itemize}
   \end{thm}
   
   \begin{proof}    
   First we show that $\Delta$ is empty in the cases described in parts (i)-(iii). This is clear in (i) (see Lemma \ref{l:quadratic}).  In cases (ii) and (iii) we claim that $G(x)$ generically fixes a totally isotropic $3$-space (and so every $G(x)$ fixes a totally isotropic $3$-space).
   
   To see this, let $Y$ be an irreducible component of $X \cap L^r$ of maximal dimension, where $L$ is a Levi subgroup of the stabilizer of a totally singular $3$-space.
   By applying Theorem \ref{t:sl}, we deduce that $G(y)$ contains $L'=\SL_3(k)$ for generic $y \in Y$. Next consider the map 
   $f: G \times  Y \to X$ given by $f(g,y)=y^g$, where $y^g = (y_1^g, \ldots, y_r^g)$ for $y = (y_1, \ldots, y_r)$. Then for generic $y \in Y$, we have $f^{-1}(y) =\{(g,y) \,:\, g \in L\}$ since $G(y)$ contains $\SL_3(k)$.  It follows that a generic fiber of $f$ is $9$-dimensional and thus Lemma \ref{l:fpsp6} implies that $f$ is dominant and thus $X_L$ is open and dense in $X$. The claim follows. 
   
        It remains to show that $\Delta$ is nonempty in the remaining cases. The following claim will play a key role in the proof.  

\vs

\noindent \emph{Claim. $\Delta$ is nonempty if $G(x)$
      is generically primitive on $V$.}
      
      \vs
      
To prove the claim, let us assume $G(x)$ is generically primitive and note that this is equivalent to assuming that $G(x)^0$ is generically irreducible on $V$.
    Now the only maximal primitive positive dimensional closed connected subgroups of $G$ are $A_1A_1 = \SO_3(k) \otimes \Sp_2(k)$ and $A_1$ (with $p \ne 3,5$ in the latter case). Let $H = A_1A_1$ and set $\O = G/H$, so $\dim \O = 15$. Then $\dim (C_i \cap H) \leqs 4$, with $\dim (C_i \cap H) = 2$ if $x_i$ is an involution. In addition, we note that $C_i \cap H$ is empty if $x_i = (J_2,J_1^4)$ and by applying \eqref{e:fpr} we conclude that $\dim \Omega^g  \leqs  7$ for all noncentral $g \in G$.  Therefore, Lemma \ref{l:basic} implies that $G(x)$ is not generically contained in a conjugate of $H$. An even easier argument handles the case $H = A_1$ since $\dim (C_i \cap H) \leqs 2$ for all $i$. This justifies the claim. 
      
      We now partition the remainder of the proof into several cases. 

\vs

\noindent \emph{Case 1. $x_1$ is either semisimple with at least $4$ distinct eigenvalues, or unipotent with Jordan form $(J_6)$, $(J_4,J_2)$ or $(J_4,J_1^2)$.}

\vs

Let $P=QL$ be the stabilizer of a $1$-space $\la v \ra$, where $Q$ is the unipotent radical and $L$ is a Levi subgroup with $L' = \Sp_4(k)$. We may assume each $x_i$ is contained in $P$ and by applying Theorem \ref{t:sp4odd} we see that there exists $y \in X$ such that $G(y)$ is contained in $P$ and it induces $\Sp_4(k)$ on the nondegenerate $4$-space $v^{\perp}/\la v \ra$. By Lemma \ref{l:submodules-conn}, $G(x)^0$ has a composition factor on $V$ of dimension at least $4$ for generic $x \in X$. Now the bound $\sum_i d_i \leqs 6(r-1)$ implies that $G(x)$ does not generically fix a $1$-space nor a nondegenerate $2$-space (see Lemmas \ref{l:1-spaces} and \ref{l:nondeg2}) and thus $G(x)$ is generically irreducible and has rank $2$ or $3$.  Since $G(x)^0$ generically has a composition factor of dimension at least $4$, it follows that $G(x)$ is generically primitive and this implies that $\Delta$ is nonempty by the above claim.

\vs

\noindent \emph{Case 2. $x_1 = (J_3^2)$ or $(I_2, \l I_2, \l^{-1}I_2)$ for some $\l \in k^{\times}\setminus \{\pm 1\}$.}

\vs

Let $P = QL$ be the stabilizer of a totally isotropic $3$-space $W$ in $V$. Here $Q$ is the unipotent radical and $L = \GL(W)$ is a Levi subgroup fixing a decomposition $V = W \oplus W'$, where $W'$ is a complementary totally isotropic $3$-space. Without loss of generality, we may assume $x_1$ is a regular semisimple or unipotent element of $L$. 

Suppose $x_2 \ne (J_2,J_1^4)$. Then by replacing $x_2$ by a suitable conjugate, we may assume that  $x_2 \in P$ and $x_2Q \in P/Q$ is nontrivial, so by Theorem \ref{t:sl} there exists $x \in X$ such that $P' \leqs G(x)^0Q$. As a consequence, for generic $x \in X$,  the smallest composition factor
of $G(x)^0$ on $V$ is at least $3$-dimensional.  Therefore, either 
\begin{itemize}\addtolength{\itemsep}{0.2\baselineskip}
\item[(a)] $G(x)^0$ is generically irreducible on $V$; or 
\item[(b)] $G(x)^0$ is contained in a conjugate of $P$ for all $x \in X$.
\end{itemize}
Now assume $x_2 = (J_2,J_1^4)$, in which case the inequality $\sum_id_i \leqs 6(r-1)$ implies that $r \geqs 3$. If $x_3 \ne (J_2,J_1^4)$ then the previous argument implies that (a) or (b) holds. On the other hand, if $x_3 = (J_2,J_1^4)$ then we can find a noncentral element $y \in C_2C_3$ of prime order with $y \ne (J_2,J_1^4)$ and once again we deduce that (a) or (b) holds. 

Recall that if (a) holds then $G(x)$ is generically primitive and we conclude that $\Delta$ is nonempty. Therefore, it remains to eliminate case (b). That is, we need to identify an element $x \in X$ such that $G(x)^0$ does not fix a totally isotropic $3$-space.

Suppose $x_2$ is not an involution. For $i=1,2$ we can choose $y_i$ in the closure of $C_i$ such that the closure of $\la y_1, y_2 \ra$ induces $\Sp_2(k)$ on some nondegenerate $2$-space (note that if $x_1$ is unipotent, then we may assume $y_1$  acts nontrivially on a nondegenerate $2$-space). If we now take a tuple $x \in X$ with $y_1$ and $y_2$ in the first and second coordinates, then 
$G(x)^0$ does not fix a totally isotropic $3$-space and we conclude that $\Delta$ is nonempty. 

Finally, let us assume each $x_i$ is an involution for $i \geqs 2$. If $r \geqs 3$, then we can repeat the previous argument, working with a noncentral element $y \in C_2C_3$ with $y^2 \ne 1$. And if $r=2$ then we are in case (ii) in the statement of the theorem.

\vs

To complete the proof of the theorem, we may assume $d_i \geqs 3$ for all $i$.

\vs

\noindent \emph{Case 3.  $d_i \geqs 3$ for all $i$.}

\vs 

First assume that $d_1=d_2=3$. Then $x_1$ and $x_2$ are quadratic, so we may assume $r \geqs 3$. We can choose $y_i \in C_i$ so that the Zariski
closure of  $\langle y_1, y_2 \rangle$ contains $\SL_2(k)^3$ and so contains a regular semisimple element of prime order. By applying Lemma \ref{l:products}, we can now complete the argument as in Case 1.  

Next assume $d_1=3$ and $d_i \geqs 4$ for all $i \geqs 2$, in which case the condition $\sum_id_i \leqs 6(r-1)$ implies that $r \geqs 3$. Note that if $i \geqs 2$ then $x_i$ is either an involution, or a unipotent element with Jordan form $(J_2^2,J_1^2)$ or $(J_2,J_1^4)$, or a semisimple element of the form $(I_4,\l I_1,\l^{-1}I_1)$ for some $\l \in k^{\times}\setminus \{\pm 1\}$. In each case we can find conjugates $y_i \in C_i$ for $i \geqs 2$ that act nontrivially on a nondegenerate $4$-space $W$ and trivially on the nondegenerate $2$-space $W^{\perp}$.     In addition, we can choose a conjugate $y_1$ of $x_1$ so that its largest eigenspace on $W$ is $2$-dimensional. 
 
Suppose $r \geqs 4$. Since $x_1$ is not an involution, Theorem \ref{t:sp4odd} implies that we can choose $y \in X$ such that $G(y)$ induces $\Sp(W)$ on $W$.  
Then $G(x)^0$ is either generically irreducible on $V$, or generically it acts irreducibly on a $4$-dimensional nondegenerate subspace, whence the same is true for $G(x)$.
Since $G(x)$ does not generically preserve a $1$-space nor a nondegenerate $2$-space, we deduce that $G(x)$ is generically primitive and the result follows. If $r=3$, then the same argument applies unless $x_2$ and $x_3$ are involutions. But in this case we observe that  $C_2C_3$ contains a semisimple element of the form $(I_2, \l I_2, \l^{-1}I_2)$ and thus Case 2 applies.  

Finally, let us assume $d_i \geqs 4$ for all $i$. First assume that $r = 3$, in which case the bound $\sum_id_i \leqs 6(r-1)$ implies that $d_i=4$ for all $i$. If each $x_i$ is an involution then we are in case (iii), so let us assume $x_1$ is not an involution.   We can choose $y$ so that $G(y)^0$ induces
a subgroup containing $\SL_3(k)$ on a totally isotropic $3$-space. Then either $G(x)^0$ is generically irreducible and the result follows, or every $G(x)^0$ preserves a totally isotropic $3$-space. Let us assume we are in the latter situation and let $\Omega$ be the variety of totally isotropic $3$-spaces, so $\dim \Omega = 6$. Since $d_i \geqs 4$, it follows that no  $x_i$ interchanges two spaces in $\O$ and thus every $G(x)$ has a fixed point on $\Omega$. However, Lemma \ref{l:fpsp6} gives $\dim \Omega^{x_i} \leqs 4$, with equality if and only if $x_i$ is an involution, whence
$\sum_{i} \dim \Omega^{x_i} < (r-1)\dim \Omega$
and Lemma \ref{l:basic} implies that $G(x)$ does not generically fix a totally isotropic $3$-space. This is a contradiction. An entirely similar argument applies if $r \geqs 4$ (including the case where $r=4$ and each $x_i$ is an involution).  
\end{proof} 
     
In order to handle the general case, we need the following result concerning the action of a symplectic group on the variety of maximal totally isotropic subspaces of the natural module. Here we allow $p=2$.

\begin{lem}\label{l:spti} 
Suppose $G = \Sp(V)= \Sp_{n}(k)$, where $n = 2m$, $m \geqs 1$.  Let $Y = G/P$, where $P=QL$ is the stabilizer of a totally isotropic $m$-space $W$, with unipotent radical $Q$ and Levi subgroup $L = \GL(W)$. Then 
\[
\dim Y^g  =  \dim C_Q(g)  =  \dim \Sym^2(W)^g
\]
for all unipotent elements $g \in L$. 
\end{lem} 

\begin{proof}   
First observe that we may identify $Q$ with the $kL$-module $\Sym^2(W)$, so the equality $\dim C_Q(g)  =  \dim \Sym^2(W)^g$ is clear.
We proceed by induction on $m$, noting that the case $m=1$ is trivial. Now assume $m \geqs 2$ and observe that $\dim Y^g \geqs \dim C_Q(g)$ since $Q$ acts simply transitively on the set of totally isotropic complements to $W$ in $V$.

Let $U$ be a $1$-dimensional $g$-invariant subspace of $W$ and let  $Y(U)$ 
be the set of spaces in $Y$ containing $U$. Then $U^{\perp}/U$ is a nondegenerate $(n-2)$-space and we may view $g$ as an element in a Levi subgroup of the corresponding parabolic subgroup $P_1 = Q_1L_1$ of $\Sp(U^{\perp}/U)$ (namely, the stabilizer of the maximal totally isotropic space $W/U$). Then there exists a positive integer $\ell$ such that $g = J_{\ell} \oplus M$ as an element of $L = \GL(W)$ and $g = J_{\ell-1} \oplus M$ as an element of $L_1 = \GL(W/U)$. By induction, the dimension of the fixed space of $g$ acting on $Y(U)$ is $\dim C_{Q_1}(g)$, which coincides with the dimension of the fixed space of $g$ acting on $\Sym^2(W/U)$.   

By applying Lemma \ref{l:tensor}, it follows that if $\ell$ is odd and $p \ne 2$, then 
\[
\dim C_Q(g) - \dim C_{Q_1}(g)  =  \frac{1}{2}(\ell +1) - \frac{1}{2}(\ell-1) + c = c + 1,
\]
where $c$ is the number of Jordan blocks of $g$ on $M$ that have size at least $\ell$.
Similarly, if $\ell$ is even and $p \ne 2$, then the same argument shows that the difference in centralizer dimensions is $c$.
If $p = 2$, then Lemma \ref{l:tensor} still implies that the difference in centralizer dimensions is at least $c$.  

Now if the action of $g$ on $W$ has the form as above, then $U$ is contained in $W^{(g-1)^{\ell-1}}$ and the variety of such $1$-spaces (in the projective space of $W$) has dimension $c -1$. Therefore, the dimension of the fixed space of $g$ on $Y'$, the variety of totally isotropic subspaces which intersect $W$ nontrivially, is at most $\dim C_Q(g)$ and the result follows.
\end{proof} 

     We also need some fixed point space computations for $G = \Sp_8(k)$. 
     Recall that $\a(g)$ is the maximal dimension of an eigenspace of $g$ on the natural module $V$.  
     
     \begin{lem} \label{l:fpsp8} 
     Suppose $m=4$, $p \ne 2$ and $\O$ is the variety of totally isotropic 
     $4$-spaces in $V$, so $\dim \O = 10$. Then
     \[
     \dim \Omega^g = \left\{\begin{array}{ll}
     7 & \mbox{if $g = (-I_2,I_6)$} \\
     6 & \mbox{if $g = (-I_4,I_4)$} \\
     4 & \mbox{if $g = (J_3^2,J_1^2)$}\\
     3 & \mbox{if $g = (J_3^2,J_2)$}
     \end{array}\right.
     \]
     \end{lem}
     
     \begin{proof} 
     This is clear if $g$ is an involution since any totally isotropic $4$-space fixed by $g$ is of the form $U_1 \perp U_{-1}$, where $U_{\l}$ is a totally isotropic subspace
     of the nondegenerate $\l$-eigenspace of $g$ on $V$. Now assume $g = (J_3^2,J_2)$ or $(J_3^2,J_1^2)$. By replacing $g$ by a suitable conjugate, we may assume $g$ is contained in the stabilizer $P=QL$ of a totally isotropic $4$-space $W$, where $Q$ is the unipotent radical of $P$ and $L$ is a Levi subgroup. If $g = (J_3^2,J_1^2)$ then we may assume $g$ is contained in $L$ and thus Lemma \ref{l:spti} implies that $\dim \Omega^g = \dim C_Q(g) = 4$. 
     
    Finally, let us assume $g = (J_3^2,J_2)$. Without loss of 
     generality, we may assume that $g$ has Jordan form $(J_3,J_1)$ on $W$. Since $g$ does not fix any totally isotropic $4$-space $W'$ with $W \cap W'=0$, it follows that $(W \cap W')^g$ contains a $1$-space $\la v \ra$ for all $W' \in \Omega^g$. 
     
     Let $U = W^g \cap [g,W]$ and note that $\dim U =1$. Then $\Omega^g = \Omega_1 \cup \Omega_2$ is the union of two subvarieties: $\Omega_1$ comprises the totally isotropic $4$-spaces $W'$ with 
     $U \subseteq W'$ and $\Omega_2$ consists of the spaces $W'$ such that $U \cap W' = 0$. It suffices to show that $\dim \O_1 \leqs 3$ and $\dim \O_2 = 3$.  
     
     Suppose $W' \in \O_2$. Then $W'/\la v \ra$ is a totally isotropic $3$-space in the nondegenerate $6$-space $v^{\perp}/\la v \ra$ and $g$ induces a Jordan block of size $3$ on $W/\la v \ra$.  By Lemma \ref{l:spti}, a unipotent element in $\Sp_6(k)$ of the form $(J_3^2)$ has a $2$-dimensional fixed point space on the variety of totally 
     isotropic $3$-spaces of the natural module for $\Sp_6(k)$.   The map $W' \mapsto W' \cap W^g$ defines  a morphism from $\Omega_2$ to the variety of $1$-dimensional subspaces of $W^g$.   Since $\dim W^g = 2$, the latter variety is $1$-dimensional and we conclude that $\dim \Omega_2=3$.
     
     Now assume $W' \in \O_1$. Then $W'/U$ is a $3$-dimensional totally isotropic subspace of $U^{\perp}/U$.
     Note that $g$ acts on $W/U$ with 
     two Jordan blocks and there exists an element $h$ in the closure of $g^P$ that is contained in a corresponding Levi subgroup of $\Sp_6(k)$.
     Then by Lemma \ref{l:spti}, $h$ has a $3$-dimensional fixed space on the variety of totally isotropic $3$-spaces in $U^{\perp}/U$ and so the same is true for $g$. This gives $\dim \Omega_1 \leqs 3$ and thus $\dim \Omega^g=3$ as required.
          \end{proof} 
    
     We are now in a position to establish our main result for symplectic groups with $p \ne 2$. We will apply induction on $m$, noting that special care is required for $m=4$. In the statement, we allow $\l = \mu$ in part (ii).
     
    \begin{thm} \label{t:spodd}  
    Suppose $G = \Sp_{n}(k)$, where $n = 2m$, $m \geqs 4$, $p \ne 2$ and the $x_i$ in \eqref{e:X} have prime order modulo $Z(G)$. If $\sum_i  d_i \leqs n(r-1)$ then $\Delta$ is empty if and only if one of the following holds (up to ordering):
\begin{itemize}\addtolength{\itemsep}{0.2\baselineskip} 
\item[{\rm (i)}] $r=2$ and $x_1,x_2$ are quadratic.
\item[{\rm (ii)}] $r=2$, $m=4$, $x_1 = (-I_4,I_4)$ and $x_2$ is either $(I_4, \l I_1, \l^{-1}I_1,\mu I_1, \mu^{-1}I_1)$ or $(J_3^2, J_1^2)$, with $\l,\mu \in k^{\times}\setminus \{\pm 1\}$.
\item[{\rm (iii)}] $r=3$, $m=4$, $x_1= x_2 = (-I_2, I_6)$ and $x_3 = (-I_4, I_4)$.
\end{itemize}
   \end{thm}
   
   \begin{proof}  
   As usual, first observe that $\Delta$ is empty if the conditions in (i), (ii) or (iii) are satisfied. This is clear in (i). Now consider (ii) and (iii), so $m=4$. Let  
   $L = \GL_4(k)$ be the stabilizer in $G$ of a pair of complementary totally isotropic $4$-spaces. By applying Theorem \ref{t:sl}, there exists $y \in X$ such that $L' = \SL_4(k) \leqs G(y) \leqs L$. 
   For each $i$, let $D_i$ be an $L$-class in $C_i \cap L$ of maximal dimension and set $Y = D_1 \times \cdots \times D_r$. Then for generic $y \in Y$, $G(y)$ is contained in a unique conjugate of $L$. It is straightforward to compute $\dim X = \sum_i \dim C_i$ and $\dim Y = \sum_i \dim D_i$, which gives 
        \[
       \dim G + \dim Y = \dim X + \dim L.
       \]
  Consider the morphism $f: G \times Y \rightarrow X$ defined by $f(g, y) = y^g$. Fix $y \in Y$ such that $G(y)$ is contained in a unique conjugate of $L$ and consider the fiber $f^{-1}(y)$. Since $G(y)$ has a unique fixed point on $G/L$, this implies that $g \in N_G(L) = L.2$ for all $(g,z) \in f^{-1}(y)$, so a generic fiber of $f$ has dimension at most $\dim L$. It follows that $f$ is dominant
       (and the dimension of a generic fiber is precisely $\dim L$) and thus $G(x)$ is conjugate to a subgroup of $L$ for generic $x \in X$. We conclude that $\Delta$ is empty if (ii) or (iii) holds. 
       
      To complete the proof, we will use induction on $m$ to prove that $\Delta$ is nonempty in all the remaining cases. 
  
  Let $P = QL$ be the stabilizer in $G$ of a $1$-dimensional subspace $\la v \ra$ of $V$, where $Q$ is the unipotent radical and $L$ is a Levi subgroup. Let $W$ be the nondegenerate 
  $(n-2)$-space $v^{\perp}/\la v \ra$. We may assume that each $x_i$ is contained in $P$ and we write $g_i \in \Sp(W)$ for the induced action of $x_i$ on $W$ (note that $g_i$ is quadratic on $W$ only if $x_i$ is quadratic on $V$). We define $d_i'$ to be the dimension of the largest eigenspace of $g_i$ on $W$ and we assume the embedding of $x_i$ in $P$ is chosen to minimize $d_i'$. Notice that one of the following holds:
   \begin{itemize}\addtolength{\itemsep}{0.2\baselineskip} 
\item[{\rm (a)}] $d_i'= d_i -2$.
   \item[{\rm (b)}]  $d_i'=d_i -1$ and $d_i \leqs n/2$. 
   \item[{\rm (c)}]   $d_i'=d_i$ and $d_i \leqs n/3$.
   \end{itemize}
   
We claim that $\sum_i d_i' \leqs (n-2)(r-1)$. Let $\ell$ be the number of $i$ with $d_i' = d_i-2$. If $\ell \leqs 1$ then
\[
\sum_{i=1}^{r}d_i' \leqs \sum_{i=1}^rd_i - 2(r-1) \leqs n(r-1) - 2(r-1) = (n-2)(r-1)
\]
as required. Similarly, if $\ell \geqs 2$ then
\[
\sum_{i=1}^{r}d_i' \leqs (r-2)(n-3) + 2(n/2-1) \leqs (n-2)(r-1).
\]
This justifies the claim. Then by induction, excluding the cases where $m \in \{4,5\}$ and the $g_i$ line up with one of the special cases for $\Sp_{2(m-1)}(k)$ in the statement of the theorem (for $m=5$) or Theorem \ref{l:sp6odd} (for $m=4$), it follows that $G(x)^0$ generically has a composition factor on $V$ of dimension at least $n-2$ and has rank $m-1$ or $m$.
  
Suppose that either $m \geqs 6$, or $m \in \{4,5\}$ and the $g_i$ do not correspond to one of the special cases for $\Sp_{2(m-1)}(k)$. By Lemmas \ref{l:1-spaces} and \ref{l:nondeg2}, the condition $\sum_id_i \leqs n(r-1)$ implies that $G(x)$ does not fix a $1$-space nor a nondegenerate $2$-space, whence $G(x)^0$ is generically irreducible on $V$ (with rank $m-1$ or $m$, as noted above). For $m \geqs 5$ we find that there is no proper closed connected subgroup of $G$ with these properties and thus $\Delta$ is nonempty. Now assume $m=4$. Here $G$ has an irreducible subgroup of rank $3$ with connected component 
\[
H = A_1^3 = {\rm Sp}_{2}(k) \otimes {\rm Sp}_{2}(k) \otimes {\rm Sp}_{2}(k).
\]
By considering the characteristic polynomials on $V$ of elements in $\Sp(W) = \Sp_6(k)$ (every such polynomial has $1$ as a double root) and applying Theorem \ref{t:toralrank}(ii), we deduce that $G(x)^0$ is not generically conjugate to $H$ and the result follows. 

To complete the proof, we may assume $m \in \{4,5\}$ and the $g_i$ correspond to one of the special cases for $\Sp_{2(m-1)}(k)$. Recall that we are assuming (i) does not hold, so we never descend to the case where $r=2$ and $g_1,g_2$ are quadratic on $W$.

First assume $m=4$, $r=3$ and we descend to the special case described in part (iii) of Theorem \ref{l:sp6odd}. Here each $x_i$ is an involution in $G = \Sp_{8}(k)$ and the condition $\sum_{i}d_i \leqs 16$ implies that we may assume $d_1=4$.  
Suppose $d_2=4$. Here we can find a semisimple element $h \in C_1C_2$ of prime order with four distinct $2$-dimensional eigenspaces on $V$ and thus $Y = h^G \times x_3^G$ does not descend to a special case for $\Sp_6(k)$. Then by the previous argument, it follows that $\Delta$ is nonempty if $d_2=4$ (and similarly if $d_3=4$). Therefore, we may assume $d_2=d_3=6$, which corresponds to the special case described in part (iii) of the theorem (up to reordering).
  
Next assume $m=4$, $r=2$ and we descend to one of the possibilities in part (ii) of Theorem  \ref{l:sp6odd}. Since we are assuming (iii) does not hold (in the statement of the theorem we are proving), it follows that neither $g_1$ nor $g_2$ is of the form $(I_2, \l I_2, \l^{-1}I_2)$. The remaining possibility is that we descend to the case where $g_1 = (-I_2,I_4)$ and $g_2 = (J_3^2)$, up to scalars and ordering. Since we are assuming (ii) does not hold, it follows that $x_2 = (J_3^2,J_2)$. Then by applying Theorem \ref{t:sl} with respect to a Levi subgroup of the stabilizer in $G$ of a totally isotropic $4$-space, we deduce that $G(y)$ contains the derived subgroup of this parabolic subgroup for some $y \in X$ (and in particular, $G(x)$ has  rank $3$ or $4$ for generic $x \in X$). Since the centralizer of $G(y)^0$ in $\mathrm{End}(V)$ is $1$-dimensional, it follows that either $G(x)^0$
is generically irreducible on $V$, or every $G(x)$ fixes a totally isotropic $4$-space. The latter possibility does not arise since
\[
\dim \Omega^{x_1} + \dim \Omega^{x_2} \leqs 6 + 3 < \dim \Omega = 10,
\]
where $\Omega$ is the variety of $4$-dimensional totally isotropic subspaces of $V$ (see Lemma \ref{l:fpsp8}). Therefore, $G(x)^0$ is generically irreducible on $V$ and we claim that this forces $\Delta$ to be nonempty. 
  
To justify the claim, let $T$ be a maximal torus of $L' =\SL_4(k)$, where $L$ is a Levi subgroup of the stabilizer in $G$ of a totally isotropic $4$-space $W$. Then 
\[
{\rm Lie}(G) \cong \mathfrak{gl}_4(k) \oplus \Sym^2(W) \oplus \Sym^2(W^*)
\]
and thus the nonzero weight spaces for $T$ on the adjoint module ${\rm Lie}(G)$ are all $1$-dimensional.  Therefore, $T$ contains strongly regular elements with respect to the adjoint module and so $G(x)^0$ generically contains strongly regular elements. As a consequence, either $G(x)=G$ for generic $x$, or $G(x)^0$ is contained in a proper maximal rank subgroup that acts irreducibly on the natural module for $G$. But there are no such subgroups and so we conclude that $\Delta$ is nonempty. 
   
Finally, let us assume $m=5$ and we descend to one of the special cases for $\Sp_8(k)$ in the statement of the theorem, so $r \leqs 3$. If $r=3$ then each $x_i$ is an involution and the condition $\sum_id_i \leqs 20$ implies that at least two of the $x_i$ are of the form $(-I_6,I_4)$ up to scalars. But these elements descend to involutions in $\Sp_8(k)$ of the form $(-I_4,I_4)$, which is not a special case. Now suppose $r=2$. Here we may assume $x_1 = (-I_6, I_4)$, so $d_2 \leqs 4$ and thus $d_2' \leqs 3$. But once again this does not correspond to a special case for $\Sp_8(k)$ and the proof of the theorem is complete. 
\end{proof}
   
\subsection{Even characteristic}\label{ss:even}
 
In this section we complete the proof of Theorem \ref{t:main2} by handling the symplectic groups $G = {\rm Sp}_{n}(k)$ with $n=2m \geqs 4$ and $p=2$. With reference to \eqref{e:X}, recall that $d_i$ denotes the maximal dimension of an eigenspace of $x_i$ on $V$. In this section we also define $e_i = \dim V^{x_i} \leqs d_i$ for $i = 1, \ldots, r$. Note that $e_i = d_i$ if $x_i$ is unipotent. As explained in Lemma \ref{l:oinsp}, $\Delta$ is empty if $\sum_ie_i \geqs n(r-1)$. 
Also note that $Z(G)=1$ since $p=2$.

We begin by considering the case $m=2$, which requires special attention. Here $G = \Sp_4(k)$ has two $4$-dimensional restricted irreducible $kG$-modules, namely $V_j = L(\omega_j)$ for $j=1,2$, which are interchanged by a graph automorphism $\tau$. Note that $\tau$ does not preserve eigenspace dimensions in general. For example, $\tau$ interchanges long and short root elements, so if $x$ has Jordan form $(J_2,J_1^2)$ on $L(\omega_1)$, then it has Jordan form $(J_2^2)$ on $L(\omega_2)$ (that is, $\tau$ fuses the $G$-classes containing $b_1$ and $a_2$ involutions, with respect to the notation in \cite{AS}). On the other hand, the dimension of the largest eigenspace of a semisimple element is invariant under $\tau$, but the set of eigenvalues is not preserved in general. For example, $\tau$ takes a quadratic semisimple element of the form $(\l I_2, \l^{-1}I_2)$ to one of the form $(I_2,\mu I_1, \mu^{-1}I_1)$, and vice versa. 

For $j = 1,2$, set $e_{ij} = \dim V_j^{x_i}$ and let $d_{ij}$ be the maximal dimension of an eigenspace of $x_i$ on $V_j$.
 
\begin{thm}\label{l:sp4even} 
Suppose $m=p=2$ and the $x_i$ in \eqref{e:X} have prime order. If $\sum_i d_{ij} \leqs 4(r-1)$ and $\sum_i e_{ij} < 4(r-1)$ for $j=1,2$, then $\Delta$ is nonempty.
\end{thm} 

\begin{proof}  
First observe that if $r=2$ and $x_1$ and $x_2$ are both quadratic on $V_1$ then $e_{12}+e_{22} \geqs 4$, which violates the hypothesis. Therefore, this situation does not arise.

More generally (for all $r \geqs 2$), we observe that $G(x)$ does not generically preserve a $1$-dimensional subspace of $V_1$ or $V_2$. And since a graph automorphism interchanges the two conjugacy classes of maximal parabolic subgroups, as well as the modules $V_1$ and $V_2$, we see that $G(x)$ does not 
generically fix a totally isotropic $2$-space in either representation. Furthermore, Lemma \ref{l:nondeg2} implies that $G(x)$ does not generically preserve a nondegenerate $2$-space and thus $G(x)$ is generically irreducible on both modules.   

The maximal imprimitive subgroups of $G$ with respect to $V_1$ are of the form 
$\Sp_2(k) \wr S_2$ and $\GL_2(k).2$, corresponding to the stabilizers in $G$ of a suitable direct sum decomposition of $V_1$ into two nondegenerate $2$-spaces and two totally isotropic $2$-spaces, respectively. Under the graph automorphism, the first subgroup is sent to $\OO_4(k)$ and the second
is mapped to a reducible subgroup. The inequality $\sum_i e_i < 4(r-1)$ implies that $G(x)$ is not generically contained in a conjugate of $\OO_4(k)$ (see Lemma \ref{l:oinsp}) and we have already noted that $G(x)$ is generically irreducible on $V_1$ and $V_2$. This implies that $G(x)$ is generically primitive with respect to both modules and we conclude that $G(x)=G$ for generic $x$. 
\end{proof}

We note some immediate consequences. Recall that \emph{short root elements} are involutions of type $a_2$ in the notation of \cite{AS}.

\begin{cor} \label{c:sp4even}   
Suppose $m=p=2$ and the $x_i$ in \eqref{e:X} have prime order. Then $\Delta$ is nonempty if any of the following hold:
 \begin{itemize}\addtolength{\itemsep}{0.2\baselineskip}
   \item[{\rm (i)}]  $x_i$ is a regular semisimple element for some $i$.   
   \item[{\rm (ii)}] $r \geqs 3$ and none of the $x_i$ are long or short root elements.
   \item[{\rm (iii)}] $r \geqs 5$.    
   \end{itemize}
\end{cor} 

By combining the previous two statements, we can present a result for $G = \Sp_4(k)$ in terms of the natural symplectic module (as in Theorem \ref{t:main2}). 

\begin{thm}\label{t:sp4even1} 
Suppose $m=p=2$ and the $x_i$ in \eqref{e:X} have prime order.  If $\sum_i d_{i} \leqs 4(r-1)$ and $\sum_i e_{i} < 4(r-1)$ with respect to the $4$-dimensional symplectic module, then $\Delta$ is empty if and only if one of the following holds (up to ordering):
\begin{itemize}\addtolength{\itemsep}{0.2\baselineskip}
   \item[{\rm (i)}]  $r=2$ and $x_1,x_2$ are quadratic.
   \item[{\rm (ii)}]  $r=3$, $x_1,x_2$ are short root elements and $x_3$ is quadratic.
   \item[{\rm (iii)}] $r=4$ and each $x_i$ is a short root element. 
   \end{itemize}
\end{thm} 

Next we turn to the case $G = {\rm Sp}_{6}(k)$, which also requires special attention.

\begin{thm}\label{l:sp6even} 
Suppose $m=3$, $p=2$ and the $x_i$ in \eqref{e:X} have prime order. If $\sum_i  d_i \leqs 6(r-1)$ and $\sum_i e_i < 6(r-1)$, then $\Delta$ is empty if and only if $r=2$ and $x_1,x_2$ are quadratic.
\end{thm} 

\begin{proof}   
As usual, if $r=2$ and the $x_i$ are quadratic, then $\Delta$ is empty by Lemma \ref{l:quadratic} and therefore it remains for us to show that $\Delta$ is nonempty in all the remaining cases. We partition the proof into several subcases. Let $V$ be the natural module.
 
 \vs
 
 \noindent \emph{Case 1. $x_1$ is semisimple with at least four distinct eigenvalues on $V$.}
 
 \vs
 
Let $U$ be a $1$-dimensional subspace of $V$ and choose a conjugate of $x_1$ so that it acts as a regular semisimple element on the nondegenerate $4$-space $U^{\perp}/U$. Then  Corollary \ref{c:sp4even} implies that there exists $x \in X$ such that $G(x)$ induces $\Sp_4(k)$ on $U^{\perp}/U$ and thus $G(x)^0$ generically has a composition factor on $V$ of dimension at least $4$. Since $G(x)$ does not generically fix a nondegenerate
$2$-space (see Lemma \ref{l:nondeg2}), this implies that $G(x)$ is generically primitive and  irreducible on $V$ with rank $2$ or $3$.  But $\OO_6(k)$ (and its connected component) are the only proper
subgroups of $G$ with these properties and the inequality $\sum_ie_i<6(r-1)$ rules out the possibility that $G(x)$ is generically contained in such a subgroup (see Lemma \ref{l:oinsp}). It follows that $\Delta$ is nonempty.

 \vs
 
 \noindent \emph{Case 2. $d_1 = 2$.}
 
 \vs
 
Here $x_1$ is semisimple and in view of Case 1 we may assume it is of the form $(I_2, \l I_2, \l^{-1}I_2)$ for some scalar $1 \ne \l \in k^{\times}$. In particular, $d_1 = e_1 = 2$.  
 
First assume $r=2$, so $d_2 \leqs 4$ and $e_2 \leqs 3$. Note that if $d_2=4$ then $e_2=d_2$, which violates the bound $e_1+e_2 \leqs 5$, whence $d_2 \leqs 3$. Now $x_1$ preserves
a $3$-dimensional totally isotropic subspace of $V$, acting as a regular semisimple element on this $3$-space. Therefore, Theorem \ref{t:sl} implies that there exists $x \in X$ such that $G(x)^0$ induces $\SL_3(k)$ on such a subspace.  In addition, we can find $y \in X$ such that $G(y)^0 = \Sp_2(k)$ acts irreducibly on a nondegenerate $2$-space and does not preserve a totally isotropic $3$-space. Therefore, $G(x)^0$ is generically irreducible and has rank $2$ or $3$. Now $\SO_6(k)$ is the only proper connected subgroup of $G$ with this property, but if $G(x)^0$ is generically conjugate to $\SO_6(k)$, then $G(x)$ is generically contained in a conjugate of $\OO_6(k)$ and this is not possible by Lemma \ref{l:oinsp}.

Now assume $r \geqs 3$. If $d_i \leqs 3$ for some $i \geqs 2$ then the result follows from the argument in the previous paragraph, so we may assume $e_i=d_i \geqs 4$ for all $i \geqs 2$. If $x_2$ and $x_3$ are unipotent, then there exists a conjugacy
class $D = y^G \subseteq C_2C_3$ of elements of prime order $t$ (we can take $t=2$ if $x_2$ or $x_3$ is a transvection, otherwise $t \geqs 3$) such that the relevant inequalities still hold with respect to the variety $C_1 \times D \times C_4 \times \cdots \times C_r$. Consequently, we may assume that at most one $x_i$ is unipotent. In particular, we may assume $x_2$ is semisimple of the form $(I_4,\eta I_1, \eta^{-1}I_1)$ for some $1 \ne \eta \in k^{\times}$. Next observe that we can choose elements $y_i \in C_i$ for $i=1,2$ so that the closure of $\la y_1, y_2 \ra$ is a subgroup $H= \Sp_2(k) \times A$ preserving an orthogonal decomposition $V = U \perp U'$ into nondegenerate spaces, where $\dim U = 2$. Here $A$ is abelian and has four distinct weight spaces on $U'$, which means that $H$ preserves only finitely many subspaces of $V$. In turn, this implies that $G(x)^0$ does not generically preserve a totally isotropic $3$-space. It follows that $G(x)^0$ is generically irreducible and has rank at least $2$, whence $\Delta$ is nonempty by arguing as above.   This completes the proof in Case 2. 
 
\vs

\noindent \emph{Case 3. $d_i \geqs 3$ for all $i$.}

\vs 

If $r = 2$ then $d_1=d_2=3$ and we deduce that $x_1$ and $x_2$ are quadratic, which is the case we are excluding. For the remainder, let us assume $r \geqs 3$. There are a number of different cases to consider.

\vs

\noindent \emph{Case 3.1. $d_1=d_2 = 3$.}

\vs 

Here we may choose $y_i \in C_i$ for $i=1,2$ such that the closure of $\la y_1, y_2\ra$ is a maximal rank subgroup of the form $\Sp_2(k)^3$. Since such a subgroup
is contained in only finitely many maximal closed subgroups of $G$, it follows that we can find $y_i \in C_i$ with $i \geqs 3$ so that $G(x) = G$ for $x = (y_1, y_2, \ldots, y_r) \in X$. 
   
\vs

\noindent \emph{Case 3.2. $d_1=3$, $x_1$ semisimple and $d_i \geqs 4$ for $i \geqs 2$.}

\vs 
   
First observe that $x_1 = (\l I_3, \l^{-1}I_3)$ and $e_i=d_i$ for $i \geqs 2$. Suppose $x_2$ and $x_3$ are both semisimple, so $e_i=d_i=4$ for $i=2,3$. Then there exists a semisimple element $g \in C_2C_3$ of the form $(I_2,\mu I_2, \mu^{-1}I_2)$ so that  
\[
d_1+\a(g) + \sum_{i=4}^rd_i =  \sum_{i=1}^rd_i-6 \leqs 6(r-2)  
\]
and the desired result follows by Case 2. Similarly, if $x_2$ and $x_3$ are unipotent with Jordan form $(J_2^2,J_1^2)$ then by passing to closures we may assume they are both short root elements and therefore we can find a semisimple element $g \in C_2C_3$ of the form $(I_2,\mu I_2, \mu^{-1}I_2)$. Once again, we deduce that $\Delta$ is nonempty via Case 2.

Next suppose $x_2$ and $x_3$ have Jordan forms $(J_2^2,J_1^2)$ and $(J_2,J_1^4)$, respectively. By passing to closures, we may assume $x_2$ is a short root element. Here we can find an involution 
$g \in C_2C_3$ with Jordan form $(J_2^3)$ and so there exists $x \in \bar{X}$ such that
$G(x)^0 = \Sp_2(k)^3$. If $r \geqs 4$ then we immediately deduce that $\Delta$ is nonempty since 
$\Sp_2(k)^3$ is contained in only finitely many maximal closed subgroups of $G$. Now assume $r = 3$. Here we deduce that for generic $x \in X$, $G(x)^0$ has rank $3$ and it does not fix a nonzero totally isotropic subspace of $V$. By Lemma \ref{l:oinsp}, $G(x)^0$ is not generically $\SO_6(k)$ and so it remains to show that $G(x)^0$ is not generically of the form $\Sp_2(k)^3$.

Seeking a contradiction, suppose $G(x)^0$ is generically a subgroup of the form $\Sp_2(k)^3$. Since $G(x)$ does not generically fix a nondegenerate $2$-space (see Lemma \ref{l:nondeg2}), it follows that $G(x)$ is generically irreducible on $V$.   Now the elements in $C_1$ have odd order and they do not transitively permute the three nondegenerate spaces in an orthogonal decomposition 
\begin{equation}\label{e:decc2}
V = V_1 \perp V_2 \perp V_3
\end{equation}
preserved by $G(x)^0 = \Sp_2(k)^3$ (this would only be possible if $x_1$ is an element of order $3$ of the form $(I_2, \omega I_2, \omega^{-1}I_2)$, which is not the case since $d_1 = 3$). Similarly, no element in $C_3$ can interchange two of the summands. Therefore, since $x_2$ is an involution, we conclude that $G(x)$ does not transitively permute the $V_i$ and thus $G(x)$ is reducible, a contradiction.  

If  $x_2 = x_3 = (J_2, J_1^4)$, then $r \geqs 4$ and we can replace $C_2$ and $C_3$ by the class of short root elements (which is contained in $C_2C_3$) and argue as above.  

To complete the analysis of Case 3.2, we may assume $x_2$ is semisimple (with $d_2=4$),  $x_3$ is unipotent and $r=3$. We can choose $y_i \in C_i$ for $i =1,2$ such that the closure of $\langle y_i, y_2 \rangle $ induces $\Sp_2(k)$ on a nondegenerate $2$-space, whence $G(x)^0$ does not generically fix a totally isotropic $3$-space. By passing to closures, we may assume that $x_3$ is either a long root element or a short root element.

First assume that $x_3 = (J_2,J_1^4)$. By applying \cite[Theorem 4.5]{Ger}, we see that there exists $x \in X$ with $G(x)^0 = \Sp_2(k) \times \Sp_2(k)$, preserving an orthogonal decomposition as in \eqref{e:decc2}. Since $G(x)$ does not generically fix a nondegenerate $2$-space nor a $1$-space, it follows that either $G = G(x)$, or $G(x)$ acts imprimitively on $V$, transitively permuting the $V_i$ in \eqref{e:decc2}. But every element in $C_1$, $C_2$ and $C_3$ acts trivially on the set of summands in any orthogonal decomposition of $V$ into nondegenerate $2$-spaces, so the latter possibility is ruled out and we conclude that $\Delta$ is nonempty.

Finally, suppose that $x_3 = (J_2^2,J_1^2)$ is a short root element.  Then $x_3$ is conjugate to an element in $\GL(W)$, a Levi subgroup of the stabilizer in $G$ of a totally isotropic $3$-space $W$. By Theorem \ref{t:sl}, there exists $x \in X$ such that $G(x)^0 = \SL(W)$. Since $G(x)^0$ does not generically fix a totally isotropic $3$-space and since the smallest composition factor of $G(x)^0$ on $V$ is generically at least $3$-dimensional, it follows that $G(x)^0$ is generically irreducible and
contains elements with distinct eigenvalues on $V$.  But as noted above, $G$ does not have a proper connected subgroup with these properties and thus $G = G(x)$ for generic $x \in X$.  

\vs

\noindent \emph{Case 3.3. $d_1=3$, $x_1$ unipotent and $d_i \geqs 4$ for $i \geqs 2$.}

\vs 

Here $x_1 = (J_2^3)$ and $e_i=d_i$ for all $i$, so $\sum_i d_i < 6(r-1)$ and $r \geqs 3$. If $x_2$ and $x_3$ are long root elements, then we can replace $C_2 \times C_3$ by the class $g^G$ of short root elements, noting that the relevant inequalities are satisfied for $Y = C_1 \times g^G \times C_4 \times \cdots \times C_r$. Therefore, we may assume $d_2 = d_3 = 4$ and $r=3$.  In the usual manner, we see that there exists $x \in \bar{X}$ such that $G(x)$
induces $\SL_3(k)$ on a totally isotropic $3$-space. Also as above, there exist $y_i \in C_i$ for $i = 1, 2$
such that the closure of $\langle y_i, y_2 \rangle $ induces $\Sp_2(k)$ on a nondegenerate $2$-space. This implies that $G(x)^0$ does not generically fix a totally isotropic $3$-space and as before this allows us to conclude that $\Delta$ is nonempty.

\vs

\noindent \emph{Case 3.4. $d_i \geqs 4$ for all $i$.}

\vs 

To complete the proof of the theorem, we may assume that $d_i \geqs 4$ for all $i$. Here $e_i=d_i$ and thus $r \geqs 4$. If $x_1$ and $x_2$ are transvections then the bound $\sum_ie_i<6(r-1)$ implies that $r \geqs 5$ and we can replace $C_1 \times C_2$ by the class of short root elements (noting that the relevant inequalities are still satisfied). This reduces the problem to the case where $r=4$ and at most one $x_i$ is a transvection. If $x_1 = (J_2,J_1^4)$ and $x_2$ is unipotent, then $x_2 = (J_2^2,J_1^2)$ and there exists $g \in C_1C_2$ with $g = (J_2^3)$, so the relevant inequalities still hold for $Y = g^G \times C_3 \times \cdots \times C_r$.  Similarly, if $x_1 = x_2 = (J_2^2,J_1^2)$ then by passing to closures, we may assume they are both short root elements and we can replace $C_1 \times C_2$ by $g^G$, where $g$ is a  semisimple element of the form $(I_2, \l I_2, \l^{-1}I_2)$. In view of the previous cases we have handled, these observations reduce the problem to the case where $r=4$ and at most one $x_i$ is unipotent. 

Suppose $x_1 = (J_2,J_1^4)$. Then there exists $y \in X$ such that $G(y)^0=\Sp_2(k) \times \Sp_2(k)$ fixes an orthogonal decomposition as in \eqref{e:decc2} and by arguing as above, we deduce that either $\Delta$ is nonempty, or $G(x)$ is generically irreducible and imprimitive on $V$. In the latter situation, this means that there exists $x \in X$ such that $G(x)$ transitively permutes the summands $V_1$, $V_2$ and $V_3$ in \eqref{e:decc2}. But each element in $C_i$ acts trivially on the set of summands and we conclude that $\Delta$ is nonempty. An entirely similar argument applies if each $x_i$ is semisimple, so we may assume that $x_1 = (J_2^2,J_1^2)$. As usual, by passing to closures, we may assume that $x_1$ is a short root element and by arguing as above we can show that there exist $x,y \in X$ such that $G(x)^0$ induces $\SL_3(k)$ on a totally isotropic $3$-space and $G(y)^0$ induces $\Sp_2(k)$ on a nondegenerate $2$-space. As before, this implies that $G(x)=G$ for generic $x$ and the proof of the theorem is complete.
\end{proof}

In the next lemma we consider a special case that arises in the proof of our main theorem for symplectic groups in even characteristic.  

\begin{lem} \label{l:twoinvolutions}
Suppose $m \geqs 4$ is even, $p=2$ and $r \geqs 3$.  If $x_1$, $x_2$ are involutions with Jordan form $(J_2^m)$, then $\Delta$ is nonempty.   
\end{lem} 

\begin{proof}   
First observe that $d_1=d_2 = m$ and we are free to assume that $r=3$.  Note that we have $\sum_i e_i \leqs \sum_i d_i  <  2n$ and by passing to closures, we may assume that $x_1$ and $x_2$ are $a$-type involutions (see \cite{AS} and Remark \ref{r:as}). In particular, no element in $C_1$ or $C_2$ acts nontrivially on a nondegenerate $2$-space. 

Next observe that there exist $y_i \in C_i$ for $i =1,2$ such that the Zariski closure of $\langle y_1, y_2 \rangle$ is $H = T.2$, where $T$ is a  torus of $G$ 
of rank $m/2$ such that all of its weight spaces on $V$ are $2$-dimensional and its fixed space is trivial. Here the involutions in $H \setminus T$ act by inversion on $T$ and we note that  $H$ does not preserve any odd dimensional subspaces of $V$.  In particular, $G(x)$ does not generically preserve an odd dimensional subspace of $V$.  There are two cases to consider.

\vs

\noindent \emph{Case 1. $x_3 = (J_2,J_1^{n-2})$ is a long root element.}

\vs

First we claim that either $G(x)$ is generically irreducible on $V$, or $G(x)$ generically fixes a totally isotropic $2$-space. 

To see this, suppose $G(x)$ generically fixes a $d$-dimensional subspace $U$ with $d \geqs 1$ minimal, so $d \leqs m$ and we may assume $U$ is either nondegenerate or totally isotropic. In the nondegenerate case, $x_3$ acts trivially on $U$ or $U^{\perp}$, and the closure of $\la x_1, x_2 \ra$ preserves a $2$-dimensional subspace by Lemma \ref{l:quadratic}, whence $d \leqs 2$ and thus $d=2$. On the other hand, if $U$ is totally isotropic, then $x_3$  acts trivially on $U$ and so once again we deduce that $d=2$. The claim now follows since we have already noted that the elements in $C_1$ and $C_2$ act trivially on any nondegenerate $2$-space.   

Our next aim is to show that $G(x)$ does not generically fix a totally isotropic $2$-space, in which case the previous claim implies that $G(x)$ is generically irreducible. Let $\Omega$ be the variety of totally isotropic $2$-spaces and note that $\dim \Omega = 2n- 5$.  We claim that 
\[
\dim \Omega^{x_1} = \dim \Omega^{x_2} = n-2,\;\; \dim \Omega^{x_3} = 2n-7.
\]
In particular, Lemma \ref{l:basic2} implies that if $x = (y_1,y_2,y_3) \in X$ with $y_3 \in C_3$ generic, then $G(x)$ does not fix a totally isotropic $2$-space.  

To justify the claim, first assume $W \in \O^{x_1}$. If $x_1$ acts trivially on $W$, then 
$W$ is contained in the variety of $2$-dimensional subspaces of $V^{x_1}$, which has dimension $n-4$ since $\dim V^{x_1}=m$. Now assume $x_1$ has Jordan form $(J_2)$ on $W$ and write $W = \la w, x_1w \ra$, so $\la w \ra \ne W^{x_1}$. The $1$-dimensional subspaces of $V$ that are not contained in $V^{x_1}$ form an open subset in the variety of all $1$-dimensional subspaces of $V$. In particular, this subvariety has dimension $n-1$. Since the set of $1$-dimensional subspaces $\la w \ra$ of $W$ with $\la w \ra \ne W^{x_1}$ forms a $1$-dimensional variety, we conclude that the subvariety of $2$-spaces in $\O^{x_1}$ on which $x_1$ acts nontrivially has dimension $n-2$. Therefore, $\dim \O^{x_1} = n-2$ as claimed (and also $\dim \O^{x_2} = n-2$ since $x_1$ and $x_2$ are conjugate).

We now compute $\dim \Omega^{x_3}$. Suppose $W \in \Omega^{x_3}$ and note that $x_3$ acts trivially on $W$, which means that $W$ is contained in the $(n-1)$-space $V^{x_3}$. The variety of $2$-dimensional subspaces of $V^{x_3}$ has dimension $2n-6$ and the subvariety of totally isotropic $2$-spaces has codimension $1$. Therefore, $\dim \Omega^{x_3} = 2n-7$ as required.    

We have now shown that $G(x)$ is generically irreducible and contains long root elements. Any proper closed subgroup of $G$ with these properties is contained in $\OO_n(k)$, but the bound $\sum_ie_i<2n$ implies that $G(x)$ is not generically contained in an orthogonal subgroup (see Lemma \ref{l:oinsp}) and thus $\Delta$ is nonempty.

\vs

\noindent \emph{Case 2. $x_3$ is not a long root element.}

\vs

For the remainder, let us assume $x_3$ is not a long root element. By passing to the closure of $C_3$, we may assume that $x_3$ is either semisimple or a short root element. Let $P=QL$ be the stabilizer of a totally isotropic $m$-space $W$, where $Q$ is the unipotent radical and $L$ is a Levi subgroup. Note that we may embed each $x_i$ in $L$.

By Theorem \ref{t:sl}, there exists $y = (y_1,y_2,y_3) \in X$ such that $G(y)^0$ induces 
$\SL_m(k)$ on $W$. Moreover, since $H^1(\SL_m(k), W)=0$ (see \cite{JP}) and the sum of the dimensions of the fixed point spaces of the $x_i$ on $Q/\rad(Q) \cong W$ is
less than $m$, it follows that there exist $q_i \in Q$ such that $P' \leqs G(y')$ with $y'=(y_1^{q_1}, y_2^{q_2}, y_3^{q_3}) \in X$. As a consequence, either $G(x)^0$
is generically irreducible, or $G(x)^0$ acts uniserially on $V$ and therefore
fixes a totally isotropic $m$-space for all  $x \in X$.  

Next observe that there exists a semisimple element $g \in C_1C_2$ such that $V^g$ is trivial and every eigenspace of $g$ on $V$ is $2$-dimensional. By applying Theorem \ref{l:sp6even}, we can find $h \in C_3$
such that the closure of $\langle g, h \rangle$ induces $\Sp_6(k)$ on a nondegenerate $6$-space. Therefore, $G(x)^0$ does not generically fix a totally isotropic $m$-space and so by the observation in the previous paragraph, we deduce that $G(x)^0$ is generically irreducible, it has rank at least $m-1$
and it contains elements with distinct eigenvalues on the natural module. Therefore, either $G = G(x)$ for generic $x \in X$, or $G(x)^0$ is contained in
a maximal rank connected irreducible subgroup. But the only such subgroup is $\SO_n(k)$ and this is ruled out by the bound $\sum_i e_i<2n$. The result follows. 
\end{proof}  

We can now complete the proof of Theorem \ref{t:main2}.
  
  \begin{thm} \label{t:speven} 
  Suppose $G = \Sp_{n}(k)$, where $n = 2m$, $m \geqs 3$, $p = 2$ and the $x_i$ in \eqref{e:X} have prime order. If $\sum_i  d_i \leqs n(r-1)$ and $\sum_i e_i < n(r-1)$, then $\Delta$ is empty 
 if and only if $r=2$ and $x_1,x_2$ are quadratic.
\end{thm} 
  
  \begin{proof}  
  We proceed by induction on $m$, noting that the base case $m=3$ is covered by Theorem  \ref{l:sp6even}.   Assume $m \geqs 4$ and let $P=QL$ be the stabilizer in $G$ of a $1$-dimensional subspace $\la v \ra$, where $Q$ is the unipotent radical and $L$ is a Levi subgroup stabilizing a nondegenerate $(n-2)$-space (note that $L' = \Sp_{n-2}(k)$).
 By replacing each $x_i$ by a suitable conjugate, we may embed $x_i$ in $P$ and we write $g_i$ for the induced action of $x_i$ on the nondegenerate $(n-2)$-space $W = v^{\perp}/\la v \ra$. Let $d_i'$ be the maximal dimension of an eigenspace of $g_i$ and set $e_i' = W^{g_i}$.  
 
 First assume $m$ is odd. If $x_i$ is unipotent then we may assume $d_i' = e_i' \leqs d_i - 1$ (and indeed $d_i' = d_i-2$ unless $x_i$ has Jordan form $(J_2^m)$ on $V$). Similarly, if $x_i$ is semisimple and $e_i = d_i$, then we may assume that one of the following holds:
\begin{itemize}\addtolength{\itemsep}{0.2\baselineskip} 
\item[{\rm (a)}] $d_i'= d_i -2$.
   \item[{\rm (b)}]  $d_i'=d_i -1$ and $d_i \leqs n/2$. 
   \item[{\rm (c)}]   $d_i' \leqs d_i$ and $d_i \leqs n/3$.
   \end{itemize}
And if $x_i$ is semisimple with $e_i < d_i$, then we may assume that either $d_i'=d-1$, or $d_i' = d_i$ and $d_i \leqs n/4$. 
In particular, it follows that  
\begin{equation}\label{e:ineq}
\sum_i d_i' \leqs  (n-2)(r-1),\;\; \sum_i e_i' < (n-2)(r-1)
\end{equation}
and so by induction we can choose $y \in X$ such that $G(y)$ 
induces $\Sp_{n-2}(k)$ on $W$. In addition, Lemmas \ref{l:1-spaces} and \ref{l:nondeg2} imply that $G(x)$ does not generically fix a $1$-space nor a nondegenerate $2$-space, so for generic $x \in X$, $G(x)^0$ is irreducible and has rank $m-1$ or $m$. By inspecting Lemma \ref{l:largeranksubs}, we see that the only proper closed connected subgroup of $G$ with these properties is $\SO_n(k)$. However, the condition 
 $\sum_i e_i<n(r-1)$ implies that for generic $x \in X$, $G(x)$ is not contained in $\OO_n(k)$ (see Lemma \ref{l:oinsp}) and the result follows.    
  
Finally, let us assume $m$ is even. We can repeat the previous argument for $m$ odd unless at least one $x_i$ is an $a$-type involution (in the sense of \cite{AS}) with Jordan form $(J_2^m)$. Here $d_i'=d_i = e_i = e_i' = m$. If there are two such classes, then $r \geqs 3$  and Lemma \ref{l:twoinvolutions} gives the result. Now assume there is a unique such class, say $C_1$. Then the relevant inequalities in \eqref{e:ineq} are satisfied unless $x_2$ is a semisimple element of the form $(\l I_m, \l^{-1}I_m)$. If $x_2$ has this form, then $r \geqs 3$ and we note that there exist $y_i \in C_i$ for $i = 1,2$ such that $H = \Sp_2(k)^m$ is the Zariski closure of $\langle y_1, y_2 \rangle$ and the restriction of $V$ to $H$ is a direct sum of $m$ totally isotropic $2$-dimensional irreducible modules, each occurring with multiplicity $2$.   This implies that $H$ contains a maximal torus and preserves only finitely many subspaces of $V$.  Therefore, for generic $y_3 \in C_3$ it follows that $\langle H, y_3 \rangle$ is irreducible and contains a maximal torus and a long root subgroup. We conclude that $G=\langle H, y_3 \rangle$ for generic $y_3 \in C_3$ and the result follows. 
\end{proof}  

This completes the proof of Theorem \ref{t:main2}.

\section{Generic stabilizers}\label{s:generic}
   
With the proof of Theorem \ref{t:main2} in hand, we now turn to our main applications. In this section, we will prove Theorem \ref{t:main3} on generically free modules. 

First, let us recall the set up. Let $G$ be a simple algebraic group over an algebraically closed field $k$ of characteristic $p \geqs 0$ and let $V$ be a finite dimensional faithful rational $kG$-module. Set
\[
V^G = \{v \in V \,:\, gv = v \mbox{ for all $g \in G$}\}
\]
and recall that $V$ is \emph{generically free} if $G$ has a trivial generic stabilizer; that is, there exists a nonempty open subset $V_0$ of $V$ such that each stabilizer $G_v$ is trivial for all $v \in V_0$. Note that we may pass to a field extension $k'/k$ in order to establish the existence of a trivial generic stabilizer, so without loss of generality we may assume that $k$ is not algebraic over a finite field.
 
By combining Theorem \ref{t:main2} with the main results in \cite{BGG, Ger}, we will show that if $\dim V/V^G$ is sufficiently large, then $V$ is generically free. As noted in Section \ref{s:intro}, the analogous result for Lie algebras was proved in \cite{GG1} and we refer the reader to Remark \ref{r:generic} for several examples. Moreover, when combined with the results in \cite{GG1} we can prove that generic stabilizers are trivial as a group scheme under suitable hypotheses (see Corollary \ref{c:main3}). 

\vs

In view of \cite[Theorem 1.3]{Ger} (for $G = {\rm SL}_{n}(k)$) and \cite[Theorem 9]{BGG} (for exceptional groups), we may assume that $G$ 
is isogenous to either $ {\rm Sp}_n(k)$ with $n \geqs 4$, or ${\rm SO}_{n}(k)$ with $n \geqs 7$. Let 
$\mathcal{P}$ be the set of conjugacy classes of elements in $G$ of prime order (including all nontrivial unipotent elements if $p=0$). Given an integer $r \geqs 2$, let $\mathcal{P}_r$ be the set of classes $C$ in $\mathcal{P}$ such that $G$ is topologically generated by $r$ elements in $C$ and no fewer. By Theorem \ref{t:main2}, each $C \in \mathcal{P}$ is contained in some $\mathcal{P}_r$ with $r \leqs n+1$. Moreover, $C$ is contained in $\mathcal{P}_{n+1}$ if and only if one of the following holds:
\begin{itemize}\addtolength{\itemsep}{0.2\baselineskip}
\item[{\rm (a)}] $G = \Sp_n(k)$, $p=2$ and $C$ is the class of long root elements (or short root elements if $n=4$);
\item[{\rm (b)}] $G = \Sp_4(k)$, $p \ne 2$ and $C$ is the class of involutions of the form $(-I_2,I_2)$.
\end{itemize}
This observation is also a corollary of \cite[Theorem 8.1]{GS}.

Let $V$ be a finite dimensional faithful rational $kG$-module and note that in order to prove Theorem \ref{t:main3}, we may assume $V^G = 0$. Given $C \in \mathcal{P}_r$, set
\[
V(C) = \{ v \in V \,:\, \mbox{$gv = v$ for some $g \in C$}\}.
\]
By \cite[Lemma 5.1]{Ger} we have 
\[
\dim V(C) \leqs \left(1-\frac{1}{r}\right)\dim V + \dim C
\]
and \cite[Lemma 5.2]{Ger} implies that $V$ is generically free if $\dim V(C) < \dim V$ for all $C \in \mathcal{P}$. By combining these observations, we get the following result.

\begin{lem}\label{l:gf}
In terms of the above notation, $V$ is generically free if
\[
\dim V > \max\{r\dim C \,:\, C \in \mathcal{P}_r,\, r \geqs 2\} =: c(G)
\]
\end{lem}

We are now ready to begin the proof of Theorem \ref{t:main3} for symplectic and orthogonal groups. As before, given $x \in G$ we will write $\a(x)$ for the maximal dimension of an eigenspace of $x$ on the natural module $V$ and we set $s = n-\a(x)$.

\begin{prop}\label{p:spgen}
The conclusion to Theorem \ref{t:main3} holds if $G = {\rm Sp}_n(k)$ with $n \geqs 4$.
\end{prop}

\begin{proof}
Let $C = x^G \in \mathcal{P}_r$. In view of Lemma \ref{l:gf}, our goal is to show that 
\begin{equation}\label{e:rC}
r\dim C \leqs \frac{9}{8}n^2+\e,
\end{equation}
where $\e = 2$ if $n=4$ or $(n,p) = (6,2)$, otherwise $\e=0$.

First assume $n \geqs 6$. If $r=2$ then $\dim C \leqs \frac{1}{2}n^2$ (maximal if $x$ is regular) so we may assume $r \geqs 3$. By \cite[Proposition 2.9]{B04} we have 
\begin{equation}\label{e:bound}
\dim x^G \leqs \frac{1}{2}(2ns-s^2+1).
\end{equation}

Suppose $r=3$. Since $G$ is not topologically generated by two elements in $C$, by applying Theorem \ref{t:main2} we deduce that either $x$ is quadratic, or $\a(x) \geqs n/2$. In the quadratic case, we calculate that $\dim C \leqs \frac{1}{4}n(n+2)$, while \eqref{e:bound} (with $s=n/2-1$) yields $\dim C \leqs \frac{3}{8}n^2-\frac{1}{2}n$ if $\a(x)>n/2$. Now assume $\a(x) = n/2$ and $x$ is not quadratic, so $p=2$ by Theorem \ref{t:main2}. Then $x$ is semisimple with a $1$-eigenspace of dimension $n/2$  and it is easy to check that $\dim C \leqs \frac{3}{8}n^2$, with equality if $x$ has $n/2+1$ distinct eigenvalues on $V$. We conclude that $3\dim C \leqs \frac{9}{8}n^2$ in all cases.

Now assume $r \geqs 4$. If $n=6$ and $x=(-I_2,I_4)$ then $r=4$, $\dim C = 8$ and clearly $r\dim C < \frac{9}{8}n^2$. In the remaining cases, Theorem \ref{t:main2} implies that $(r-1)\a(x) \geqs n(r-2)$ (with equality only if $p=2$), so 
\begin{equation}\label{e:alpha}
\a(x) \geqs \left\lceil \frac{n(r-2)}{r-1}\right\rceil.
\end{equation}
By applying the bound in \eqref{e:bound}, we deduce that
\[
r\dim C \leqs \frac{r}{r-1}\left(1-\frac{1}{2(r-1)}\right)n^2+\frac{r}{2}.
\]
One can check that this upper bound is maximal when $r=4$, which gives
\[
r\dim C \leqs \frac{10}{9}n^2+2.
\]
Now $\frac{10}{9}n^2+2 \leqs \frac{9}{8}n^2$ if and only if $n \geqs 12$, so the cases with $n \in \{6,8,10\}$ need closer attention. By combining the bounds in \eqref{e:bound} and \eqref{e:alpha}, we reduce to the cases where $(n,r)=(8,5)$ or $(6,4)$, and also $(n,r) = (6,7)$ if $p=2$. In the latter case, $x=(J_2,J_1^4)$ is a long root element, $\dim C = 6$ and $7\dim C = 42 = \frac{9}{8}n^2+\frac{3}{2}$. Next assume $(n,r) = (6,4)$. Here the bound $(r-1)\a(x) \geqs n(r-2)$ implies that $\a(x) \in \{4,5\}$. In fact, since $r=4$, we see that $\a(x)=4$ is the only option (if $\a(x)=5$ then $r\a(x) > n(r-1)$), so $p=2$ and it is easy to check that $\dim C \leqs 10$, which yields $4\dim C < \frac{9}{8}n^2$ (note that either $x$ is semisimple of the form $(I_4,\l,\l^{-1})$, or $x$ is an involution of type $a_2$ or $c_2$ in the notation of \cite{AS}). Similarly, if $(n,r) = (8,5)$ then $\a(x)=6$, $p=2$ and we calculate that $\dim C \leqs 14$, which gives $5\dim C < \frac{9}{8}n^2$ as required.

To complete the proof of the proposition, we may assume $n=4$, so $r \leqs 5$ and $\frac{9}{8}n^2 = 18$. First assume $p \ne 2$. If $r=2$ then the desired bound $r \dim C \leqs 20$ holds since $\dim C \leqs 8$. Next assume $r=3$. If $x$ is quadratic, then $\dim C \leqs 6$ (with equality if $x = (\l I_2, \l^{-1}I_2)$ or $(J_2^2)$) and thus $3\dim C \leqs 18$. Otherwise $2\a(x)>4$, so $\a(x) = 3$ and this case does not arise since $r\a(x)>n(r-1)$. If $r=4$ then $3\a(x)>8$, so $\a(x)=3$, $x = (J_2,J_1^2)$ and the result follows since $\dim C = 4$. Finally, if $r=5$ then $x = (-I_2,I_2)$, $\dim C = 4$ and thus $5\dim C = 20 = \frac{9}{8}n^2+2$.

Finally, assume $n=4$ and $p=2$. The above argument applies when $r=2$, or if $r=3$ and $x$ is quadratic. If $r=3$ and $x$ is not quadratic, then $x = (I_2,\l,\l^{-1})$, $\dim C = 6$ and the desired bound holds. The previous argument handles the case $r=4$, and for $r=5$ we have $x = b_1$ or $a_2$, so $\dim C  = 4$ and thus $5\dim C = 20 = \frac{9}{8}n^2+2$.
\end{proof}

\begin{prop}\label{p:sogen}
The conclusion to Theorem \ref{t:main3} holds if $G = {\rm SO}_n(k)$ with $n \geqs 7$.
\end{prop}

\begin{proof}
Let $C = x^G \in \mathcal{P}_r$. By Lemma \ref{l:gf}, it suffices to show that \eqref{e:rC} holds with $\e=0$. First assume $n \geqs 10$ is even and note that $\dim C \leqs \frac{1}{2}n^2-n$, so we may assume $r \geqs 3$. By \cite[Proposition 2.9]{B04} we have
\begin{equation}\label{e:bound2}
\dim x^G \leqs \frac{1}{2}(2ns-s^2-2s),
\end{equation}
where $s=n-\a(x)$ as above.

If $r=3$ then either $x$ is quadratic, or $\a(x) > n/2$. For $x$ quadratic, we calculate that $\dim C \leqs \frac{1}{4}n^2$ (maximal if $p=2$, $n \equiv 0 \imod{4}$ and $x$ is an involution of type $c_{n/2}$). Similarly, if $\a(x) > n/2$ then \eqref{e:bound2} implies that $\dim C \leqs \frac{3}{8}n^2-n+\frac{1}{2}$ and so in both cases we conclude that $3\dim C < \frac{9}{8}n^2$. Now assume $r \geqs 4$. Here \eqref{e:alpha} holds and by applying the bound in \eqref{e:bound2} we deduce that
\[
r\dim C \leqs \frac{r}{r-1}\left(1-\frac{1}{r-1}\right)n^2 - \frac{rn}{r-1} < \frac{r}{r-1}\left(1-\frac{1}{r-1}\right)n^2 <n^2,
\]
which gives the desired bound.

A very similar argument applies if $n \geqs 7$ is odd (recall that $p \ne 2$ in this case). For example, suppose $r \geqs 4$. As before, \eqref{e:alpha} holds, which in turn implies that $s \leqs n/(r-1)$ and we note that \cite[Proposition 2.9]{B04} gives
\[
\dim x^G \leqs \frac{1}{2}(2ns-s^2-2s+1).
\]
In this way, we get 
\[
r\dim C \leqs \frac{r}{r-1}\left(1-\frac{1}{2(r-1)}\right)n^2 \leqs \frac{10}{9}n^2 < \frac{9}{8}n^2
\]
and the result follows. 

Finally, let us assume $G = \SO_8(k)$. Here Theorem \ref{t:so88} implies that $r \in \{2,3,4\}$ and we claim that $r\dim C \leqs 48 = \frac{3}{4}n^2$ is best possible. To see this, first note that $\dim C \leqs 24$, so the bound holds when $r=2$. Next suppose $r=3$. If $x$ is quadratic then $\dim C \leqs 16$, with equality if $x$ is an involution of the form $(-I_4,I_4)$ or $c_4$, according to the parity of $p$. Otherwise, $2\a(x)>8$ and thus $\a(x) = 6$, but this is incompatible with the condition $r=3$ since $3\a(x) >2n$. Finally, suppose $r=4$. Here $3\a(x)>2n$, so $\a(x) = 6$ and it is straightforward to check that $\dim C \leqs 12$, with equality if and only if $x$ is semisimple of the form $(I_6, \l, \l^{-1})$, or $p \ne 2$ and $x$ is unipotent with Jordan form $(J_3,J_1^5)$, or $p=2$ and $x$ is an involution of type $c_2$. In particular, $4 \dim C \leqs 48 = \frac{3}{4}n^2$ and the proof of the proposition is complete.
\end{proof}

This completes the proof of Theorem \ref{t:main3} and we conclude this section by presenting a brief proof of Corollary \ref{c:main3}. 

\begin{proof}[Proof of Corollary \ref{c:main3}]
Define $G$, $V$, $V'$ and $d'(G)$ as in the statement of the corollary and define $V^G$ and $d(G)$ as in Theorem \ref{t:main3}. It is well known that a generic stabilizer is trivial as a group scheme if and only if there are no $k$-points and the corresponding Lie algebra is trivial (this is a special case of \cite[Proposition 3.16]{Milne}). By Theorem \ref{t:main3}, a generic stabilizer is trivial as an algebraic group if $\dim V/V^G > d(G)$, while the Lie algebra is trivial if $\dim V/V' > d'(G)$ by \cite[Theorem A]{GG1}. The result follows.
\end{proof}
 
\section{Random generation of finite simple groups}\label{s:random} 
   
In this section we prove Theorem \ref{t:main4} and Corollary \ref{c:sylow} on the generation of finite simple groups of Lie type. As discussed in Section \ref{s:intro}, Theorem \ref{t:main4} extends work of Liebeck and Shalev \cite{LiSh0, LiSh} and Gerhardt \cite{Ger} on random $(r,s)$-generation of finite classical groups, as well as similar results of Guralnick et al. \cite{BGG, GLLS} for exceptional groups of Lie type.

As in the statement of Theorem \ref{t:main4}, let $r$ and $s$ be primes with $s>2$ and let $\mathcal{S}_{r,s}$ be the set of finite simple groups whose order is divisible by both $r$ and $s$. Given a group $L$ in $\mathcal{S}_{r,s}$, let $\mathbb{P}_{r,s}(L)$ be the probability that $L$ is generated by randomly chosen elements of order $r$ and $s$ (see \eqref{e:Prs}). Our goal is to prove that if $(G_i)$ is a sequence of simple groups in $\mathcal{S}_{r,s}$ with $|G_i| \to \infty$, then either
\begin{itemize}\addtolength{\itemsep}{0.2\baselineskip}
\item[{\rm (a)}] $\mathbb{P}_{r,s}(G_i) \to 1$, or 
\item[{\rm (b)}] $(r,s) = (2,3), (3,3)$ and $(G_i)$ contains an infinite subsequence of groups of the form ${\rm PSp}_{4}(q)$. 
\end{itemize}

By combining \cite[Theorem 1.4]{Ger} and \cite[Theorem 12]{BGG} with the main theorem of \cite{LiSh}, we deduce that if $(G_i)$ is any sequence of alternating, linear, unitary, or exceptional groups in $\mathcal{S}_{r,s}$ with $|G_i| \to \infty$, then $\mathbb{P}_{r,s}(G_i) \to 1$. Therefore, in view of the main theorem of \cite{LiSh}, to complete the proof of the theorem we need to extend this result to symplectic and orthogonal groups of bounded rank, noting the anomaly of the $4$-dimensional symplectic groups when $(r,s) = (2,3)$ or $(3,3)$ (see Remark \ref{r:sp4} in Section \ref{s:intro}).

Let us briefly introduce our notational set up for the proof of Theorem \ref{t:main4}. Let $G$ be a simply connected simple algebraic group defined over an algebraically closed field $k$ of positive characteristic $p$ that is not algebraic over a finite field. Given a Steinberg endomorphism $F : G\rightarrow G$, let $G^F = G(q)$ be the fixed points of $F$ on $G$ for some $p$-power $q$, where $G(q)$ is possibly twisted. Set $Z(q) = Z(G) \cap G(q)$ and note that $G(q)/Z(q)$ is almost always a finite simple group of Lie type over $\mathbb{F}_q$ (the handful of exceptions include the groups ${\rm Sp}_4(2)$ and ${}^2F_4(2)$, which are not perfect). Let $r$ be a prime and set 
\begin{align*}
m(G,r,q) & = \max\{ \dim g^G \,:\, \mbox{$g \in G(q)$ has order $r$ modulo $Z(G)$}\} \\
\mathcal{C}(G,r,q) & = \{ g^G \,:\, \mbox{$\dim g^G = m(G,r,q)$ and $g \in G(q)$ has order $r$ modulo $Z(G)$}\}
\end{align*}
For primes $r$ and $s$, let $\mathcal{Q}(r,s)$ be the set of powers $q=p^a$ such that $G(q)$ contains elements of orders $r$ and $s$ modulo $Z(G)$. 

The following result is a key ingredient in the proof of Theorem \ref{t:main4}. In the statement, the integer $N$ is defined in \eqref{e:N}. Also recall the definition of $\Delta$ in \eqref{e:delta}.

\begin{prop}\label{p:maximal} 
Let $G={\rm Sp}_n(k)$ or ${\rm Spin}_n(k)$, where $n \geqs N$. Let $r$ and $s$ be primes with $s >2$ and assume $(r,s)\ne(2,3), (3,3)$ if $G={\rm Sp}_4(k)$. Fix $q\in \mathcal{Q}(r,s)$ and set $X = C_1 \times C_2$, where $C_1 \in \mathcal{C}(G,r,q)$ and $C_2 \in \mathcal{C}(G,s,q)$. Then $\Delta$ is nonempty. 
\end{prop}

\begin{proof}
Write $C_i = x_i^G$ and let $d_i$ be the dimension of the largest eigenspace of $x_i$ on the natural $n$-dimensional $kG$-module $V$. In addition, let $e_i = \dim V^{x_i}$ be the dimension of the $1$-eigenspace of $x_i$ on $V$. Then by Theorem \ref{t:main2}, it follows that $\Delta$ is nonempty if all the following conditions are satisfied:
\begin{itemize}\addtolength{\itemsep}{0.2\baselineskip}
    \item[(a)] $d_1+d_2\leqs n$; 
    \item[(b)] $e_1+e_2 < n$ if $G = {\rm Sp}_{n}(k)$ and $p=2$;
    \item[(c)] $x_1$ and $x_2$ are not both quadratic;
    \item[(d)]  $x_1$ and $x_2$ do not appear in Table \ref{tab:main} (up to ordering).
\end{itemize}
We refer the reader to \cite[Chapter 3]{BGiu} for a convenient source of information on the conjugacy classes of elements of prime order in finite classical groups.

\vs

\noindent \emph{Case 1. $G = {\rm Sp}_{n}(k)$}

\vs

To begin with, let us assume $G = {\rm Sp}_{n}(k)$ with $n \geqs 4$ and fix a conjugacy class $C = x^G$ in $\mathcal{C}(G,r,q)$. Set $e = \dim V^x$ and let $d$ be the maximal dimension of an eigenspace of $x$ on $V$. 

If $r=2$ then it is straightforward to show that $d = n/2$. For example, if $p \ne 2$ and $g \in G(q)$ has order $2$ modulo $Z(G)$, then $g$ is either $G$-conjugate to an involution of the form $(-I_{\ell},I_{n-\ell})$ 
for some even integer $2 \leqs \ell \leqs  n-2$, or an element of order $4$ of the form $(\lambda I_{n/2}, \lambda^{-1}I_{n/2})$. Now $\dim g^G = \ell(n-\ell)$ or $\frac{1}{4}n(n+2)$ in the two cases, whence $\dim g^G$ is maximal when $g = (\lambda I_{n/2}, \lambda^{-1}I_{n/2})$ and thus $d=n/2$ as claimed. 

Now assume $r>2$. We claim that either $d \leqs n/2-1$ or $(n,r,d)=(4,3,2)$.

To see this, let us first assume $r=p$, so $x$ is unipotent and the Jordan form of $x$ on $V$ corresponds to the largest partition $\pi$ of $n$ (with respect to the usual dominance ordering on partitions) with the property that all parts of $\pi$ have size at most $p$ and the multiplicity of every odd part is even (as noted in \cite[Proposition 3.4.10]{BGiu}, every partition of this form corresponds to an element of order $p$ in $G(q)$). Write $n = ap+b$ with $0 \leqs b<p$. If $a$ is even then $\pi = (p^a,b)$, otherwise $\pi = (p^{a-1},p-1,b+1)$. In both cases, $d = \lceil n/p \rceil$ and the claim quickly follows (note that $\pi = (2^2)$ if $n=4$ and $p=3$).

Now assume $r>2$ and $r \ne p$, so $x$ is semisimple. Let $i$ be the smallest positive integer such that $r$ divides $q^i-1$ and set $t = (r-1)/i$. First we establish the bound $d \leqs n/2$. To see this, suppose $d>n/2$ and note that $d=e$ since each eigenvalue $\lambda \in k$ has the same multiplicity as $\lambda^{-1}$. Suppose $i$ is even and in the notation of \cite[Section 3.4.1]{BGiu} write $x$ as a block-diagonal matrix
\[
x = (\Lambda_1^{a_1}, \ldots, \Lambda_t^{a_t},I_e) \in G(q)
\]
where each $\Lambda_j  \in \GL_i(q)$ is irreducible and each $a_j$ is a nonnegative integer (here the $\Lambda_j$ represent the distinct conjugacy classes in $\GL_i(q)$ of elements of order $r$, while $a_j$ denotes the multiplicity of $\Lambda_j$ in the block-diagonal form of $x$). Then
\[
\dim C = \dim G - \frac{1}{2}e^2-\frac{1}{2}e - \frac{i}{2}\sum_{j=1}^{t}a_j^2. 
\]
Consider the following element
\[
y = (\Lambda_1^{a_1+1}, \Lambda_2^{a_2}, \ldots, \Lambda_t^{a_t},I_{e-i}) \in G(q)
\]
of order $r$ and set $D = y^G$. Then 
\begin{equation}\label{e:CD}
\dim D = \dim C +i(e-a_1-i/2) >\dim C
\end{equation}
since $a_1<n/2i$ and $i<n/2$. This is a contradiction and one can check that a very similar argument applies when $i$ is odd. 

To complete the proof of the claim, it remains to rule out $d=n/2$ (unless $(n,r)=(4,3)$, in which case $d=2$ for every element in $G$ of order $r$). Seeking a contradiction, let us assume $d=n/2$. If $n \equiv 2 \imod{4}$ then $x$ cannot have a $d$-dimensional $1$-eigenspace (since the $1$-eigenspace of any semisimple element has to be even-dimensional) and thus $i \in \{1,2\}$, $e=0$ and $x = (\lambda I_{n/2}, \lambda^{-1}I_{n/2})$ is quadratic. But then $G(q)$ contains elements of the form $y = (\lambda I_{n/2-1},\lambda^{-1}I_{n/2-1},I_2)$ and we have $\dim y^G = \dim x^G +n-4$, which is a contradiction. Finally, suppose $n \equiv 0 \imod{4}$. If $n=4$ and $r \geqs 5$ then $G(q)$ contains regular semisimple elements of order $r$, so $d=1 = n/2-1$. Now assume $n \geqs 8$. If $e \ne n/2$ then the previous argument applies, so let us assume $e = n/2$. If $i$ is even then we can define $D = y^G$ as in the previous paragraph and we note that \eqref{e:CD} holds since $a_1 \leqs n/2i$ and $i \leqs n/2$. Once again, we have reached a contradiction. And similarly if $i$ is odd.

In view of the above bounds, and recalling that $(r,s) \ne (2,3), (3,3)$ when $n=4$, it is now easy to see that properties (a)-(d) hold when $G = \Sp_n(k)$, whence $\Delta$ is nonempty by Theorem \ref{t:main2}.

\vs

\noindent \emph{Case 2. $G = {\rm Spin}_n(k)$}

\vs

For the remainder, let us assume $G = {\rm Spin}_n(k)$ with $n \geqs 7$. The case $n$ odd is straightforward; here $p \ne 2$ and it is easy to check that $d = (n+1)/2$ if $r=2$ and $d \leqs (n-1)/2$ if $r>2$. In particular, we observe that (a), (c) and (d) are satisfied.

Now assume $n$ is even, so $n \geqs N = 10$. If $r=2$ then one can check that 
\[
d = \left\{\begin{array}{ll}
n/2 & \mbox{if $n \equiv 0 \imod{4}$} \\
n/2+1 & \mbox{if $n \equiv 2 \imod{4}$}
\end{array}\right.
\]
For example, if $p=2$ and $g \in G$ is an involution, then $\dim g^G$ is maximal when $g$ is of type $c_{n/2}$ if $n \equiv 0 \imod{4}$ (in the notation of \cite{AS}) and of type $c_{n/2-1}$ if $n \equiv 2 \imod{4}$. In particular, in the latter case $g$ has Jordan form $(J_2^{n/2-1},J_1^2)$ and thus $d = n/2+1$ (note that in this situation, there are no involutions of type $b_{n/2}$ in $G$). 

Now assume $r>2$. We claim that $d \leqs n/2-1$ if $r=p$. As in the symplectic case, the Jordan form of $x$ corresponds to the largest partition $\pi$ of $n$ with the property that all parts have size at most $p$, but here we require that every even part has an even multiplicity. Write $n = ap+b$ with $0 \leqs b < p$. If $a$ is odd then $b$ is odd and $\pi = (p^a,b)$. On the other hand, if $a$ is even then $\pi = (p^a)$ if $b=0$, otherwise $\pi = (p^a,b-1,1)$. It is now straightforward to check that $d \leqs n/2-1$. For example, suppose $p=3$. From the above observations we deduce that $d \leqs n/3+2$, which is less than $n/2$ for $n \geqs 18$. The remaining cases with $10 \leqs n \leqs 16$ can be checked directly. For instance, if $n=10$ then $\pi = (3^3,1)$ and thus $d = 4 = n/2-1$. 

Finally, suppose that $r>2$ and $r \ne p$. As before, let $i \geqs 1$ be minimal such that $r$ divides $q^i-1$. We claim that
\[
d \leqs \left\{\begin{array}{ll}
n/2 & \mbox{if $n \equiv 0 \imod{4}$} \\
n/2-1 & \mbox{if $n \equiv 2 \imod{4}$}
\end{array}\right.
\]
First we establish the bound $d \leqs n/2$ for all $n$. Seeking a contradiction, suppose $d>n/2$. As in the symplectic case, this implies that $d=e$ and it is easy to construct an element $y \in G(q)$ of order $r$ modulo $Z(G)$ with $\dim y^G > \dim x^G$, which gives the desired contradiction. For example, suppose $i$ is odd and write 
\[
x = ((\Lambda_1,\Lambda_1^{-1})^{a_1}, \ldots, (\Lambda_{t/2},\Lambda_{t/2}^{-1})^{a_{t/2}}, I_e)
\]
as in \cite[Proposition 3.5.4]{BGiu}, where $t = (r-1)/i$. Here the $\Lambda_j^{\pm}$ represent the distinct conjugacy classes of elements of order $r$ in $\GL_i(q)$ and we note that $\Lambda_j$ and $\Lambda_j^{-1}$ must have the same multiplicity $a_j$ in the block-diagonal form of $x$, as indicated by the notation. Now define
\[
y = ((\Lambda_1,\Lambda_1^{-1})^{a_1+1}, (\Lambda_2,\Lambda_2^{-1})^{a_2}, \ldots, (\Lambda_{t/2},\Lambda_{t/2}^{-1})^{a_{t/2}}, I_{e-2i}) \in G(q)
\]
and note that
\[
\dim y^G = \dim x^G + 2i(e-a_1-i-1).
\]
Since $a_1 \leqs (n/2-1)2i$ and $i \leqs (n/2-1)/2$, it is easy to check that $a_1+i+1<e$ and thus $\dim y^G > \dim x^G$ as required. A similar argument applies when $i$ is even.

This establishes the desired bound on $d$ when $n \equiv 0 \imod{4}$, so let us assume $n \equiv 2 \imod{4}$. If $d = n/2$ then $x = (\lambda I_{n/2},\lambda^{-1}I_{n/2})$ is the only option (note that the $1$-eigenspace of any semisimple element in $G$ of order $r$ modulo $Z(G)$ is even dimensional) and it is easy to see that $\dim y^G>\dim x^G$ with $y = (\lambda I_{n/2-1},\lambda^{-1}I_{n/2-1},I_2) \in G(q)$. 

With the bounds on $d$ in hand, it is straightforward to check that properties (a) and (c) hold. In addition, by inspecting Table \ref{tab:main} we observe that (d) holds. This completes the proof of the proposition. 
\end{proof}

We also need an analogous result in the special case $G = {\rm Spin}_8(k)$.

\begin{prop}\label{p:spin88}
Let $G={\rm Spin}_8(k)$ and $r,s$ be primes with $s >2$. Fix $q\in \mathcal{Q}(r,s)$ and set $X = C_1 \times C_2$, where $C_1 \in \mathcal{C}(G,r,q)$ and $C_2 \in \mathcal{C}(G,s,q)$. Then $\Delta$ is nonempty. 
\end{prop}

\begin{proof}
Fix a class $C = x^G$ in $\mathcal{C}(G,r,q)$ and let $d_j$ be the maximal dimension of an eigenspace of $x$ on the $8$-dimensional irreducible $kG$-module $V_j = L(\omega_j)$ for $j = 1,3,4$. By inspecting the relevant conjugacy classes in $G$ and their images under a triality graph automorphism $\tau$ of $G$, it is straightforward to show that $d_j \leqs 4$ for all $j$.

To see this, first assume $r=2$. If $p \ne 2$ then $x$ has Jordan form $(-I_4,I_4)$ on $V_1$ and we note that $C$ is stable under $\tau$, so $d_j = 4$ for all $j$. Similarly, if $p=2$ then $x$ is a $c_4$-type involution and the same conclusion holds. Now assume $r>2$. If $r=p$ then 
\[
x = \left\{\begin{array}{ll}
(J_3^2,J_1^2) & \mbox{if $p=3$} \\
(J_5,J_3) & \mbox{if $p=5$} \\
(J_7,J_1) & \mbox{if $p \geqs 7$} 
\end{array}\right.
\]
and in each case $C$ is stable under $\tau$, whence $d_j = 4$ if $p=3$, otherwise $d_j = 2$. Finally, suppose $r>2$ and $r \ne p$. Let $i \geqs 1$ be minimal such that $r$ divides $q^i-1$, so $i \in \{1,2,3,4,6\}$. If $i \in \{3,6\}$ then $d_j = 2$ for all $j$ and similarly $d_j \in \{2,4\}$ if $i=4$.  Now assume $i \in \{1,2\}$. If $r=3$ then $x = (I_2, \lambda I_3, \lambda^{-1}I_3)$ or $(I_4, \lambda I_2, \lambda^{-1}I_2)$, noting that $\dim x^G = 18$ in both cases, so $d_j \in \{3,4\}$ in this case. For $r=5$ we get $x = (I_4,\lambda, \lambda^2,\lambda^{-1},\lambda^{-2})$ or $(\lambda I_2, \lambda^2I_2,\lambda^{-1}I_2,\lambda^{-2}I_2)$, where $\lambda \in k$ is a primitive $5$-th root of unity; in both cases $\dim x^G = 20$ and $d_j \in \{2,4\}$. Finally, if $r \geqs 7$ then $\dim x^G = 24$ and $d_j \in \{1,2\}$. 

This justifies the claim and the result now follows via Theorem \ref{t:so8}. 
\end{proof}

We are now in a position to complete the proof of Theorem \ref{t:main4}.

\begin{proof}[Proof of Theorem \ref{t:main4}]
Let $r$ and $s$ be primes with $s>2$. By combining \cite[Theorem 1.4]{Ger} and \cite[Theorem 12]{BGG} with the main results in \cite{LiSh} on classical and alternating groups of large rank and degree, respectively, we only need to consider symplectic and orthogonal groups of fixed rank. So let $G = {\rm Sp}_n(k)$ or ${\rm Spin}_n(k)$ be a simply connected simple algebraic group over an algebraically closed field $k$ of characteristic $p>0$ and assume $k$ is not algebraic over a finite field. Since ${\rm Sp}_2(k) = {\rm SL}_2(k)$ and ${\rm Spin}_6(k)$ is isogenous to ${\rm SL}_{4}(k)$, we may assume that either $n \geqs N$ (see \eqref{e:N}) or $G = {\rm Spin}_8(k)$. Let us also assume that $(r,s) \ne (2,3), (3,3)$ if $G = {\rm Sp}_4(k)$. 

Fix $q \in \mathcal{Q}(r,s)$ and set $X = C_1 \times C_2$, where $C_1 \in \mathcal{C}(G,r,q)$ and $C_2 \in \mathcal{C}(G,s,q)$. By applying Propositions \ref{p:maximal} and \ref{p:spin88} we deduce that $\Delta$ is nonempty and thus
\[
\lim_{q\in \mathcal{Q}(r,s), \, q\rightarrow\infty} \mathbb{P}_{r,s}(G(q))=1
\]
by \cite[Lemma 6.4]{Ger}. Finally, since $Z(q)$ is contained in the Frattini subgroup of $G(q)$, we deduce that the same conclusion holds for the simple groups $G(q)/Z(q)$.
 \end{proof}

Let us highlight the anomaly of the $4$-dimensional symplectic groups.

\begin{thm}\label{t:sp4rs}  
Suppose $G=\Sp_4(k)$, where $k$ has characteristic $p>0$.
\begin{itemize}\addtolength{\itemsep}{0.2\baselineskip}
\item[{\rm (i)}] If $(r,s) = (2,3)$, then $\PP_{r,s}(G(q))=0$ if $p \leqs 3$ and 
$\lim_{q \to \infty}\PP_{r,s}(G(q)) = 1/2$ if $p \geqs 5$.
\item[{\rm (ii)}] If $(r,s)=(3,3)$, then $\PP_{r,s}(G(q))=0$ if $p=3$ and
\[
\lim_{q \to \infty}\PP_{r,s}(G(q)) = \left\{ \begin{array}{ll}
3/4 & \mbox{if $p \geqs 5$} \\
1/2 & \mbox{if $p=2$.}
\end{array} \right.
\]
\end{itemize}
\end{thm} 

\begin{proof}  
Let $C = x^G$ and $D = y^G$ be conjugacy classes of maximal dimension, where $x$ and $y$ have order $2$ and $3$ modulo $Z(G)$, respectively. Note that $C \cap G(q)$ and $D \cap G(q)$ are nonempty.

If $p \ne 2$ then $x = (\l I_2, \l^{-1}I_2)$ with $\l^2=-1$. Similarly, if $p \ne 3$ then $y = (\l I_2, \l^{-1}I_2)$ or $(I_2, \l, \l^{-1})$ up to conjugacy, where $\l^3=1$. In particular, if $p \geqs 5$ and $X = C \times D$ then Theorem \ref{t:main2} implies that $\Delta$ is nonempty if and only if $D$ is the class of elements of the form $(I_2,\l, \l^{-1})$. Therefore, by arguing as in the proof of Theorem \ref{t:main4}, we deduce that 
\[
\lim_{q \to \infty}\PP_{2,3}(G(q)) = \frac{1}{2}
\]
if $p \geqs 5$. Similarly, by considering $X = D_1 \times D_2$ where the $D_i$ are classes of elements of order $3$ of maximal dimension, we deduce that $\Delta$ is nonempty unless $D_1 = D_2$ is the class of quadratic elements of order $3$, whence
\[
\lim_{q \to \infty}\PP_{3,3}(G(q)) = \frac{3}{4}
\]
for $p \ne 3$.

Next assume $p=3$. As above, $x$ is quadratic and we note that $y$ is also quadratic, with Jordan form $(J_2^2)$. Since $G$ is not topologically generated by two quadratic elements, we deduce that $\PP_{2,3}(G(q)) = \PP_{3,3}(G(q))=0$ for all $q=3^f$.

Finally, let us assume $p=2$ and recall that $G$ has two $4$-dimensional irreducible restricted $kG$-modules, denoted $V_j = L(\omega_j)$ for $j=1,2$. First note that $x$ acts quadratically on both modules (with Jordan form $(J_2^2)$). Similarly, $y$ acts quadratically on exactly one of the two modules and we deduce that $\PP_{2,3}(G(q))=0$. Finally, if $D_1$ and $D_2$ denote  
the two classes of elements of order $3$, then $\Delta$ is nonempty if $X = D_1 \times D_2$ or $D_2 \times D_1$, and empty if $X = D_i \times D_i$ for $i=1,2$ (see Theorem \ref{t:sp4even1}). We conclude that $\PP_{3,3}(G(q)) \to 1/2$ when $p=2$.
\end{proof} 

Finally, let us turn to Corollary \ref{c:sylow}, which gives an asymptotic version of Conjecture \ref{c:syl} on the generation of finite simple groups by two Sylow subgroups. 

\begin{proof}[Proof of Corollary \ref{c:sylow}]
Let $G$ be a finite simple group and let $r,s$ be prime divisors of $|G|$ with $r \leqs s$. Clearly, if $\mathbb{P}_{r,s}(G)>0$ then $G$ is generated by a pair of Sylow subgroups corresponding to the primes $r$ and $s$. Therefore, by combining Theorems \ref{t:main4} and \ref{t:sp4rs}, the proof of Corollary \ref{c:sylow} is reduced to the following cases: 
\begin{itemize}\addtolength{\itemsep}{0.2\baselineskip}
\item[{\rm (a)}] $(r,s)=(3,3)$ and $G = \Sp_4(q)$ with $q=3^f$; or
\item[{\rm (b)}] $(r,s)=(2,3)$ and $G = \Sp_4(q)$ with $q=2^f$ or $3^f$; or 
\item[{\rm (c)}] $(r,s)=(2,2)$.
\end{itemize}

First consider cases (a) and (b). Write $q=p^f$ where $p$ is a prime and note that $p \in \{r,s\}$. In both cases, a maximal subgroup of $G$ contains a Sylow $p$-subgroup of $G$ if and only if it is 
a parabolic subgroup. In particular, there are only two maximal subgroups of $G$ containing a fixed Sylow $p$-subgroup. The probability that a randomly chosen element of a given nontrivial conjugacy class is contained in a fixed maximal parabolic subgroup tends to $0$ as $f$ tends to infinity and the desired result follows. 

Finally, consider case (c).  By the main theorem of \cite{Gur}, it follows that every nonabelian finite simple group $G$ can be generated by a Sylow $2$-subgroup and an involution. The result follows.
\end{proof}

\begin{rem}\label{r:sylow}
Fix primes $r$ and $s$ and let $(G_i)$ be a sequence of finite simple groups, with $|G_i|$ tending to infinity, such that each $|G_i|$ is divisible by $r$ and $s$. In a sequel we will prove that with probability tending to $1$, $G_i = \la P, Q \ra$ for randomly chosen Sylow subgroups $P$ and $Q$ corresponding to the primes $r$ and $s$. Let us briefly outline the main steps: 
\begin{itemize}\addtolength{\itemsep}{0.2\baselineskip}
\item[{\rm (a)}] By applying Theorems \ref{t:main4} and \ref{t:sp4rs}, we can reduce the problem to the case where $r=s=2$. 
\item[{\rm (b)}] Suppose $G = A_n$. Fix a Sylow $2$-subgroup $P$ of $G$ and note that $P$ fixes at most one subset of $\{1, \ldots, n\}$ of a given size. The probability that a random conjugate of $P$ fixes the same subset of size $k$ is either $0$ or $\binom{n}{k}^{-1}$, and 
clearly the sum of these probabilities for $1 \leqs k < n$ goes to $0$ as $n \to \infty$. Therefore, with probability tending to $1$ with $n$, the subgroup of $G$ generated by two random Sylow $2$-subgroups acts transitively on $\{1, \ldots, n\}$. By a classical theorem of Jordan (see \cite[Example 3.3.1]{DM}, for example), if $n \geqs 9$ then $G$ has no proper primitive subgroup containing a double transposition. Therefore, the only obstruction to randomly generating $G$ by a pair of Sylow $2$-subgroups is the possibility that they generate a transitive imprimitive subgroup. But for any divisor $m$ of $n$, a Sylow $2$-subgroup of $G$ stabilizes at most one partition of $\{1, \ldots, n\}$ into parts of size $m$. Therefore, the probability that two random Sylow $2$-subgroups generate an imprimitive subgroup goes to $0$ as $n \to \infty$.
\item[{\rm (c)}] Finally, let $G$ be a group of Lie type over $\mathbb{F}_q$ of twisted Lie rank $\ell$. If $\ell$ is increasing, then the desired result follows from \cite[Theorems 3.1 and 3.2]{King}.
\item[{\rm (d)}]  Now assume $\ell$ is fixed and $q$ tends to infinity. Suppose $q$ is even and let $P$ be a Sylow $2$-subgroup of $G$. By a lemma of Tits (see \cite[1.6]{Seitz}), there are precisely $\ell$ maximal subgroups of $G$ that contain $P$; one for each conjugacy class of maximal parabolic subgroups of $G$. Therefore, the probability that $P$ and a random conjugate of any given nontrivial element generate $G$ tends to $1$ as $q \to \infty$. 
\item[{\rm (e)}] Finally, suppose $\ell$ is fixed, $q=p^f$ is odd and $f$ tends to infinity. Let $\bar{G}$ be the corresponding simply connected simple algebraic group over an algebraically closed field of characteristic $p$ that is not algebraic over a finite field. Here the key step is to extend our results on topological generation by establishing the existence of conjugacy classes $C_1$ and $C_2$ in $\bar{G}$ containing elements of order $2$ and $4$ (modulo the center of $\bar{G}$), respectively, with several desirable properties. In particular, we will show that there exists a tuple $(y_1,y_2) \in C_1 \times C_2$ such that $\la y_1, y_2 \ra$ is Zariski dense in $\bar{G}$. From here, it is relatively straightforward to complete the argument. 
\end{itemize}
\end{rem}
 
\section{Proof of Corollary \ref{c:cor2}}\label{s:corol}

In this final section, we present a proof of Corollary \ref{c:cor2}, which is another consequence of our main results.   

As in the corollary, let $G$ be a simple classical algebraic group over an algebraically closed field $k$ of characteristic $p \geqs 0$ that is not algebraic over a finite field. Let $V$ be the natural module for $G$ and set $n = \dim V$. Recall that $n \geqs M$, where $M$ is the integer defined in \eqref{e:M}. Define $X$ as in \eqref{e:conj}, where each $x_i$ has prime order modulo $Z(G)$, and let $d_i$ be the maximal dimension of an eigenspace of $x_i$ on $V$. Let us assume there exists $y \in X$ such that $G(y)$ acts irreducibly on $V$.

First observe that the existence of such an element $y$ implies that $G(x)$ does not generically fix a $1$-dimensional subspace of $V$ and thus $\sum_i d_i \leqs n(r-1)$.    

Suppose $G=\Sp_{n}(k)$, $n \geqs 4$ and $p =2$. Let $e_i = \dim V^{x_i}$ be the dimension of the $1$-eigenspace of $x_i$ on $V$. If $\sum_i e_i = n(r-1)$ then by arguing as in the proof of Lemma \ref{l:oinsp} we see that $G(x)$ generically fixes a
$1$-dimensional subspace of the indecomposable orthogonal $kG$-module $W$ of dimension $n+1$ (with socle of dimension $n$). Since $G(x)$ does not
generically fix a $1$-dimensional subspace of the socle of $W$, it follows that $G(x)$ generically fixes a complement to the socle and is therefore contained in an orthogonal subgroup ${\rm O}_n(k)$. 

Therefore, we have $\sum_i d_i \leqs n(r-1)$ and we may assume $\sum_i e_i < n(r-1)$ if $G = \Sp_n(k)$ and $p=2$. Then by Theorem \ref{t:main2}, $\Delta$ is nonempty unless we are in one of the exceptional cases recorded in parts (i) and (ii) in the statement of the theorem. If (i) holds, in which case $r=2$ and the $x_i$ are quadratic, then Lemma \ref{l:quadratic} implies that each $G(x)$ acts reducibly on $V$. Similarly, by carefully inspecting the proof of Theorem \ref{t:main2}, we find that each $G(x)$ acts reducibly on $V$ whenever we are in any of the exceptional cases in part (ii) of the theorem. The result follows.

\vs

This completes the proof of Corollary \ref{c:cor2}.


\begin{thebibliography}{9999}

\bibitem{AS} M. Aschbacher and G.M. Seitz, \emph{Involutions in Chevalley groups over fields of even order}, Nagoya Math. J. \textbf{63} (1976), 1--91.

\bibitem{Barry}
M.J.J. Barry, \emph{Decomposing tensor products and exterior and symmetric squares}, J. Group Theory \textbf{14} (2011), 59--82.

\bibitem{BrG}
T. Breuer and R.M. Guralnick, \emph{Finite groups can be generated by a $\pi$-subgroup and a $\pi'$-subgroup}, Israel J. Math., to appear.

\bibitem{BGGT} 
E. Breuillard, B. Green, R. Guralnick and T. Tao, \emph{Strongly dense free subgroups of semisimple algebraic groups}, Israel J. Math. \textbf{192} (2012),  347--379.

\bibitem{B04}
T.C. Burness, \emph{Fixed point spaces in actions of classical algebraic groups}, J. Group Theory \textbf{7} (2004), 311--346.

\bibitem{B23}  
T.C. Burness, \emph{On the topological generation of exceptional groups by unipotent elements}, Transform. Groups, to appear.

\bibitem{BGG}  T.C. Burness, S. Gerhardt and R.M. Guralnick, \emph{Topological generation 
of exceptional algebraic groups}, Adv. Math. \textbf{369} (2020), 107177, 50 pp. 

\bibitem{BGiu}
T.C. Burness and M. Giudici, \emph{Classical groups, derangements and primes}, Aust. Math. Soc. Lecture Series, vol. 25,  Cambridge University Press, Cambridge, 2016.

\bibitem{BG23}  T.C. Burness and R.M. Guralnick,  \emph{On the generation of simple groups by Sylow subgroups}, Contemp. Math., to appear. 

\bibitem{BGS} 
T.C. Burness, R.M. Guralnick and J. Saxl, \emph{On base sizes for algebraic groups}, J. Eur. Math. Soc. (JEMS) \textbf{19} (2017), 2269--2341.

\bibitem{DM}
J.D. Dixon and B. Mortimer, \emph{Permutation groups}, Graduate Texts in Mathematics, vol. 163. Springer-Verlag, New York, 1996.

\bibitem{GG1} S. Garibaldi and R. Guralnick, \emph{Generically free representations I:  large representations}, Algebra Number Theory \textbf{14} (2020),  1577--1611.

\bibitem{GG2}  S. Garibaldi and R.M. Guralnick, \emph{Generically free representations II: irreducible representations}, Transform. Groups \textbf{25} (2020), 793--817.

\bibitem{GG3} S. Garibaldi and R.M. Guralnick, \emph{Generically free representations III:  extremely bad characteristic}, Transform. Groups \textbf{25} (2020), 819--841.

\bibitem{GGL}  S. Garibaldi and R.M. Guralnick,  \emph{Generic stabilizers for simple algebraic groups}, Michigan Math. J. \textbf{72} (2022), 343--387.

\bibitem{Ger} S. Gerhardt, \emph{Topological generation of special linear groups}, J. Algebra \textbf{569} (2021), 511--543.

\bibitem{GNATO} R.M. Guralnick, \emph{Some applications of subgroup structure to probabilistic generation and covers of curves}, in Algebraic Groups and their Representations (Cambridge, 1997), 301--320, NATO Adv. Sci. Inst. Ser. C Math. Phys. Sci., 517, Kluwer Acad. Publ., Dordrecht, 1998.

\bibitem{Gur}
R.M. Guralnick, \emph{Generation of simple groups}, J. Algebra \textbf{103} (1986), 381--401.

\bibitem{GL}  R.M. Guralnick and R. Lawther, \emph{Generic stabilizers in actions of simple algebraic groups}, Mem. Amer. Math. Soc., to appear.

\bibitem{GLLS}
R.M. Guralnick, M.W. Liebeck, F. L\"{u}beck and A. Shalev, \emph{Zero-one generation laws for finite simple groups}, Proc. Amer. Math. Soc. \textbf{147} (2019), 2331--2347.

\bibitem{GMT} R.M. Guralnick, G. Malle and P.H. Tiep, \emph{Products of conjugacy classes in finite and algebraic simple groups}, Adv. Math. \textbf{234} (2013), 618--652.

\bibitem{GS}
R.M. Guralnick and J. Saxl, \emph{Generation of finite almost simple groups by conjugates}, J. Algebra \textbf{268} (2003), 519--571.

\bibitem{GT} R.M. Guralnick and P.H. Tiep,  \emph{Decompositions of small tensor powers and Larsen's conjecture}, Represent. Theory \textbf{9} (2005), 138--208.

\bibitem{He1} W. Hesselink, \emph{Singularities in the nilpotent scheme of a classical group}, Trans. Amer. Math. Soc. \textbf{222} (1976), 1--32.

\bibitem{He}  W.H. Hesselink, \emph{Nilpotency in classical groups over a field of characteristic $2$}, Math. Z.  \textbf{166} (1979), 165--181.

\bibitem{Hu}  J.E. Humphreys,  \emph{Algebraic groups and modular Lie algebras}, Mem.   Amer. Math. Soc. 71 (1967). 

\bibitem{JP}  W. Jones and B. Parshall, \emph{On the $1$-cohomology of finite groups of Lie type}, in Proceedings of the Conference on Finite Groups (Univ. Utah, Park City, Utah, 1975), 
313--328. Academic Press, New York, 1976.

\bibitem{King}
C.S.H. King, \emph{On the finite simple images of free products of finite groups}, Proc. Lond. Math. Soc. \textbf{118} (2019), 143--190.

\bibitem{LLS} 
R. Lawther, M.W. Liebeck and G.M. Seitz, \emph{Fixed point spaces in actions of exceptional algebraic groups}, Pacific J. Math. \textbf{205} (2002), 339--391.

\bibitem{LS04}
M.W. Liebeck and G.M. Seitz, \emph{The maximal subgroups of positive dimension in exceptional algebraic groups}, Mem. Amer. Math. Soc. 802 (2004).

\bibitem{LS}  M.W. Liebeck and G.M. Seitz, \emph{Unipotent and nilpotent classes in simple algebraic groups and Lie algebras}, Mathematical Surveys and Monographs, vol. 180. Amer.  Math. Soc., 2012. xii+380 pp. 

\bibitem{LiSh0} M.W. Liebeck and A. Shalev, \emph{Classical groups, probabilistic methods, and the $(2,3)$-generation problem}, Ann. of Math. \textbf{144} (1996), 77--125.

\bibitem{LiSh} M.W. Liebeck and A. Shalev,  \emph{Random $(r,s)$-generation of finite classical groups}, Bull. London Math. Soc. \textbf{34} (2002), 185--188.

\bibitem{LT}
M.W. Liebeck and D.M. Testerman, \emph{Irreducible subgroups of algebraic groups}, Q. J. Math. \textbf{55} (2004), 47--55.

\bibitem{Lu} F. L\"ubeck,   \emph{Small degree representations of finite Chevalley groups in defining characteristic}, LMS J. Comput. Math. \textbf{4} (2001), 135--169.

\bibitem{Martin}   B.M.S. Martin, \emph{Generic stabilisers for actions of reductive groups}, 
Pacific J. Math. \textbf{279} (2015),  397--422.

\bibitem{Miller}
G.A. Miller, \emph{On the groups generated by two operators}, Bull. Amer. Math. Soc. \textbf{7} (1901),  424--426.

\bibitem{Milne}  J.S. Milne, \emph{Algebraic groups: The theory of group schemes of finite type over a field}, Cambridge Studies in Advanced Mathematics, vol. 170. Cambridge University Press, Cambridge, 2017.

\bibitem{Ri} R.W. Richardson,  \emph{Principal orbit types for algebraic transformation spaces in characteristic zero}, Invent. Math. \textbf{16} (1972), 6--14.

\bibitem{Ro}  L. Rowen,  \emph{Polynomial identities in ring theory}, Pure and Applied Mathematics, vol. 84, Academic Press, New York--London, 1980.

\bibitem{scott} L.L. Scott,  \emph{Matrices and cohomology}, Ann. of Math. \textbf{105} (1977),  473--492.

\bibitem{Seitz} G.M.  Seitz, \emph{Flag-transitive subgroups of Chevalley groups},  Ann. of Math. \textbf{97} (1973), 27--56. 

\bibitem{Seitz2} G.M. Seitz,  \emph{Unipotent elements, tilting modules, and saturation},  Invent. Math. \textbf{141}  (2000),  467--502. 

\bibitem{Sp}  N. Spaltenstein, \emph{Classes unipotentes et sous-groupes de Borel}, 
Lecture Notes in Mathematics, vol. 946, Springer-Verlag, Berlin--New York, 1982.  

\bibitem{St}  R. Steinberg, \emph{Endomorphisms of linear algebraic groups}, 
Mem. Amer. Math. Soc. 80 (1968).

\bibitem{St2} R. Steinberg,  \emph{Torsion in reductive groups},  Adv. Math.  \textbf{15} (1975), 63--92. 

\bibitem{Ste}  R. Steinberg, \emph{Conjugacy classes in algebraic groups}, Lecture Notes in Mathematics, vol. 366. Springer-Verlag, Berlin--New York, 1974.

\bibitem{Tits} J. Tits, \emph{Free subgroups in linear groups}, J. Algebra \textbf{20} (1972), 250--270.
\end{thebibliography}
\end{document}